\documentclass[a4paper,  twoside, 10pt, leqno]{amsart}
\pdfoutput=1 

\usepackage[subnum]{cases}
\usepackage{ulem}

\usepackage[a4paper, margin=1.5in]{geometry}
\usepackage{xcolor}
\usepackage{bbm}
{%
\setcounter{enumi}{0}

\begin{enumerate}}%
{\end{enumerate} }

\usepackage{amsmath, amsmath, amsthm,amssymb}
\usepackage{graphicx}
\usepackage{subcaption} % several images in one figure

\usepackage{enumerate} % for list options

\allowdisplaybreaks
\usepackage{color}
\usepackage[colorlinks,breaklinks=true]{hyperref}
\hypersetup{
    linkcolor=blue,
}

\renewcommand{\em}{\it}
\newcommand{\N}{\mathbb{N}}
 
\newcommand{\TR}{[0,T]\times\mathbb{R}^d}
\newcommand{\Z}{\mathbb{Z}}
\newcommand{\zd}{\mathbb{Z}^d}
\newcommand{\R}{\mathbb{R}}
\newcommand{\rd}{\mathbb{R}^d}

\newcommand{\Lip}{\mbox{Lip}}

\newcommand{\vertiii}[1]{{\left\vert\kern-0.25ex\left\vert\kern-0.25ex\left\vert #1 \right\vert\kern-0.25ex\right\vert\kern-0.25ex\right\vert}}

\newtheorem{thm}{Theorem}[section]
\newtheorem{corollary}[thm]{Corollary}
\newtheorem{lemma}[thm]{Lemma}
\newtheorem{proposition}[thm]{Proposition}

\newtheorem{definition}[thm]{Definition}

\theoremstyle{remark}
\newtheorem{remark}[thm]{Remark}

\theoremstyle{definition}
\newtheorem{example}{Example}

\newenvironment{description*}%
  {\begin{description}
    \setlength{\itemsep}{0.33em}
  }
  {\end{description}}
   
\makeatletter%
\let\orgdescriptionlabel\descriptionlabel
\renewcommand*{\descriptionlabel}[1]{%
  \let\orglabel\label
  \let\label\@gobble
  \phantomsection
  \edef\@currentlabel{#1}%
  \let\label\orglabel
  \orgdescriptionlabel{#1}%
}
\makeatother

\title[Approximation of nonlocal MFGs]{
 On Numerical approximations  of fractional and nonlocal Mean Field Games}

\author[I. Chowdhury]{Indranil Chowdhury}
\address{\parbox{.95\linewidth}{{ (I. Chowdhury)} 
Department of Mathematics, Faculty of Science, University of Zagreb, \ Croatia \medskip }}
\email{indranil.chowdhury@ntnu.no, indranill2011@gmail.com}
\author[O. Ersland]{Olav Ersland}
\address{(O. Ersland) Norwegian University of Science and Technology, Norway}
\curraddr{}
\email{olav.ersland@ntnu.no}

\author[E.~R. Jakobsen]{Espen R. Jakobsen}
\address{(E.~R. Jakobsen) Norwegian University of Science and Technology, Norway}
\curraddr{}
\email{espen.jakobsen@ntnu.no}


\begin{document}  
%  \BC RED: This part is not finished or need to update. \ee

%  \bi Indranil's recent changes and Olav has to verify and agree. \ee

% \bo Olav's recent changes. \ee 

\begin{abstract}
We construct numerical approximations for Mean Field Games with
fractional or nonlocal diffusions. The schemes are based on
semi-Lagrangian approximations of the underlying control
problems/games along with dual approximations of the
distributions of agents. The methods are monotone, stable, and consistent,  and we prove 
convergence along subsequences for  (i) degenerate
equations in one space dimension
and (ii) nondegenerate equations in arbitrary
dimensions. We also give results on full convergence and
convergence to classical solutions.
Numerical tests are implemented
for a range of different nonlocal diffusions and support our analytical findings.
% We study a fully discrete numerical approximation for Mean Field Games
% involving nonlocal/jump diffusion. 
% The proposed scheme is based on a semi-Lagrangian approximation 
% and hence the diffusion operators are allowed to be degenerate. 
% We establish the existence of solutions of the discrete systems and, study the convergence of the resulting schemes 
% both for degenerate (in one space dimension) 
% and non-degenerate cases (in arbitrary space dimension). 
% We further provide numerical examples and simulations, 
% validating the convergence results and comparing the 
% degenerate and non-degenerate cases. 
\end{abstract}
\maketitle  
\tableofcontents

\section{Introduction}

 In this article we study numerical approximations of Mean Field Games (MFGs) with fractional and general non-local diffusions. We consider the mean field game system
\begin{align}
 \begin{cases}
        -u_t - \mathcal{L} u + H(x,Du) = F (x, m(t)), \quad &\text{ in } (0,T)\times\R^d, \\
        m_t - \mathcal{L}^*m - \text{div} (m D_p H(x,Du)) = 0 \quad &\text{ in } (0,T)\times\R^d, \\
        u (T,x) = G(x,m(T)), \ m(0) = m_0 \quad &\text{ in } \R^d,
    \end{cases}
    \label{eqn:MFG}
\end{align}
where
\begin{align}
    \mathcal{L} \phi (x)= \int_{|z|>0} \big[ \phi (x+z) - \phi (x) - \mathbbm{1}_{\{|z|<1\}} 
    %\nabla 
    D \phi (x) \cdot z   \big] d\nu(z),
    \label{levy_operator}
\end{align}
is a nonlocal diffusion operator (possibly degenerate), $\nu$ is a L\'evy measure (see assumption \ref{A0}), and the adjoint $\mathcal{L}^{*}$ is defined as 
$ ( \mathcal{L}^{*} \phi, \psi )_{L^{2}} = 
( \phi, \mathcal{L} \psi )_{L^2 }$ for 
$\phi,\psi \in C_{c}^{\infty} ( \rd )$.

The first equation in \eqref{eqn:MFG} is a backward in time Hamilton-Jacobi-Bellman (HJB) equation with terminal data $G$, 
and the second equation is a forward in time Fokker-Planck-Kolmogorov (FPK) equation with initial data $m_0$. 
Here $H$ is the %so-called 
Hamiltonian, and 
the system is coupled through the cost functions $F$ and  $G$.  
There are two different types of couplings: \textit{(i) Local couplings} where $F$ and $G$ depend on point values of $m$, 
and \textit{(ii) non-local} or {\it smoothing couplings} where they depend on distributional properties induced from $m$ through integration or convolution. 
Here we work with nonlocal couplings. 

A mathematical theory of MFGs were
introduced by Lasry--Lions
\cite{MR2295621} and Caines--Huang--Malhame
\cite{MR2346927}, and describes the limiting behavior of $N$-player stochastic differential games when the number of players $N$ tends to $\infty$ \cite{cardaliaguet2015master}. 
In recent years there has been significant progress on MFG systems with local (or no) diffusion,  including e.g. modeling, wellposedness, numerical approximations, long time behavior, convergence of Nash equilibria, and various control and game theoretic questions, see e.g.~\cite{ACDPS:Note,CD:Books,cardaliaguet2015master,BFY:Book,GPV:Book,GLL03} and references therein. 
The study of  MFGs  with `non-local diffusion' is quite recent, and few results exist so far. Stationary problems with fractional Laplacians were studied in \cite{cesaroni2017stationary}, and parabolic problems including \eqref{eqn:MFG}, in \cite{cirant2019existence} and  \cite{ersland2020classical}. 
We refer to \cite{MR3580196} and references therein for some development using probabilistic methods.

 The difference between problem \eqref{eqn:MFG} and standard MFG formulations lies in the type of noise driving the underlying controlled stochastic differential equations (SDEs). Usually Gaussian noise is considered \cite{MR2295621,MR3305653,MR3399179, MR3072222,ACDPS:Note},
or there is no noise (the first order case) \cite{MR3408214,MR3358627}.
Here the underlying SDEs are driven by pure jump
L\'evy processes,  which leads to nonlocal operators
\eqref{levy_operator} in the MFG system. In many real world
applications, jump processes model the observed noise better 
than Gaussian processes \cite{applebaum2009levy,metzler2000random,tankov2003financial,woyczynski2001levy}. 
 Prototypical examples are symmetric $\sigma$-stable processes and their generators, the fractional Laplace operators $(-\triangle)^{\frac{\sigma}{2}}$.  
 In Economy and Finance 
the observed noise is not symmetric and
$\sigma$-stable, but rather non-symmetric and tempered. 
 A typical example is the one-dimensional CGMY process
\cite{tankov2003financial} where $\frac{d\nu}{dz}(z)=\frac
C{|z|^{1+Y}}e^{-Gz^+-Mz^-}$ for $C,G,M>0$ and $Y\in(0,2)$. Such models
are covered by the results of this article.  Our assumptions on the nonlocal operators
(cf. \ref{A2}) are quite general, allowing for  degenerate operators
and no restrictions on the tail of the L\'evy measure $\nu$.

There has been some development  on  numerical
approximations for MFG systems with local operators.  Finite difference
schemes for nondegenerate second order equations have been designed and analyzed e.g. by Achdou \textit{et al.}
\cite{achdou2012mean,achdou2013mean,achdou2014numerical,achdou2010mean,achdou2012iterative,achdou2016convergence,achdou2020mean} 
 and Gueant 
\cite{gueant2012mean, gueant2013mean, gueant2012new}. Semi-Lagrangian (SL)
schemes for MFG system have been developed by Carlini--Silva both for
first order equations \cite{carliniSilva2014semi1st} and 
possibly degenerate second order equations \cite{carlini2015semi}. Other numerical schemes for MFGs
include recent machine learning methods
\cite{carmona2019convergence,carmona2019linear,ruthotto2020machine} for high dimensional problems.   We refer to the
survey article \cite{achdou2020mean} for recent developments on
numerical methods for MFG. We know of  no prior schemes or numerical analysis for  MFGs with fractional or nonlocal  diffusions. 
%nonlocal MFGs. 

In this paper we will focus on SL schemes. They  are  monotone,
stable, connected to the underlying control
problem, easily handles degenerate and arbitrarily directed diffusions, and large 
time steps are allowed. Although the
SL schemes for HJB equations have been studied for some time (see e.g. \cite{MR3341715,CD83,MR1326802,MR3042570}),
there are few results for FPK equations (but see
\cite{MR3828859}) and the coupled MFG system. For nonlocal problems we only know of the results in \cite{camilli2009finite} for HJB equations.

\subsection*{Our contributions}\
\smallskip

{\bf A. Derivation.} We construct fully discrete monotone  numerical schemes for the MFG system \eqref{eqn:MFG}. These dual SL schemes are closely related to the underlying control formulation of the MFG. In our case it is based on the following controlled SDE:
$$dX_t = -\alpha_t \, dt +  dL_t,$$
 where $\alpha_t$ is the control and $L_t$ a pure jump L\'evy process (cf. \eqref{SDE}). Note that $L_t$ can be decomposed into small and large jumps, where the small jumps may have infinite intensity. We derive our approximation in several steps:
\begin{enumerate}[1.]
\item (Approximate small jumps) The small jumps are approximated by Brownian motion (see \eqref{SDE_2}) following 
e.g. \cite{MR1834755,camilli2009finite,karlsen-Elghanjaoui-2002}. This is done to avoid infinitely many jumps per time-interval and
singular integrals, and gives a better approximation compared to simply neglecting these terms.
\smallskip
\item (SL scheme for HJB) We discretise the resulting SDE from step 1 in time  and
approximate the noise by random walks and approximate compound Poisson processes in the spirit of \cite{camilli2009finite} (Section
\ref{subsec:SDE}). From the corresponding discrete
time optimal control problem, dynamic programming, and interpolation we construct an SL scheme for the HJB equation 
% \bii \sout{in the usual way  \cite{camilli2009finite}} \be 
(Section \ref{sec:SL}). 
\smallskip
\item (Approximate control) We define an approximate optimal feedback control for the SL scheme in step 2 from the continuous optimal feedback control as in \cite{carliniSilva2014semi1st,carlini2015semi}:  $\alpha^*_{\textup{approx}} =
D_p H(\cdot, Du_d^\epsilon)$, where $u_d^{\epsilon}$ is a regularization
of the (interpolated) solution from step 2  (Section
\ref{subsec:ctrl}). 
\smallskip
 \item (Dual SL scheme for FPK) 
The control of step 3 and the scheme in step 2 define a controlled approximate SDE with a corresponding discrete FPK equation for the densities of the solutions. 
% The approximate control defines a controlled time discrete SDE with a corresponding discrete FPK equation for the densities of the solutions. 
We explicitly derive this FPK equation in weak
form, and obtain the final dual SL scheme taking test functions to be
linear interpolation basis functions (Section \ref{subsec:discrtFPK}). 
\end{enumerate}
See \eqref{schme_HJ} and \eqref{Fokker-Planck_discretized} in Section
\ref{sec:discretization} for the specific form of our
discretizations. These seem to be the first
numerical approximations of MFG systems with nonlocal or fractional
diffusion and the first SL approximations of nonlocal FPK equations. 
Our dual SL schemes are extensions to the nonlocal case of the schemes in
\cite{carliniSilva2014semi1st,carlini2015semi,MR3828859}, but a clear
derivation of such type of schemes seems to be new. The schemes come
in the form of nonlinear coupled systems \eqref{disc_mfg_system} that need to be resolved
numerically. We prove existence of solutions using fixed point
arguments, see Proposition \ref{thm:existence_discrete_system}.
\smallskip

{\bf B. Analysis.} We establish a range of properties for the scheme  including  monotonicity, consistency, stability, (discrete) regularity, convergence of individual equations, and convergence to the full MFG system. 
\begin{enumerate}[ 1.]
\item (HJB approximation) For the approximation of the HJB equation we
  prove pointwise consistency and uniform discrete $L^\infty$,
  Lipschitz, and semiconcavity bounds. Convergence to a viscosity
  solution is obtained via the % Barles-Perthame-Souganidis
  half relaxed limit method \cite{barles1991convergence}.
\smallskip
\item (FKP approximation) We prove consistency in the sense of
  distributions, preservation of  mass and positivity,  $L^1$-stability, 
  tightness, and equi-continuity in time. In dimension $d=1$, we also
  prove uniform $L^p$-estimates for all $p\in(1,\infty]$. Convergence is obtained from compactness and stability arguments.
\smallskip
\item (The full MFG approximation) We prove convergence along
  subsequences to viscosity-very weak solutions of the MFG system in
  two cases: (i) Degenerate equations in dimension $d=1$, and (ii)
  non-degenerate equation in $\rd$ under the assumption that solutions
  of the HJB equation are $C^1$ in space. Full convergence follows for
  MFGs with unique solutions, and convergence to classical solutions
  follows under certain regularity and weak uniqueness
  conditions. Applying the results to the setting of
  \cite{ersland2020classical}, we obtain full convergence to classical
  solutions in this case.
\end{enumerate}

Because of the nonlocal or smoothing couplings, the HJB approximation can be analysed almost independently of the FKP approximation. The  analysis of the FKP scheme on the other hand, strongly depends on boundedness and regularity properties of solutions of the  HJB scheme. Compactness in measure is enough in the nondegenerate case when the HJB equation has $C^1$ solutions, while stronger weak ($*$) compactness in $L^p$ for some $p\in(1,\infty]$ is needed in the degenerate case. As in \cite{carliniSilva2014semi1st}, we are only able to prove this latter compactness in dimension $d=1$. A priori estimates and convergence for $p\in(1,\infty)$ seems to be new also for local MFGs.

% There are several forms of duality at play here. The FKP scheme is almost in duality with the HJB scheme by construction. The  approximate solutions $Du_{\textup{num}}$ and $m_{\textup{num}}$ must essentially lie in dual spaces (e.g. $C_b$ versus finite measures).

In this paper we study general L\'evy jump processes and nonlocal
operators. This means that the underlying stochastic processes may not
have first moments whatever initial distribution we take (like
e.g. $\sigma$-stable processes with $\sigma<1$), and then we can no
longer work in the commonly used Wasserstein-1 space $(P_1,d_1)$ for
the FKP equations. Instead we work in the space $(P,d_0)$ of
probability measures under weak convergence metrizised by the
Rubinstein-Kantorovich metric $d_0$ (see Section
\ref{sec:prelim}). Surprisingly, a result from
\cite{Espen-Indra-Milosz-2020} (Proposition
\ref{prop:tail-control-function}) allow us to prove tightness and compactness in this space without any moment assumptions! We refer to section \ref{subsec:ext} for a more detailed discussion along with convergence results in the traditional $(P_1,d_1)$ topology when first moments are available.

This $(P,d_0)$ setting can be adapted to local problems, to give results
also there without moment assumptions. Finally, we note that our results for degenerate
problems cover the first order equations and improve
\cite{carliniSilva2014semi1st} in the sense that more general initial
distributions $m_0$ are allowed: $P\cap L^p$ for some $p\in(1,\infty]$
  instead of $P_{1+\delta}\cap L^\infty$ for some $\delta>0$. 
  
%% \be 
%% Some discussion on the fact that using our arguments you could get slightly more general results also for local problems: E.g. no moment assumption is needed on $m_0$ and $m$ need not be bounded in the 1d case, it is enough to have $m$ (or equivalently $m_0$) in $L^p$ for any $p>1$. We give a more extensive derivation and motivation of the schemes compared to Carlini. We have more general class of Hamiltonians than Carlini-Silva who only considered $H=H(p)=|p|^2$, we have $x$-dep and more general $p$-dep. They have diffusion coefficient depending on $t$, while we have not.
%%  
\smallskip

{\bf C. Testing.} We provide several numerical simulations. In Example \ref{ex1} and \ref{ex2} we use a similar setup as in
%cost functions
%and initial distributions as in 
\cite{carlini2015semi},
comparing the effects of a range of different diffusion operators: Fractional Laplacians of different powers, 
CGMY-diffusions, 
a degenerate diffusion, a spectrally one-sided diffusion, as well as classical local diffusion and the case of no diffusion.
%We then compare how the diffusion affect the solution of the MFG system.
In Example \ref{ex3} we solve the MFG system on a long time horizon and observe the turnpike property in a nonlocal setting. Finally, in Example \ref{ex4} we study the convergence of the scheme. 

\subsection*{Outline of the paper} 
In section \ref{sec:prelim} we list our assumptions and state mostly known
results of the MFG system \eqref{eqn:MFG} and its individual HJB and FKP
equations. 
In section \ref{sec:discretization} we construct  the discrete schemes
for the HJB,  FKP, and full MFG equations from the underlying stochastic
control problem/game. The convergence results are given in Section
\ref{subsec:main-result}, along with extensions and a discussion
section.
%we give the main results of this paper, 
%Theorem \ref{thm:convergence_MFG} and Theorem \ref{thm:convergence_MFG-nondeg}.
In sections \ref{sec:HJB} and \ref{sec:FPK} we analyze the
discretisations of the HJB and FKP equations respectively, including
establishing a priori estimates, stability, and some consistency results.
Using these results, we prove the convergence results of section
\ref{subsec:main-result} in section \ref{sec:proof-main}.
In section \ref{sec:numerics} we provide and discuss numerical
simulations of various nonlocal MFG systems. Finally, there are three
appendices with proofs of technical results.

 \section{Assumptions and Preliminaries} \label{sec:prelim}

We start with some notation. By $C, K$ we mean various constants which may change from line to line. 
The Euclidean norm on any $\mathbb{R}^d$-type space is
denoted by $|\cdot|$. For any subset $Q\subseteq
\mathbb{R}^d$ or $Q \subseteq [0,T] \times \mathbb{R}^d$, and for any bounded, possibly vector valued
function on $Q$, we will consider $L^p$-spaces $L^{p}(Q)$ and spaces $C_b(Q)$ of bounded continuous %real valued 
functions. Often we use the notation $\|\cdot\|_0$ as an alternative notation for the norms in $C_b$ or $L^\infty$. The space $C^m_b(Q)$ is the  subset of $C_b(Q)$ with $m$ bounded and continuous derivatives, and for $Q \subseteq [0,T] \times \rd$, 
%and $l,k \in \N\cup \{0\}$, 
$C^{l,k}_b(Q)$ is the subset of $C_b(Q)$ with $l$ bounded and continuous derivatives in time and $k$ in space. 
By $P(\rd)$ we denote the set of probability measure on $\rd$. The Kantorovich-Rubinstein  distance $d_0(\mu_1,\mu_2)$ on the space $P(\rd)$ is defined as 
$$d_0(\mu_1,\mu_2) := \sup_{f\in \Lip_{1,1}(\rd)}\Big\{\int_{\rd}f(x) d(\mu_1-\mu_2)(x)\Big\},$$
where $\Lip_{1,1}(\rd) = \Big\{f : f \, \mbox{is Lipschitz continuous and} \, \|f\|_{0}, \|Df\|_{0}\leq 1 \Big\}$. 
%\begin{definition}\label{moment_fn}
%There exists a function $0\leq \Psi \in C^{2}(\rd)$ with $\|D\Psi\|_{0}, \|D^2\Psi\|_{0} < \infty $ such that $$\lim_{|x|\rightarrow \infty} \Psi(x) = \infty \quad \mbox{and} \quad {\color{magenta}\sup_{x\in \rd}\bigg|\int_{\rd} \big(\Psi(x+z) - \psi(z) \big)\nu(dz)\bigg| < \infty}.$$
%\end{definition}
\medskip
We define the Legendre transform $L$ of $H$ as:
\begin{align*}
    L (x,q) := \sup_{p \in \R^d} \big\{ p\cdot q - H(x,p) \big\}.
\end{align*}
We use the following assumptions for equation \eqref{eqn:MFG}: 

\begin{description*}
\item[($\nu$0)\label{A0}] (L\'evy condition) $\nu$ is a positive Radon measure that satisfies
\begin{align*}
   \int_{\R^d} 1 \wedge |z|^2 d\nu(z) < \infty.
\end{align*}
\item[($\nu$1)\label{A2}] (Growth near singularity) There exists constants $\sigma \in (0,2)$ and $C >0$ such that the density of $\nu$ for $|z|<1$ satisfies
        \begin{align*}
            0 \leq \frac{d\nu}{dz} \leq  \frac{C}{|z|^{d+\sigma}}, \text{ for } |z| < 1.            
        \end{align*}
        
         \item[(L0)\label{L0}](Continuity and local boundedness) The
           function $L: \R^d \times \R^d \to \R$ is continuous in $x,
           q$,  and for any $K>0$, there exists $C_{L}(K)>0$ such that
        \begin{align*}
            \sup_{|q|\leq K}|L(x,q)|\leq C_L(K), \qquad x \in \R^d.
        \end{align*}
        
         \item[(L1)\label{L1}](Convexity and growth) The function $L(x,q)$ is convex in $q$ and satisfies
        \begin{align*}
            \lim_{|q| \to +\infty} \frac{L(x,q)}{|q|} = +\infty, \qquad x \in \R^d.
        \end{align*}
         
         \item[(L2)\label{L2}](Lipschitz regularity) There exists a constant $L_L >0$ independent of $q$, such that
        \begin{align*}
            |L (x,q) - L (y,q) | \leq L_L |x-y|.
        \end{align*}
    \item[(L3)\label{L3}](Semi-concavity) There exists a constant $c_L >0$ independent of $q$, such that
        \begin{align*}
            L(x+y,q) - 2 L(x,q) + L(x-y,q) \leq c_L |y|^2.
        \end{align*}
     \item[(F1)\label{F1}](Uniform bounds) There exists constants $C_{F}, C_{G} >0$ such that
        \begin{align*}
            |F ( x,\mu ) | \leq C_{F}, |G ( x,\mu ) | \leq C_{G}, \qquad \forall x \in \rd, \mu \in P (\R^d ).
        \end{align*}
    \item[(F2)\label{F2}] (Lipschitz assumption) There exists constants $L_F, L_G >0$ such that
        \begin{align*}
            |F (x,\mu_1 ) - F(y,\mu_2) | \leq L_F \big[ |x-y| + d_0(\mu_1,\mu_2) \big], \\
            \medskip
            |G (x,\mu_1 ) - G(y,\mu_2) | \leq L_G \big[ |x-y| + d_0(\mu_1,\mu_2) \big]. 
        \end{align*}
    \item[(F3)\label{F3}] (Semi-concavity) There exists constants $c_{F}, c_{G} >0$ such that
        \begin{align*}
            F (x+y,\mu ) - 2 F(x,\mu) + F ( x-y,\mu ) \leq c_{F}  \\
            \medskip
            G (x+y,\mu ) - 2 G(x,\mu) + G ( x-y,\mu ) \leq c_{G} 
        \end{align*} 
       \item[(M)\label{M1}](Initial condition) 
   We assume $m_0\in P(\rd)$.    
   \medskip
    \item[(M')\label{M11}] The dimension $d=1$, and $m_0\in P(\R) \cap L^p(\R)$ for some $p\in (1, \infty]$. 
\medskip
\end{description*}
%\begin{Assumptions}
%    \item \label{A0} (L\'evy condition) $\nu$ is a positive Radon measure that satisfies
%\begin{align*}
%   \int_{\R^d} 1 \wedge |z|^2 d\nu(z) < \infty.
%\end{align*}
%\item \label{A2} (Growth near singularity) There exists constants $\sigma \in (0,2)${\color{red}Check $\sigma=0$!!} and $C >0$ such that the density of $\nu$ for $|z|<1$ satisfies
%        \begin{align*}
%            0 \leq \frac{d\nu}{dz} \leq  \frac{C}{|z|^{d+\sigma}}, \text{ for } |z| < 1.            
%        \end{align*}
%
%%\item \label{A3} (Growth of long jumps)   There exists a function $0\leq \Psi \in C^{2}(\rd)$ with $\|D\Psi\|_{0},$ $\|D^2\Psi\|_{0} < \infty $ such that  $\displaystyle \lim_{|x|\rightarrow \infty} \Psi(x) = \infty$ and $${\color{magenta}\sup_{x\in \rd}\bigg|\int_{\rd} \big(\Psi(x+z) - \psi(z) \big)\nu(dz)\bigg| < \infty}.$$       
%\end{Assumptions}
%\begin{Assumptions1}        
%   
%\end{Assumptions1}
%\begin{Assumptions2}        
%  
%\end{Assumptions2}
%\begin{Assumptions3}        
%  
%    %and $\int_{\rd} \Psi(x)\, dm_0(x)<\infty$, where $\Psi$ is given by \ref{A3}.}
%\end{Assumptions3}
%\smallskip \smallskip

By \ref{L1}, the Legendre transform $H= L^*$ is welldefined and the optimal $q$ is $q^* = D_p H(x,p)$. To study the convergence of the numerical schemes we further assume local uniform bounds on the derivatives of Hamiltonian: \smallskip

\begin{description*}
\item[(H1)\label{H1}] The function $D_p H \in C (\rd \times \rd)$, and for every
$R >0$, there is a constant $C_R>0$ such that for every $x\in \rd$ and $p\in B_R$ we have $|D_p H (x,p) | \leq C_R$.  \medskip
 \item[(H2)\label{H2}]  The function $D_p H\in C^1(\rd\times \rd)$.  For every $R>0$ there exists a constant $C_R>0$ such that for every $x\in \rd$ and $p\in B_R$ we have 
\begin{align*}
|D_{pp}H(x,p)| + |D_{px} H(x,p)| \leq C_R.
\end{align*}
\end{description*}

\begin{remark}
We impose most of the conditions on $L$, and not on $H$, as $L$ appears in optimal control problem, which would be the basis of our semi-Lagrangian approximation. Assumptions \ref{L1} and \ref{L2} (but, not \ref{L3}!) would immediately carry forward to the corresponding Hamiltonian $H$ from the definition of Legendre transform. Whereas, we require to assume \ref{H1}--\ref{H2} on $H$, in contrary to the other assumptions, as it does not follow from the condition on $L$ in general. However, when the Lagrangian $L$ behaves like $|\cdot|^{r}$ in $q$ variable for large $q$ and $r>1$, the growth of the corresponding Hamiltonian $H$ would be $|\cdot|^{\frac{r}{r-1}}$ in $p$ variable for large $p$ (cf. \cite[Proposition 2.1]{MR3285897}). The growth of the derivatives of $H$ for large $p$ can be computed similarly, which would correspond to  similar condition as in \ref{H1}--\ref{H2}. 
\end{remark}

In  most of this paper solutions of the HJB equation in \eqref{eqn:MFG} are interpreted in the \textit{viscosity sense}, we refer to \cite{MR2129093} and references therein for general definition and wellposedness results, while solutions  of FPK equation in \eqref{eqn:MFG} are considered in the \textit{very weak} sense defined as follows:
\begin{definition}\label{def:weak-sol-FPK}
(a) If $u\in C^{0,1}_b((0,T)\times\rd)$,
a measure $m\in C([0,T],P(\rd))$ is a very weak solution of the FPK equation in \eqref{eqn:MFG}, if for every $\phi\in C_c^{\infty}(\rd)$ and $t\in[0,T]$ 
\begin{align}\label{dist_sol_FPK}
\begin{split}
&\int_{\rd} \phi(x) \,dm(t)(x) -\int_{\rd} \phi(x) \,dm_0(x)   \\ &= \int_{0}^t \int_{\rd} \Big(\mathcal{L}[\phi](x) -D_pH(x,Du)\cdot D\phi(x)\Big) dm(s)(x)ds. 
\end{split}
\end{align} 
(b) If $u\in L^\infty(0,T; W^{1,\infty}(\rd))$ and $p\in [1,\infty]$, a function $m\in C([0,T],P(\rd))\cap L^p([0,T]\times \rd)$,  is a very weak solution of the FPK equation in \eqref{eqn:MFG}, if \eqref{dist_sol_FPK} holds for every $\phi\in C_c^{\infty}(\rd)$ and $t\in[0,T]$.
\end{definition}

\begin{remark}\label{rem:defn-weaksol}
Inequality \eqref{dist_sol_FPK} holding for every $\phi\in C_c^{\infty}(\rd)$ and $t\in[0,T]$ is equivalent to 
\begin{align*}%\label{dist_sol_FPK1}
\begin{split}
&\int_{\rd} \phi(T,x) \,d(m(T))(x) -\int_{\rd} \phi(0,x) \,dm_0(x)   \\ &= \int_{0}^T \int_{\rd} \Big(\phi_t(s,x) + \mathcal{L}[\phi](s,x) -D_pH(x,Du)\cdot D\phi(s,x)\Big) dm(s)(x)ds, 
\end{split}
\end{align*}
holding for every $\phi\in C^{1,2}_b([0,T]\times \rd)$  (cf. \cite[Lemma 6.1]{Espen-Indra-Milosz-2020}).
\end{remark}

\begin{definition}
A pair $(u,m)$ is a viscosity-very weak solution of the MFG system \eqref{eqn:MFG}, 
if $u$ is a viscosity solution of the HJB equation, and 
$m$ is a very weak solution of the FPK equation (see, Definition \ref{def:weak-sol-FPK}).
\end{definition}

\begin{proposition} \label{prop:viscosity_sol_HJB}
Fix, $\mu \in C([0,T],P(\rd))$. Let \ref{A0}, \ref{L2} and \ref{F1} hold.  

\smallskip\noindent
(a) (Comparison principle) If $u$ is a viscosity subsolution and $v$ is a viscosity supersolution of the HJB equation in \eqref{eqn:MFG} with $u(T,\cdot)\leq v(T,\cdot)$,  then  $u\leq v$. 

\medskip\noindent (b) There exists a unique bounded viscosity solution $u \in C_b(\TR)$ of the HJB equation in \eqref{eqn:MFG}, and 
for any $t\in [0,T]$ we have $\|u(t)\|_{0} \leq C_FT + C_G$. 

\medskip\noindent(c) If \ref{L2} and \ref{F2} hold, then the viscosity solution $u$ is Lipschitz continuous in space variable and for every $t\in [0,T]$ and $x,y \in \rd$ we have \begin{align*}
|u(t,x) - u(t,x+y)| \leq \big(T(L_L+L_F) + L_G\big) \, |y|. 
\end{align*}  
 In addition, if \ref{L3} and \ref{F3} hold, then $u$ is semiconcave in space variable and for every $t\in [0,T]$ and $x,y \in \rd$ we have \begin{align*}
u(t,x+y) + u(t,x-y) - 2u(t,x) \leq \big(T(c_L+c_F) + c_G\big) \, |y|^2. 
\end{align*}  
\end{proposition}
\begin{proof}
These results are by now standard: \textit{(a)} follows by a similar argument as for \cite[Theorem 3.1]{MR2129093}, \textit{(b)} follows by e.g.~Perron's method, and \textit{(c)} by adapting the comparison arguments of \cite{MR2129093} in a standard way. We omit the details.
Under some extra assumptions, \textit{(b)} and \textit{(c)} also follows from Theorem \ref{HJ_convergence} and Lemma \ref{lem:aprrox_HJB_reg} below.
\end{proof}
\begin{proposition}\label{prop:weak_sol_FPK-general-d}
Assume \ref{A0}, \ref{A2}, \ref{H1}, and \ref{M1}. 

\smallskip\noindent(a) If $u \in C([0,T] ; C^1_b(\rd))$, then there exists a very weak solution $m \in C([0,T];P(\rd))$ of the
     FPK equation in \eqref{eqn:MFG}. 
     
     \medskip\noindent(b) If $d=1$,  
     $u \in C([0,T] ; W^{1,\infty}(\R))$, $u$ semi-concave,
     and  \ref{M11} holds,
     %$m_0\in L^p(\R)$ for some $p\in (1,\infty]$. 
     then there exists a very weak solution $m \in C([0,T];P(\R)) \cap L^{p}([0,T]\times \R)$ 
     of the FPK equation in \eqref{eqn:MFG}.
     Moreover,  $\|m(t)\|_{L^p(\R)} \leq e^{CT}\|m_0\|_{L^p(\R)}$ for some constant $C>0$ and $t\in[0,T]$.
\end{proposition}
\begin{proof}
The results  follow from the convergence of the discrete scheme in this article. The proof of (a) follows the proof of Theorem \ref{thm:convergence_MFG-nondeg}, setting
$Du_{\rho,h} = Du$. 
The proof of (b) follows the proof of Theorem \ref{thm:convergence_MFG} and Theorem \ref{thm:Lp_esti}, setting $Du_{\rho,h} = Du$. Note that semi-concavity of $u$ is crucial for the the $L^p$-bound of Theorem \ref{thm:Lp_esti}.
\end{proof}

Existence and uniqueness results are given in \cite{ersland2020classical}  for classical solutions of MFGs with nonlocal diffusions under additional assumptions:
\begin{description*}

    \item[($\nu$2)\label{nu3}](Growth near singularity) There exists constants $\sigma \in (1,2)$ and $c>0$ such that the density of $\nu$ for $|z|<1$ satisfies
        \begin{align*}
            \frac{c}{|z|^{d+\sigma}} \leq  \frac{d\nu}{dz}, \text{ for } |z| < 1.
        \end{align*} \smallskip
    \item[(F4)\label{F4}] There exists constants $C_{F}, C_{G} > 0$, 
        such that $\| F ( \cdot,m ) \|_{C_{b}^{2}} \leq C_{F} $ and 
        $\| G ( \cdot, \tilde{m} )  \|_{C_{b}^{3}} \leq C_{G}$ 
        for all $m, \tilde{m} \in P ( \rd )$. \smallskip
           \item[(F5)\label{F5}] $F$ and $G$ satisfy monotonicity conditions: 
\begin{align*}
    \int_{\R^d} \left( F \left( x, m_1 \right) - F \left( x, m_2 \right) \right) d \left( m_1 -m_2 \right) \left( x \right) &\geq 0 \qquad \forall m_1,m_2 \in P ( \R^d ), \\[0.2cm]
  \int_{\R^d} \left( G \left( x, m_1 \right) - G \left( x, m_2 \right) \right) d \left( m_1 -m_2 \right) \left( x \right) &\geq 0 \qquad \forall m_1,m_2 \in P ( \R^d ).
\end{align*}
   \smallskip
    \item[(H3)\label{H3}] The Hamiltonian $H \in C^3 ( \rd \times \rd ) $, and for every $R>0$ there is  $C_{R} >0$ such that for 
        $x \in \rd$, $p \in B_{R}$, $\alpha \in \mathbb{N}_{0}^{N}$, 
        $ | \alpha | \leq 3$, then 
        $|D^{\alpha} H ( x,p ) | \leq C_{R}$. \smallskip 
    \item[(H4)\label{H4}] For every $R > 0$ there is $C_R >0$ such that for $x,y \in \R^d, u \in \left[ -R,R \right], p \in \R^d$: $|H \left( x,u,p \right) - H \left( y,u,p \right)| \leq C_R \left( |p|+1 \right) |x-y|$. \smallskip
    \item[(H5)\label{H5}] (Uniform convexity) There exists a constant $C >0$ such that $\frac{1}{C} I_d \leq D_{pp}^2 H \left( x,p \right) \leq C I_d$.  \smallskip
    \item[(M'')\label{M111}] The probability measure $m_{0}$ has a 
    density (also denoted by $m_0$) $m_{0} \in C_{b}^{2}$. \smallskip
\end{description*}

\begin{thm} \label{mfg:classical_solution}
     Assume \ref{A0}, \ref{A2}, \ref{nu3}, \ref{F2}, \ref{F4}, \ref{H3},\ref{H4}, and \ref{M111}.\smallskip
     
\noindent  (a) There exists a classical solution $ ( u,m )$ of \eqref{eqn:MFG} such that $u\in C^{1,3}_b((0,T)\times \R^d)$ and $m\in C^{1,2}_b((0,T)\times \R^d)\cap C(0,T; P(\R^d))$.
     
\medskip\noindent (b) If in addition \ref{F5} and \ref{H5} hold, then the classical solution is unique.
\end{thm}
This is a consequence of \cite[Theorem 2.5 and Theorem 2.6]{ersland2020classical}. We refer to \cite{ersland2020classical} for more general results, where in particular assumptions \ref{A2} and \ref{nu3} can be relaxed to allow for a much larger class of nonlocal operators $\mathcal L$.  In the nondegenerate case, for the individual equations in \eqref{eqn:MFG} we also have uniqueness of viscosity-very weak solutions and existence of classical solutions. Uniqueness for HJB equations and existence for HJB and FPK equations follows by Theorem 5.3, Theorem 5.5, and Proposition 6.8 in \cite{ersland2020classical}. We prove uniqueness for very weak solutions of FPK equations here.

\begin{proposition}[Uniqueness for the FPK equation] \label{uniqueness_weak_soln_fp}
Assume \ref{A0}, \ref{A2}, \ref{nu3}, and $D_p H (x,Du (t)) \in C_b^{0,2} ((0,T)\times\rd)$. Then there is
at most one very weak solution of the FPK equation in \eqref{eqn:MFG}. 
\end{proposition}

\begin{proof}
Let $m_1,m_2$ be two very weak solutions, define $\tilde m := m_1 - m_2$  
and take any $\psi \in C_c^{\infty} \left(  \rd \right)$.
For any $\tau \in (0,T)$, the terminal value problem  \begin{align*}
    \partial_t \phi + \mathcal{L} \phi - D \phi \cdot D_p H (x, Du)  = 0 \quad \text{in} \quad \R^d\times(0,\tau) \quad \text{and} \quad 
\phi (x,\tau) = \psi(x)\quad \text{in} \quad \R^d,
\end{align*}
has a unique classical solution $\phi \in C_b^{1,2} ( (0,\tau)\times \rd) $ essentially by \cite[Theorem 5.5]{ersland2020classical} (the result follows from Proposition 5.8 with $k=2$ and the observation that the proof of Theorem 5.5 also holds for $k=2$). Using the definition of very weak solution (see Remark \ref{rem:defn-weaksol}) we get 
\begin{align*}
     \int_{\rd} \psi(x) \, d\tilde m(\tau)(x) = 
    \int_0^\tau \int_{\rd} \big(\partial_t \phi + \mathcal{L} \phi -  D \phi \cdot D_p H (x,Du) \big) \, d\tilde m(t)(x) \, dt  
   %\\  = \lim_{\epsilon \to 0} \int_0^T \int_{\rd} \left( \partial_t \phi + \mathcal{L} \phi -  D \phi \cdot D_p H (x,Du)  \right) \, d\tilde m(t)(x) \, dt 
   = 0,
\end{align*}
for any $\tau \in [0,T]$. Since $\psi$ was arbitrary, it follows that $\tilde m (\tau) = 0$ in $P (\rd) $ for every $\tau \in [0,T]$, and uniqueness follows.
\end{proof}

\section{Discretisation of the MFG system} \label{sec:discretization}

%\subsection{Control representation for the HJB equation for fixed $m$}

To discretise the MFG system \eqref{eqn:MFG}, we first follow \cite{camilli2009finite} and derive a Semi-Lagrange
approximation of the HJB equation in \eqref{eqn:MFG}. Using this
approximation and the 
optimal control of the original problem, we derive an approximation
of the FPK equation in \eqref{eqn:MFG} which is in (approximate) duality with the approximation of the HJB-equation. 

This derivation is based on the following control interpretation of the HJB equation.
For a fixed given density $m= \mu$, the solution $u$ of the
HJB equation in \eqref{eqn:MFG} is the value function of the
optimal stochastic control problem: 
\begin{align}\label{eq:value-fun}
    u (t,x) = \inf_{\alpha} J \big( x,t, \alpha \big),
\end{align}
where $\alpha_t$ is an admissible control, $J$ is the total cost to be minimized,
\begin{align}
    J \big( x,t, \alpha \big) = \mathbb{E} \bigg[ \int_t^T  \Big( L (\tilde X_s, \alpha_s ) + F (\tilde X_s, \mu_s \Big) ds + G (\tilde X_T, \mu_T)\bigg],
    \label{cost_fun}
\end{align}
and $\tilde X_s=\tilde X_s^{x,t}$ solves the controlled stochastic differential
equation (SDE)
\begin{align}
    \begin{cases}
    &d\tilde X_s =  -\alpha_s\, ds + \int_{|z| <1} z \tilde{N} (dz,ds) +
      \int_{|z| \geq 1} z N (dz,ds), \quad s>t,\\
    &\tilde X_t = x,
    \end{cases}
    \label{SDE}
\end{align}
where
%$\alpha_t$ is a control, $Z_0$ a random
%variable,
$N$ a Poisson random measure with
intensity/L\'evy measure $\nu (dz) ds$, and $\tilde{N} = N (dz,ds) - \nu
(dz) ds$ is the compensated Poisson measure.\footnote{The $N$-integral is just
a (difficult way of writing a)
compound Poisson jump-process, while the $\tilde N$-integral
is a centered jump process with an infinite number of (small) jumps per time
interval a.s. \cite{applebaum2009levy}.}

\subsection{Approximation of the underlying controlled SDE} \label{subsec:SDE}
\subsubsection*{A. Approximate small jumps by Brownian motion.}
First we approximate small jumps in \eqref{SDE} by (vanishing) Brownian
 motion\footnote{To avoid singular integrals and infinite number of
   jumps per time interval.} (cf. \cite{MR1834755}): For $r\in(0,1)$, let $X_s=X_s^{x,t}$ solve
\begin{align}
    \begin{cases}
    dX_s = \bar{b} (\alpha_s ) ds + \sigma_r \, dW_s + \int_{|z| \geq r} z N (dz,ds), \quad s>t\\
    X_t = x, 
    \end{cases}
    \label{SDE_2}
\end{align}
where $W_s$ is a standard Brownian motion, $\bar{b} ( \alpha_s) = -\,\alpha_s  - b_{r}^{\sigma}$, and 
\begin{align} \label{B_r}
&b_r^{\sigma} := \int_{r <|z| < 1} z\, \nu (dz), \\
&\sigma_r := \bigg( \frac{1}{2}  \int_{|z|<r} zz^T \nu (dz) \bigg)^{1/2} \label{sigma_r}.
\end{align}
 The last integral in \eqref{SDE_2} is a compound Poisson process (cf. e.g. \cite{applebaum2009levy}): For any $t\geq0$,
\begin{align}\label{copo}
  \int_0^{t}\int_{|z| \geq r} z N (dz,dt) = \sum_{j=1}^{\hat N_t}
J_j%% \ \ \text{where}\ \
%% \hat N_t \sim \textup{Poisson}(h\lambda_r), \, \{Z_j\}_{j\in\N}\, \text{iid,} \sim \nu_r, \, Z_0=0, 
\end{align}
where the number of jumps up to time $t$ is $\hat N_t \sim
\textup{Poisson}(t\lambda_r)$, the jumps
$\{J_j\}_{j}$ are iid rv's in $\R^d$ with distribution $\nu_r$
and $J_0=0$, and for $r\in (0,1]$,
\begin{align}    \label{lambda_r}
    \nu_r := \nu \mathbbm{1}_{|z|>r} \qquad\text{and}\qquad \lambda_r := \int_{\rd} \nu_r(dz).
\end{align}
%\subsubsection*{B. Approximate nonlocal operators/infinitesimal
%  generators.} {\bf This section is not needed here! Move it to a more
%  suitable place}
The infinitesimal generators
$\mathcal{L}^{\alpha}$ and $\hat{\mathcal{L}}^{\alpha}$ of the SDEs
\eqref{SDE} and \eqref{SDE_2} are
(cf. \cite{applebaum2009levy})
\begin{align*}
    &\mathcal{L}^{\alpha} \phi (x)= -\alpha_{t} \cdot \nabla \phi + \mathcal{L}_1 \phi ( x ) + \mathcal{L}^1 \phi ( x ),
    %\label{generator}
    \\
&\hat{\mathcal{L}}^{\alpha} \phi (x)=  \, \bar{b}(\alpha_{t}) \cdot \nabla \phi(x) + tr \big(  \sigma_{r}^T \cdot D^{2} \phi ( x )\cdot \sigma_{r}  \big) + \mathcal{L}^{r}[\phi](x)
\end{align*}
for $\phi\in C^2_b(\rd)$, where
    \begin{align} \label{inner_outer_operator}
    \begin{split}
       &\mathcal{L} \phi ( x ) = \mathcal{L}_r \phi ( x ) +
      \mathcal{L}^r \phi ( x )\\
      &:= \bigg(\int_{|z| < r}+   \int_{|z| > r}\bigg) \Big(\phi ( x+z
      ) - \phi ( x ) - \mathbbm{1}_{\{|z|<1\}} D \phi ( x )\cdot z
      \Big) d \nu ( z ).
%      \\
%      &\quad   \Big( \phi ( x+z ) - \phi ( x ) - \mathbbm{1}_{\{|z|<1\}} D \phi ( x )\cdot z  \Big) d \nu ( z ) 
\end{split}    
    \end{align}
The operator  $\hat{\mathcal{L}}^{\alpha}$ is an approximation of $\mathcal{L}^{\alpha}$.
       
\begin{lemma}[\cite{espen_chioma_keneth}]\label{lem:small_jump}
    If \ref{A2} holds and $\phi \in C_{b}^{3}(\rd)$, then for $\mathcal{L}_r$ and $\sigma_r$ defined in \eqref{inner_outer_operator} and \eqref{sigma_r} respectively, we have 
    \begin{align*}
        | \mathcal{L}_r \phi ( x ) - tr \big(  \sigma_{r}^T \cdot D^{2} \phi ( x )\cdot \sigma_{r}  \big) | \leq C r^{3-\sigma} \| D^{3} \phi \|_{0}.
    \end{align*}
 If in addition, $\phi \in C_{b}^{4}(\rd)$ and the L\'evy measure $\nu$ is symmetric, then 
 \begin{align*}
        | \mathcal{L}_r \phi ( x ) - tr \big(  \sigma_{r}^T \cdot D^{2} \phi ( x )\cdot \sigma_{r}  \big) | \leq C r^{4-\sigma} \| D^{4} \phi \|_{0}.
    \end{align*}   
%    \label{Levy_approximation}
\end{lemma}

\subsubsection*{B. Time discretization of the approximate SDE}
Fix a time step $h=\frac TN\in(0,1)$ 
for some $N\in\N$ 
and discrete times $t_k=kh$ for $k\in\{0,1,\dots,N\}$.
Following \cite{camilli2009finite}, we propose the following Euler-Maruyama discretization of the SDE \eqref{SDE_2}: Let $X_n^{t_l,x}\approx X^{t_l,x}_{t_n}$, where  $X_n=X^{t_l,x}_n$solves
\begin{align}
    \begin{cases}
        X_l = x \\
        X_n = X_{n-1} + h \bar{b} (\alpha_{n-1}) + \sqrt{h} \displaystyle \sum_{m=1}^d \sigma_r^m \xi_{n-1}^m, \ \ n=l+N_i+1, \dots,l+N_{i+1}-1,\hspace{-0.6cm}\\
        X_{l+N_{i+1}} = X_{l+N_{i+1}-1} + J_i.
    \end{cases}
    \label{discretization_process}
\end{align}
Here the control $\alpha_n$ is constant on each time interval, $\sigma_r^m$ is the $m$th-column of $\sigma_r$, and 
 $\xi_n = ( \xi_{n}^{1}, \ldots , \xi_{n}^{d} )$ is a random walk in
$\mathbb{R}^d$ with 
\begin{align*}
    \mathbf{P} \big( \xi^i_n = \pm 1 \big) = \frac{1}{2d}.
\end{align*}
The processes $J_k$ % $Z_k$
and $N_k$ defines an approximation of the compound
Poisson part of \eqref{SDE_2} through equation \eqref{copo} where
$\hat N_t$ is replaced by an approximation 
$$\tilde N_t  =
  \max\{k:\Delta T_1+\Delta T_2+\dots +\Delta T_k\leq t\},$$
where exponentially distributed waiting times
(time between jumps) are replaced by approximations
$\{\Delta T_k\}_{k\in\N}$\footnote{In the new model, $\tilde N_t$ still gives the number of jumps up to
  time $t$.}: $\Delta T_k=h \Delta N_k= h(N_{k}-N_{k-1})$ where $N_k:\Omega\to
\N\cup\{0\}$, $N_0=0$, and $\Delta N_k$ iid with approximate $h
\lambda_{r}$-exponential distribution given by
\begin{align*}
  \mathbf{P} [\Delta N_k > j%| N_0, N_1, \dots, N_n
  ] = e^{-h \lambda _rj} \quad \mbox{for} \quad j=0,1,2,\dots. 
\end{align*} 
Then for $p_{j} := P [ \Delta N_k = j ]$, $p_0=0$ and $p_j  =  P [
  \Delta N_k > j-1 ]-P [ \Delta N_k > j ]  = e^{- jh \lambda_{r}} ( e^{h \lambda_{r}}
-1)$ for $j>0$. We find that $\sum_{j=0}^{\infty} p_{j} =1$ and $E(\Delta
N_k)=\sum_{j=0}^\infty e^{- jh \lambda_{r}} =\frac{e^{h
    \lambda_{r}}}{e^{h \lambda_{r}}-1}$. 
%% To limit the number of jumps per time-interval we assume $\mathbb{E} ( \Delta N_k
%% )\leq 1$, or equivalently,
%% \begin{align}\label{r-cfl}
%%   h \lambda_{r}\geq \ln 2 \qquad\text{for all}\qquad h>0.
%% \end{align}
Note that in each
time interval, approximation \eqref{discretization_process} either
diffuses (the second equation) or jumps (the third equation), and that
we have ignored the  unlikely event of more than one jump per time interval.
For the scheme to converge,
we will see that we need to send both $h\to0$ and $h\lambda_r\to0$. In
this case  $E(\Delta N_k)\to\infty$  and the  jumps become less and less
frequent and the random walk dominates the evolution of $X_k$ (which
is to be expected).

\subsection{Semi-Lagrangian approximation of the HJB equation}\label{sec:SL}

\subsubsection*{A. Control approximation of the HJB equation}
We approximate the control problem \eqref{eq:value-fun} -- \eqref{SDE}
by a discrete time control problem: Define the value function
\begin{align}\label{eq:approx-value-fun}
    \tilde u_h (t_l , x ) = \inf_{ \{\alpha_n \} } J_h \big( x,t_l , \{\alpha_n \} \big),
\end{align}
where the  controls $\{\alpha_n\}$ are piecewise constant in time, the cost
function $J_h$ is given by
\begin{align}
    J_h \big( x,t_l , \{\alpha_n \} \big) = \mathbb{E} \bigg[ \sum_{n=l}^{N-1} \Big( L ( X_{n}, \alpha_n ) + F ( X_n, \mu(t_n) ) \Big) h %\Delta t 
    + G (X_N,\mu(t_N)) \bigg],
    \label{discretized_value_fun}
\end{align}
and the controlled discrete time process $X_n=X_n^{t_l,x}$  is the solution of
\eqref{discretization_process}. 
By the (discrete time) Dynamic Programming principle it follows that
 \begin{align*}
    \tilde u_h (t_l ,x ) = \inf_{\alpha_n}\mathbb{E} \bigg[ \sum_{n=l}^{l+p} \Big( L ( X_{n}^{t_l ,x}, \alpha_n ) + F ( X_n^{t_l ,x}, \mu(t_n) ) \Big) h %\Delta t
    + \tilde u_h (t_{l+p+1}, X_{l+p+1}^{t_l,x} ) \bigg],
 \end{align*} 
for $l+p+1\leq N$. Taking $p=0$ and computing the expectation
using conditional probabilities (the probability to jump in a time
interval is $p_1=1-e^{-h\lambda_r}$), we
find a (discrete time) HJB equation 
%which is an 
%approximation of the HJB equation in \eqref{eqn:MFG}:
\begin{align}
    \tilde u_h (t_l ,x&) =  \inf_{\alpha} \bigg\{ h F ( x, \mu(t_l ) ) + h L (x,\alpha) + \Big[\frac{e^{-h \lambda_r}}{2d} \sum_{m=1}^d \big(\tilde u_h (t_l +h, x + h \bar{b} (\alpha) + \sqrt{hd} \sigma_r^m )  \nonumber\\ 
        & + \tilde u_h(t_l +h, x+ h\bar{b} (\alpha) - \sqrt{hd} \sigma_r^m )
      \big) + \frac{1-e^{-h\lambda_r}}{\lambda_r} \int_{|z| \geq r}
      \tilde u_h (t_l +h,x+z) \nu (dz) \Big] \bigg\}.\label{approx-HJB}
\end{align}

\subsubsection*{B. Interpolation and the fully discrete scheme}
For $\rho>0$ we fix a grid
    $\mathcal{G}_{\rho} = \{i \rho : i\in \Z^d \}$
and a linear/multilinear $\mathcal{G}_{\rho}$-interpolation $I$. For functions $f: \mathcal{G}_{\rho} \rightarrow \mathbb{R}$,
    \begin{align}\label{linear_interpolation}
        I[f] (x) := \sum_{i\in \Z^d} f(x_i) \beta_i (x), \qquad x \in \R^d,
    \end{align}
   where the $\beta_j$'s are piecewise linear/multilinear basis functions satisfying  
\begin{align*}
   \beta_j\geq 0 , \quad  \beta_j (x_i) = \delta_{j,i},  \quad
   \sum_{j} \beta_j (x) = 1, \quad \text{and} \quad \|I[\phi]
   -\phi\|_{0} = \|D^2\phi\|_0\rho^2 
\end{align*}   
for any $\phi\in C^2_b(\rd)$. A fully discrete scheme is then
obtained from \eqref{approx-HJB} as follows:
\begin{align}\label{schme_HJ}
%\begin{cases}
    \tilde u_{i,k}[\mu] = S_{\rho,h,r} [\mu] (\tilde u_{\cdot,k+1},i,k), \ k<N, \quad \text{and} \quad 
    \tilde u_{i,N}[\mu] = G (x_i, \mu (t_N)), &
%\end{cases}
\end{align}
where
\begin{align}\label{schme_HJ_S}
  \notag  S_{\rho,h,r} [\mu] & (v,i,k) =   \inf_{\alpha} \Bigg\{ h F ( x_i, \mu(t_k) ) + h L ( x_i, \alpha) + \frac{1-e^{-h\lambda_r}}{\lambda_r} \int_{|z| \geq r}  I [ v ] (x_i + z ) \nu (dz)   \\ 
    & + \frac{e^{-h \lambda_r}}{2d} \sum_{m=1}^d \Big( I [ v ] (x_i + h \bar{b} (\alpha) + \sqrt{hd} \sigma_r^m) + I [ v] (x_i+ h \bar{b} (\alpha) - \sqrt{hd} \sigma_r^m) \Big)   \Bigg\}.
\end{align}

Finally, we extend  the solution of the discrete scheme $\tilde u_{i,k}[\mu]$ to the whole
$\rd\times[0,T]$ by linear interpolation in $x$ and piecewise constant
interpolation in $t$:
\begin{align}\label{extnd_dscrt_schm_HJ}
\tilde u_{\rho,h}[\mu](t,x) = I\big(\tilde u_{\cdot,[\frac{t}{h}]}[\mu]\big)(x)= \sum_{i\in \zd}  \beta_i(x) \, \tilde u_{i,[\frac{t}{h}]}[\mu]\quad \mbox{for any} \quad (t,x)\in [0,T)\times \rd .
\end{align}

\subsection{Approximate optimal feedback control} \label{subsec:ctrl}\ 
\smallskip

\noindent 
For the HJB equation in \eqref{eqn:MFG}, satisfied by the
value function \eqref{eq:value-fun}, it easily follows that the optimal feedback
control is
$$\alpha(t,x) = D_p H(x, Du[\mu](t,x)).$$
 Based on this feedback law, we define an approximate
feedback control for the  discrete time optimal control problem
\eqref{discretization_process}--\eqref{discretized_value_fun}
 in the following way: For 
$h,\rho,\epsilon>0$ and $(t,x)\in \R^d\times[0,T]$,
\begin{align}\label{alp_num}\alpha_{\text{num}}(t,x) :=
  D_p H(x, D\tilde u^{\epsilon}_{\rho,h}[\mu](t,x)),
  \end{align}
where $\tilde u_{\rho,h}[\mu]$ is given by \eqref{extnd_dscrt_schm_HJ},
\begin{align}\label{cont_extnd_dscrt_schm_HJ}
\tilde u_{\rho,h}^{\epsilon}[\mu](t,x) =
\tilde u_{\rho,h}[\mu](t, \cdot)*\rho_{\epsilon}(x),
\end{align}
and the mollifier $\rho_{\epsilon}(x)=
\frac{1}{\epsilon^{d}}\rho\big(\frac{x}{\epsilon}\big)$ for $0\leq 
\rho \in C_c^{\infty}(\rd)$ with $\int_{\rd} \rho(x)dx=1$.  We state a standard result on mollification. 
\begin{lemma} \label{regularized_Lipschitz}
If $u\in W^{1,\infty}(\rd)$, $\epsilon>0$, and
$u^{\epsilon}=u*\rho_\epsilon$. Then 
$u^{\epsilon} \in C_b^\infty(\R^d)$, and there exists a constant $c_\rho>0,$ such that for all $\epsilon>0$,
\begin{align*}
\|u^{\epsilon} -u\|_0 \leq \|Du\|_0\, \epsilon\qquad  
\mbox{and} \qquad  \|D^p u^{\epsilon}\|_0 \leq c_\rho \|Du\|_0 \,\epsilon^{1-p}
\ \ \text{for any} \ \ p \in \N.
\end{align*}  
\end{lemma}
 By construction, we expect $\alpha_{\text{num}}$ to
 be an approximation of the optimal feedback 
 control for the approximate control problem with value function
 \eqref{eq:approx-value-fun} when $h,\rho,\epsilon$ are small and
 $\tilde u^{\epsilon}_{\rho,h}$ is close to $u$.

\subsection{Dual SL discretization of the FPK equation}\label{subsec:discrtFPK}

%\BC
%In this section we explicitly show the dependence of $u=u[\mu]$ on the fixed
%distribution $m=\mu\in C([0,T],\rd)$.
%\bo

\subsubsection*{A. Dual approximation of the FPK equation} 

First note that if $\tilde X_s=\tilde X^{0,Z_0}_s$ solves \eqref{SDE} with $t=0$ and $X_0  = Z_0$, a rv with distribution $m_0$, then the FPK equation for $\tilde m:=Law(\tilde X_s)$ is
% the FPK equation for $\tilde m:=Law(\tilde X_s)$, where $\tilde X_s=\tilde X^{0,Z_0}_s$ solves \eqref{SDE} with $t=0$, $X_0  = Z_0$, and a rv $Z_0$ with distribution $m_0$, is
\begin{align*}
\begin{cases}
        \tilde m_t - \mathcal{L}^*\tilde m - \text{div} ( \tilde m \alpha) = 0, \\
        \tilde m ( 0 ) = m_{0}.
\end{cases}
    %\label{fokker_planck_general}
\end{align*}
Setting $\alpha = \alpha_{\text{num}}$, this equation becomes an approximation of the FPK equation in \eqref{eqn:MFG}. With this choice of $\alpha$, we further approximate $\tilde m$ by the density $\tilde m_k:= Law(X_k)$, of the approximate process $X_k=X_k^{0,Z_0}$ solving \eqref{discretization_process} with $l=0$ and $X_0=Z_0$. 

We now derive a FPK equation for $\tilde m_k$ which in discretised form will serve as our approximation of the FPK equation in \eqref{eqn:MFG}. To simplify we consider dimension $d=1$. By definition of $\tilde m_k$, 
\begin{align*}
\mathbb{E}[\phi(X_{k+1})] = \int_{\R} \phi(x) \,  d\tilde m_{k+1}(x),
\end{align*}
for $\phi\in
C_b(\R^d)$ and $k \in \mathbb{N}\cup \{0\}$. Let $A_{k}$ be the event
of at least one jump in $[t_k,t_{k+1})$, i.e. $A_{k}= \{\omega:
N_{k+1}(\omega)- N_{k}(\omega)\geq 1  \}$ where $N_k$ is the random jump
time defined in Section \ref{subsec:SDE} B. Then by the definition of $X_k$ in
\eqref{discretization_process}, the fact that $N_k$, $J_k$, and $\xi_k$ are
i.i.d. and hence independent of $X_k$, and conditional expectations,
we find that 
\begin{align*}
& \int_{\R} \phi(x) \,  d\tilde m_{k+1}(x) = \mathbb{E}[\phi(X_{k+1})] \\
& = \mathbb{E}[\phi(X_{k+1})| A_{k}^c] \, P(A_{k}^c) + \mathbb{E}[\phi(X_{k+1})| A_{k}] \, P(A_{k})  \\
& = e^{-h \lambda_r} \mathbb{E}(\phi(X_k + h \bar{b}(\alpha_{\text{num}}) + \sqrt{h}\sigma_r \xi_{k})) + (1-e^{-h \lambda_r}) \mathbb{E}(\phi(X_k + J_i))  \\ 
& = \frac{e^{-h \lambda_r}}{2} \int_{\R} \big( \phi(x+h \bar{b}(\alpha_{\text{num}}) + \sqrt{h}\sigma_r) + \phi(x+h \bar{b}(\alpha_{\text{num}}) - \sqrt{h}\sigma_r)\big) \tilde m_k(dx) \\
& \qquad + (1-e^{-h \lambda_r})  \int_{\R} \int_{|z|>r} \phi(x+z) \frac{\nu(dz)}{\lambda_r} \tilde m_k(dx). 
\end{align*}
Let $ E_i:= \big(x_i- \frac{\rho}{2}, x_i + \frac{\rho}{2}\big)$,
$\tilde m_{j,k} =  \int_{E_j} \tilde m_k(dx)$. We approximate the above expression by a midpoint (quadrature) approximation, i.e. $\int_{E_j} f(x) \tilde m_{k}(dx) \approx f(x_j) \tilde m_{j,k}$,  then by choosing $\phi(x) = \beta_j(x)$ (linear interpolant) for $j\in \Z$ and using $\beta_j(x_i)= \delta_{j,i}$ we get a fully discrete approximation
\begin{align*} %\label{fokker_planck_1d}
\tilde  m_{j,k+1} \approx \sum_{i\in \Z} & \tilde  m_{i,k}\Big[ \frac{e^{-h \lambda_r}}{2} \Big(
\beta_j(x_i+h \bar{b}(\alpha_{\text{num}}) + \sqrt{h}\sigma_r) + \beta_j(x_i+h
\bar{b}(\alpha_{\text{num}}) - \sqrt{h}\sigma_r)\Big) \\
&  \qquad + \frac{1-e^{-h \lambda_r}}{\lambda_r} \int_{|z|>r} \beta_j(x_i+z) \nu(dz)\Big] . \notag
\end{align*}

In arbitrary dimension $d$, we
%take $\alpha=\alpha_{\text{num}}$ %from \eqref{alp_num} 
%in \eqref{defn_characteristics_general} and denote the result by
denote
 \begin{align}\label{defn_characteristics}
    \Phi^{\epsilon, \pm}_{j,k,p} := x_j - h\,\big( 
  H_{p} ( x_j, D\tilde u_{\rho,h}^{\epsilon}[\mu] (t_k,x_j) )    + B_r^{\sigma} \big)
 \pm \sqrt{hd} \sigma_r^p . % \\[0.21cm]
\end{align} 
for $j \in \zd$, $k=0,\ldots, N$, $p=1, \ldots,d$. Redefining $E_i := x_i + \frac{\rho}{2} (-1,1)^d$ and reasoning as for 
$d=1$ above, we get the following discrete FPK equation
\begin{align} \label{Fokker-Planck_discretized}
\begin{cases}
    \tilde m_{i,k+1} [ \mu ] & := \displaystyle \sum_{j\in \zd} \tilde m_{j,k}[\mu] \, \mathbf{B}_{\rho,h,r} [ H_{p} ( \cdot, D\tilde u_{\rho,h}^{\epsilon} [\mu]  ) ( i,j,k ), \\ 
    \tilde m_{i,0} & = \displaystyle \int_{E_i} dm_0(x),
\end{cases}
\end{align}
where 
\begin{align}
\begin{split}
     \mathbf{B}_{\rho,h,r} [ H_{p} ( \cdot, D\tilde u_{\rho,h}^{\epsilon} [\mu] ] ( i,j,k ) & := \bigg[  
        \frac{e^{-\lambda_r h} }{2d} \sum_{p=1}^d \Big(\beta_i \big( \Phi^{\epsilon, +}_{j,k,p} \big) 
        + \beta_i\big( \Phi^{\epsilon, -}_{j,k,p} \big) \Big) \\
     & \hspace*{2cm} +  \frac{1-e^{-\lambda_r
          h}}{\lambda_r} \int_{|z| > r} \beta_i (x_j+z) \nu (dz)
      \bigg].
      \end{split}
    \label{fokker-planck-general-discretized}
\end{align}

The solution is a probability
distribution on $\mathcal G_\rho \times h \mathcal  N_h$, where
$\mathcal{N}_h:=\{0,\dots,N\}$:  
\begin{lemma} \label{mass_conservation}
Let $(\tilde m_{i,k})$ be the solution of
\eqref{Fokker-Planck_discretized}. If $m_0\in P(\rd)$, then
$(\tilde m_{i,k})_i\in P(\Z^d)$, i.e. $\tilde m_{i,k}\geq 0$,  $i\in\Z^d$, and
$\sum_{j\in\Z^d} \tilde m_{j,k} =1$ for all $k\in\mathcal{N}_h$. 
    
\end{lemma}
\begin{proof}
First note that $\tilde m_{i,k}\geq 0$ follows directly from the definition
of the scheme and $m_{i,0}\geq 0$. Changing the order of summation and
as $\sum_i \mathbf{B}_{\rho,h,r} [ H_{p} ( \cdot, D\tilde u_{\rho,h}^{\epsilon} [\mu]] ( i,j,k ) =1$, we find that 
\begin{align*}
  \sum_i \tilde m_{i,k+1} = \sum_{i} \sum_{j} \tilde m_{j,k}  \mathbf{B}_{\rho,h,r} [ H_{p} ( \cdot, D\tilde u_{\rho,h}^{\epsilon} [\mu] ] ( i,j,k ) = \sum_{j} \tilde m_{j,k}.
\end{align*} 
The result follows by iteration since $\sum_{j}\tilde m_{j,0}=1$.
\end{proof}

We extend $(\tilde m_{i,k}[\mu])$ to $\rd$ by piecewise constant interpolation in
$x$ and then to $[0,T]$ by linear interpolation in $t$: For $t\in
[t_k,t_{k+1}]$ and $k\in \mathcal{N}_h$,
%by \eqref{Fokker-Planck_discretized} we define the probability measure $m^{\epsilon}_{\rho,h}(t_k)\in P(\rd)$ for each $k\in \mathcal{N}_h$ by 
\begin{align}\label{exten_disc_measure}
\tilde m_{\rho,h}^{\epsilon}[\mu](t,x)& := \frac{t-t_k}{h}
\tilde m_{\rho,h}^{\epsilon}[\mu](t_{k+1},x)+ \frac{t_{k+1}-t}{h}
\tilde m_{\rho,h}^{\epsilon}[\mu](t_{k},x), 
%\label{cont_exten_disc_measure}
\end{align}
where, 
$\tilde m_{\rho,h}^{\epsilon}[\mu](t_k,x) := \frac{1}{\rho^d} \sum_{i\in
  \zd} \tilde m_{i,k}[\mu] \, \mathbbm{1}_{E_i}(x)$. 
 Note that $\tilde m_{\rho,h}^{\epsilon}[\mu] \in C([0,T],P(\rd))$ and the
duality with the linear in $x$/constant in
$t$ interpolation  used for $\tilde{u}_{\rho,h}$ in
\eqref{extnd_dscrt_schm_HJ}.

\subsection{Discretisation of the coupled MFG system}\ 
\smallskip

\noindent 
The discretisation of the MFG system is obtained by coupling the two discretisations above by setting $\mu=\tilde m^\epsilon_{\rho,h}[\mu]$. With this choice and $u=\tilde u[\mu]$ and $m=\tilde m[\mu]$ we get the following discretisation of \eqref{eqn:MFG}:
\begin{align}\label{disc_mfg_system}
\begin{cases}
    u_{i,k} = S_{\rho,h,r} [m^\epsilon_{\rho,h}] (u_{\cdot,k+1},i,k),  \\[0.2cm]
    u_{i,N} = G (x_i, m^\epsilon_{\rho,h} (t_N)), \\[0.2cm]
    m_{i,k+1} =  \sum_{j\in \zd} m_{j,k} \, \mathbf{B}_{\rho,h,r} [ H_{p} ( \cdot, Du_{\rho,h}^{\epsilon}   ) ] ( i,j,k ), \\[0.2cm]
    m_{i,0} = \int_{E_i} dm_0(x),
\end{cases}
\end{align}
where $S_{\rho,h,r}, \mathbf{B}_{\rho,h,r}, u_{\rho,h}^{\epsilon}, m^\epsilon_{\rho,h}$ are defined above.

The individual discretisations are explicit, but due to the forward-backward nature of the coupling, the total discretisation is not explicit. It yields a nonlinear system that must be solved by some method like e.g.~a fixed point iteration or a Newton type method. 

The approximation scheme \eqref{disc_mfg_system} has a least one solution:
 
 \begin{proposition} \label{thm:existence_discrete_system}
    (Existence for the discrete MFG system) Assume
    \ref{A0}, \ref{A2}, \ref{L1}--\ref{L2}, \ref{F1}--\ref{F2}, \ref{H1}, and \ref{M1}.
    %the assumptions of Theorem \ref{thm:convergence_MFG} 
    %or Theorem \ref{thm:convergence_MFG-nondeg}.
    Then there exist a pair $ ( u_{\rho,h} , \ m_{\rho,h}^{\epsilon})$ 
    solving 
    \eqref{disc_mfg_system}.
    %the discrete MFG system 
    %in Algorithm
    %\ref{alg1}.
\end{proposition}
The proof of this result is non-constructive and given in Appendix \ref{app:pf_ex}.

\section{Convergence to the MFG system} \label{subsec:main-result}
 In this section we give the main theoretical results of this paper, various convergence results as $h,\rho,\epsilon,r \to 0$
 under CFL-conditions. The proofs will be given in Section \ref{sec:proof-main} and %required results for the individual schemes in Sections \ref{sec:HJB} and \ref{sec:FPK}. 
 require results for the individual schemes given in Sections \ref{sec:HJB} and \ref{sec:FPK}.

\subsection{Convergence to viscosity-very weak solutions}

We consider degenerate and non-degenerate cases separately. For the degenerate case, the convergence holds only in dimension $d=1$.

\begin{thm}[Degenerate case, $d=1$]\label{thm:convergence_MFG}
Assume \ref{A0}, \ref{A2}, \ref{L1}--\ref{L3}, \ref{F1}--\ref{F3}, 
\ref{H1}--\ref{H2}, \ref{M11}, $\{(u_{\rho,h}, m^{\epsilon}_{\rho,h})\}_{\rho,h,\epsilon>0}$ are  solutions of the discrete MFG system
\eqref{disc_mfg_system}. If $\rho_n,h_n,\epsilon_n,r_n\to 0$ under the CFL conditions $\frac{\rho_n^2}{h_n},\frac{h_n}{r_n^{\sigma}},\frac{\sqrt{h_n}}{\epsilon_n}=o(1)$, then:
\begin{itemize}
\item[(i)] $\{u_{\rho_n,h_n}\}_n$ is precompact in  $C_b([0,T]\times K)$ for every compact set $K \subset \R$. 
\item[(ii)] $\{m^{\epsilon_n}_{\rho_n,h_n}\}_n$ is sequentially precompact in $C ( [ 0,T ], P ( \R ) )$, and (a) in $L^1$ weak if $p \in ( 1,\infty )$ in \ref{M11}, or (b) in $L^{\infty}$ weak $*$ if $p= \infty$ in \ref{M11}.
% \item[(iii)] If $\rho_n,h_n,\epsilon_n,r_n,\frac{\rho_n^2}{h_n},\frac{h_n}{r_n^{\sigma}},\frac{\sqrt{h_n}}{\epsilon_n}\to 0$, and $(u,m)$ is a limit point of $\{(u_{\rho_n,h_n}, m^{\epsilon_n}_{\rho_n,h_n}) \}_n$, then $(u,m)$ is a viscosity-very weak solution of the MFG system \eqref{eqn:MFG}.
% \item[(iii)] (ALTERNATE OPTION) If  $(u,m)$ is the limit point of a converging subsequence $\{(u_{\rho_n,h_n}, m^{\epsilon_n}_{\rho_n,h_n}) \}_n$ as $\rho_n,h_n,\epsilon_n,r_n\to 0$,  then $(u,m)$ is a viscosity-very weak solution of the MFG system \eqref{eqn:MFG}.
% \item[(iii)]  If $(u_{\rho_n,h_n}, m^{\epsilon_n}_{\rho_n,h_n}) \to(u,m)$ in $C_b([0,T]\times\R)\times C ( [ 0,T ], P ( \R) )$ as $\rho_n,h_n,\epsilon_n,r_n$, $\frac{\rho_n^2}{h_n},\frac{h_n}{r_n^{\sigma}},\frac{\sqrt{h_n}}{\epsilon_n}\to 0$,  then $(u,m)$ is a viscosity-very weak solution of the MFG system \eqref{eqn:MFG}.
% \item[(iii)] If $\rho_n,h_n,\epsilon_n,r_n\to 0$ under CFL conditions $\frac{\rho_n^2}{h_n},\frac{h_n}{r_n^{\sigma}},\frac{\sqrt{h_n}}{\epsilon_n}=o(1)$, and $(u_{\rho_n,h_n}, m^{\epsilon_n}_{\rho_n,h_n}) \to(u,m)$ in $C_b([0,T]\times\R)\times C ( [ 0,T ], P ( \R) )$, then $(u,m)$ is a viscosity-very weak solution of the MFG system \eqref{eqn:MFG}.
\item[(iii)] If $(u,m)$ is a limit point of $\{(u_{\rho_n,h_n}, m^{\epsilon_n}_{\rho_n,h_n})\}_n$, then $(u,m)$ is a viscosity-very weak solution of the MFG system \eqref{eqn:MFG}.
% \item[(iii)] Assume $(u_{\rho_n,h_n}, m^{\epsilon_n}_{\rho_n,h_n}) \to(u,m)$ in $C_{b,\textup{loc}}([0,T]\times\R)\times C ( [ 0,T ], P ( \R) )$. If $\rho_n,h_n$, $\epsilon_n,r_n\to 0$ under CFL conditions $\frac{\rho_n^2}{h_n},\frac{h_n}{r_n^{\sigma}},\frac{\sqrt{h_n}}{\epsilon_n}=o(1)$, then $(u,m)$ is a viscosity-very weak solution of the MFG system \eqref{eqn:MFG}.
\end{itemize}
\end{thm}

Note that $\{m^{\epsilon}_{\rho,h}\}$ is precompact  in $C([0,T], P(\rd))$, 
just by assuming \ref{M1} for the initial distribution. But in the degenerate case 
this is not enough for convergence of the MFG system, due to lower regularity 
%(differentiability at every points) 
 of the solutions of the HJB equation (no longer $C^1$).  
Therefore we need assumption \ref{M11} and the stronger compactness given by Theorem \ref{thm:convergence_MFG}(ii) part (a) or (b).  
This latter result we are only able to show in $d=1$.

In arbitrary dimensions we assume more regularity on solutions of the HJB equation in \eqref{eqn:MFG}:  \smallskip

\begin{description*}
\item[(U)\label{U1}] Let $u[m]$ be a  
viscosity solution of the HJB equation in \eqref{eqn:MFG}. For any $m \in C([0,T],P(\rd))$ and $t\in(0,T)$, $u[m](t)\in C^1(\rd)$. 
\end{description*}
% \begin{description*}
% \item[(U)\label{U1}] For  $m \in C([0,T],P(\rd))$, if $u[m]$ is a %the unique 
% viscosity solution of the HJB equation in \eqref{eqn:MFG}, then $u[m](t)\in C^1(\rd)$. 
% \end{description*}

\begin{remark}
Assumption \ref{U1} holds in non-degenerate cases, e.g. under assumption \ref{nu3}, see Theorem \ref{mfg:classical_solution} and the discussion below.
\end{remark}  
We have the following convergence result in arbitrary dimensions.

\begin{thm}[Non-degenerate case]\label{thm:convergence_MFG-nondeg}
Assume \ref{A0}, \ref{A2}, \ref{L1}--\ref{L3}, \ref{F1}--\ref{F3}, \ref{H1}--\ref{H2}, \ref{U1}, \ref{M1}, $\{(u_{\rho,h}, m^{\epsilon}_{\rho,h})\}_{\rho,h,\epsilon>0}$ are  solutions of the discrete MFG system
\eqref{disc_mfg_system}. If $\rho_n,h_n,\epsilon_n,r_n\to 0$ under the CFL conditions $\frac{\rho_n^2}{h_n},\frac{h_n}{r_n^{\sigma}},\frac{\sqrt{h_n}}{\epsilon_n}=o(1)$, then:
\begin{itemize}
\item[(i)] $\{u_{\rho_n,h_n}\}_n$  is precompact in  $C_b([0,T]\times K)$ for every compact set $K \subset \rd$. 
\item[(ii)] $\{m^{\epsilon_n}_{\rho_n,h_n}\}_n$ is precompact in $C([0,T],P(\rd))$.  
% \item[(iii)] If $\rho_n,h_n,\epsilon_n,r_n,\frac{\rho_n^2}{h_n},\frac{h_n}{r_n^{\sigma}},\frac{\sqrt{h_n}}{\epsilon_n}\to 0$, and $(u,m)$ is a limit point of $\{(u_{\rho_n,h_n}, m^{\epsilon_n}_{\rho_n,h_n}) \}_n$, then $(u,m)$ is a viscosity-very weak solution of the MFG system \eqref{eqn:MFG}.
% \item[(iii)] (ALTERNATE OPTION) If  $(u,m)$ is the limit point of a converging subsequence $\{(u_{\rho_n,h_n}, m^{\epsilon_n}_{\rho_n,h_n}) \}_n$ as $\rho_n,h_n,\epsilon_n,r_n\to 0$,  then $(u,m)$ is a viscosity-very weak solution of the MFG system \eqref{eqn:MFG}.
\item[(iii)] If $(u,m)$ is a limit point of $\{(u_{\rho_n,h_n}, m^{\epsilon_n}_{\rho_n,h_n})\}_n$, then $(u,m)$ is a viscosity-very weak solution of the MFG system \eqref{eqn:MFG}.
\end{itemize}
\end{thm}

% The proofs of these three results are given in Section \ref{sec:proof-main}. 
 These results give compactness of the approximations and convergence along subsequences. To be precise, by part (i) and (ii) there are convergent subsequences, and by part (iii) the corresponding limits are solutions of the MFG system \eqref{eqn:MFG}. 

We immediately have existence for \eqref{eqn:MFG}. 
\begin{corollary}[Existence of solutions of \eqref{eqn:MFG}] \label{corollary:mfg_ex}
Under the assumptions of either Theorem \ref{thm:convergence_MFG} or \ref{thm:convergence_MFG-nondeg}, there exists a viscosity-very weak solution $(u,m)$ of the MFG system \eqref{eqn:MFG}.   
\end{corollary} 

If in addition we have uniqueness for the MFG system \eqref{eqn:MFG}, then we have full convergence of the sequence of approximations.

\begin{corollary} \label{corollary:mfg_convergence}
Under the assumption of either Theorem \ref{thm:convergence_MFG} or Theorem \ref{thm:convergence_MFG-nondeg}, if the MFG system \eqref{eqn:MFG} has at most one viscosity-very weak solution, then the whole 
%the whole (otherwise there is a overfull badbox)
sequence $\{(u_{\rho_n,h_n}, m^{\epsilon_n}_{\rho_n,h_n}) \}_n$ converges to a limit $(u,m)$ which is the (unique) viscosity-very weak solution of the MFG system \eqref{eqn:MFG}.   
\end{corollary} 

\subsection{Convergence to classical solutions}
In the case the individual equations are regularising, we can get convergence to classical solutions of the MFG system. To be precise we need:
\smallskip
\begin{itemize}
\item[1.] (``Weak" uniqueness of individual PDEs) The HJB equation have unique viscosity solutions, and the FPK equation have unique very weak solutions.\smallskip

\item[2.] (Smoothness of individual PDEs) Both equations have classical solutions.
\smallskip
\end{itemize}
This means that viscosity-very weak solutions of the MFG system automatically (by uniqueness for individual equations) are classical solutions.
If in addition 
\smallskip
\begin{itemize}
    \item [3.] (Classical uniqueness for MFG) classical solutions of the MFG system are unique,
\end{itemize}
we get full convergence of the approximate solutions to the solution of the MFG system.\smallskip

We now give a precise result in the setting of \cite{ersland2020classical}, see Theorem \ref{mfg:classical_solution} in Section \ref{sec:prelim} for existence and uniqueness of classical solutions of \eqref{eqn:MFG}.

\begin{corollary} \label{corollary:mfg_convergence_classical}
   Assume \ref{A0}--\ref{nu3}, \ref{L1}--\ref{L3}, \ref{F1}--\ref{F4}, \ref{H3}--\ref{H4}, and
   \ref{M111}. Let $(u_{\rho,h}, m^{\epsilon}_{\rho,h})$ be  solutions of the discrete MFG system
\eqref{disc_mfg_system}. If $\rho_n,h_n,\epsilon_n,r_n\to 0$ under the CFL conditions $\frac{\rho_n^2}{h_n},\frac{h_n}{r_n^{\sigma}},\frac{\sqrt{h_n}}{\epsilon_n}=o(1)$, then:

\smallskip\noindent   (a) $\{(u_{\rho_n,h_n}, m^{\epsilon_n}_{\rho_n,h_n})\}_n$ has a convergent subsequence in the space $C_{b,\textup{loc}}([0,T]\times\rd)\times C ( [ 0,T ], P ( \rd) )$, and any limit point is a classical-classical solution of \eqref{eqn:MFG}.
   
 \smallskip\noindent  (b) If in addition \ref{F5} and \ref{H5} hold, then the whole sequence in (a)
 %$ \{( u_{\rho,h}, m_{\rho,h}^{\epsilon} )\}$
 converges to the unique classical-classical solution $(u,m)$ of \eqref{eqn:MFG}.
\end{corollary}

\begin{proof}
1. Assumption \ref{U1} holds by Theorem \ref{mfg:classical_solution}, and then by Theorem \ref{thm:convergence_MFG-nondeg}, there is a convergent subsequence $\{(u_{\rho_n, h_n}, m_{\rho_n,h_n}^{\epsilon_n} )\}_n$ such that
$(u_{\rho_n, h_n}, m_{\rho_n,h_n}^{\epsilon_n} ) \to (u,m)$ and $(u,m)$ is a viscosity-very weak solution of \eqref{eqn:MFG}. 
% We know from Theorem \ref{thm:convergence_MFG-nondeg} that $u$ is a viscosity solution of the HJB equation and $m$ is a very weak solution of the FPK equation, solving the equations in \eqref{eqn:MFG} simultaneously. 

\smallskip\noindent
2. Since $m \in C ([0,T] , P (\rd))$, the viscosity
solution $u$ is unique by Proposition \ref{prop:viscosity_sol_HJB} (b) (see also \cite[Theorem $5.3$]{ersland2020classical}). Hence it coincides with the classical $C_b^{1,3} ((0,T) \times \rd)$ solution given by 
%Theorem \ref{mfg:classical_solution}.
\cite[Theorem $5.5$]{ersland2020classical}. 

\smallskip\noindent
3. Now $D_p H (x,Du (t)) \in C_b^2 (\rd)$ by part 2 and \ref{H3}, and then by Proposition \ref{uniqueness_weak_soln_fp} there is at most one very weak solution of the FPK equation. Hence it coincides with the classical $C_b^{1,2}((0,T) \times \rd)$ solution given by \cite[Proposition $6.8$]{ersland2020classical}. 
%  Moreover, there exists a classical $C_b^{1,2}((0,T) \times \rd)$ solution of the FPK equation  by \cite[Proposition $6.8$]{ersland2020classical}. Since classical solutions are very weak solutions, by uniqueness this means that $m\in C_b^{1,2}((0,T) \times \rd)$  is a classical solution of the FPK equation. Hence part (a) follows.

\smallskip\noindent
4. In addition if 
%Since now $(u,m)$ are classical solutions, by 
%the monotonicity conditions on $F$ and $G$ 
\ref{F5} and 
%uniform convexity of $H$ 
\ref{H5} hold, there is a most one classical solution $(u,m)$ by Theorem \ref{mfg:classical_solution}~(b). 

\smallskip\noindent
5. This shows (compactness, smoothness, and uniqueness) that all convergent subsequences of $\{(u_{\rho_n, h_n}, m_{\rho_n,h_n}^{\epsilon_n} )\}_n$ have the same limit, and thus the whole sequence converges to $(u,m)$, the unique classical solution of  \eqref{eqn:MFG}.
%(a) follows from 1 -- 3, and (b) from 4 -- 5. 
\end{proof}

% \begin{remark}
% In Corollary \ref{corollary:mfg_convergence_classical}~(a)  we give a partially 
% constructive existence proof for solutions of \eqref{eqn:MFG}. The proof depends on
% Proposition \ref{thm:existence_discrete_system}, which uses Schauder fixed point theorem.
% %This means that the proof is not fully constructive.
% \end{remark}

\subsection{Extension and discussion}\label{subsec:ext}
\subsection*{Extension to more general L\'evy operators}

The results of Theorem \ref{thm:convergence_MFG} and \ref{thm:convergence_MFG-nondeg} 
hold under much more general assumptions on the L\'evy operator 
$\mathcal{L}$. 
In \cite{ersland2020classical} they use \ref{A0} together with the assumptions,
\begin{description*}
 \item[($\nu$1$'$)\label{nu22}]
                $ \displaystyle r^{-2+\sigma}\int_{|z|<r} |z|^2 d\nu + r^{-1+\sigma}\int_{r<|z|<1}
                |z| d\nu + r^{\sigma}\int_{r<|z|<1} d\nu\leq c
                 ,  \  r\in(0,1) $.
            \smallskip

 \item[($\nu$2$'$)\label{nu33}] There are $\sigma \in (1,2)$ and $\mathcal K
>0$ such that the heat kernels $K_\sigma$ and $K_\sigma^*$ of $\mathcal
L$ and $\mathcal L^*$ satisfy for $K=K_\sigma,K_\sigma^*$\ :
$K\geq0$, $\|K(t,\cdot)\|_{L^1(\R^d)}=1$, and 
        \begin{align*}
            \|D^{\beta} K
  (t,\cdot) \|_{L^p (\R^d)} \leq \mathcal K t^{-\frac{1}{\sigma}\big(|\beta|+(1-\frac1p)d\big)}\quad
            \text{for $t\in(0,T)$}
        \end{align*}
and any $p\in[1,\infty)$ and multi-index $\beta\in \N^{d}\cup \{0\}$.
%\bii \sout{where $D$ is the
%gradient in $\R^d$}, 
\end{description*}

\medskip

\noindent where the heat kernel of the operator $\mathcal{L}$ is defined as the fundamental solution of the heat equation $\partial_{t} u - \mathcal{L} u = 0$.
These assumptions cover lots of new cases compared to \ref{A0}, \ref{A2}, and \ref{nu3}. 
New cases include (i) sums of operators satisfying 
\ref{A2} on subspaces spanning $\rd$, having possibly different orders, 
(ii) more general non-absolutely continuous L\'evy measures, 
and (iii) L\'evy measures supported on positive cones. 
An example of (i) (cf. \cite{ersland2020classical}) is
$$\mathcal L=-\Big(\!-\frac{\partial^2}{\partial
  x_1^2}\Big)^{\sigma_1/2}-\dots-\Big(\!-\frac{\partial^2}{\partial 
      x_d^2}\Big)^{\sigma_d/2}, \qquad \sigma_1,\dots,\sigma_d\in(1,2),$$
      which satisfies \ref{nu22} with $\sigma=\min_i\sigma_i$
   and $d\nu(z)= 
   \sum_{i=1}^d\frac{dz_i}{|z_i|^{1+\sigma_i}}\Pi_{j\neq
  i}\delta_0(dz_j)$. 
  This is a sum of one-dimensional fractional Laplacians of different orders.
  An example of (iii) is given by the spectrally positive ``fractional Laplacian" in one space dimension: 
  $\mathcal{L} u = c_{\sigma} \int_{0}^{\infty} ( u ( x+z ) - u ( x ) - Du ( x ) \cdot z \mathbbm{1}_{\{z < 1\}} ) z^{-1-\sigma} dz$.

We have the following generalization of the wellposedness result for classical solutions
given in Theorem \ref{mfg:classical_solution}.
\begin{thm}[\cite{ersland2020classical}] \label{thm:classical_extension}
    Theorem \ref{mfg:classical_solution} holds when you replace \ref{A2} -- \ref{nu3} by \ref{nu22} -- \ref{nu33}.
 \end{thm}
 It follows that \ref{U1} holds whenever Theorem \ref{thm:classical_extension} holds. Since \ref{A2} implies \ref{nu22} and the integrals in \ref{nu22} are what appear in the different proofs, it is easy to check that all estimates in this paper are true for L\'evy measures satisfying \ref{nu22} instead of \ref{A2}. 
This means that under assumption \ref{nu22} and \ref{nu33} 
we have the following extensions of Theorems \ref{thm:convergence_MFG}
and \ref{thm:convergence_MFG-nondeg} and Corollary \ref{corollary:mfg_convergence_classical}.

\begin{thm}
    Theorem \ref{thm:convergence_MFG} holds when you replace \ref{A2} with \ref{nu22}.
\end{thm}

\begin{thm}
    Theorem \ref{thm:convergence_MFG-nondeg} holds when you replace \ref{A2} -- \ref{nu3} by \ref{nu22} -- \ref{nu33}.
 \end{thm}
\begin{corollary}
    Corollary \ref{corollary:mfg_convergence_classical} holds when you replace \ref{A2} -- \ref{nu3} by \ref{nu22} -- \ref{nu33}.
\end{corollary}

\subsection*{The Wasserstein metric $d_1$ versus our metric $d_0$}
The typical setting for the FPK
 equations in the MFG literature seems to be the metric space $ ( P_{1} ( \rd ) , d_{1} )$, that is the $1-$Wasserstein space
 $\mathcal{W}_{1}$ of probability measures with finite first moment. This is also the case in \cite{MR3828859} where convergence results are given for SL schemes for local nondegenerate MFGs in $\rd$. In this paper we can not assume finite first moments if we want
 to cover general non-local operators. An example is the fractional Laplacian $ -( - \Delta)^{\frac{\sigma}{2}}$ 
 for $\sigma < 1$, where the underlying
 $\sigma$-stable process only has finite moments
 of order less than $\sigma$.
 Instead we consider the weaker metric space $(P ( \rd ), d_{0} )$, which
 is just a metrization of the weak (weak-* in $C_{b}$) convergence of probability measures.
In this topology we can consider processes, probability measures and solutions
of the FPK equations that do not have
any finite moments or any restrictions on the tail behaviour of the corresponding L\'evy measures.

 Of course, under additional assumptions convergence in $d_0$ implies convergence in $d_1$.

%\bii we have the following result:
%if a sequence that converges in $ ( P ( \rd ) ,d_{0} )$ has uniformly bounded \bii more than first moments, then it will also converge in $ ( P_{1} ( \rd ), d_{1} )$.
\begin{lemma}
If $m_{n}$ converges to $m$ in $(P ( \rd ),d_0)$ and $m_{n}$ and $m$ has uniformly bounded $(1+\delta)$-moments for $\delta>0$,
then $m_{n} \to m$ in $ ( P_{1} ( \rd ) , d_{1} ) $.
\end{lemma}
Convergence in $P_{1} ( \rd )$ \cite[Definition 6.8]{villani2008optimal} is by definition equivalent to  weak convergence plus convergence of first moments, and the result follows from e.g. Proposition 1.1 and Lemma 1.5 in \cite{ACDPS:Note}.

We then have the following version of Theorem \ref{thm:convergence_MFG} and Theorem \ref{thm:convergence_MFG-nondeg}.
\begin{corollary}
    Assume $m_{0} \in P_{1+\delta} ( \mathbb{R} )$, $\int_{\rd \setminus B_{1}} |z|^{1+\delta} d \nu ( z ) < \infty $ for some $\delta>0$,
    and the assumptions of Theorem \ref{thm:convergence_MFG} and Theorem \ref{thm:convergence_MFG-nondeg}. 
    Then the statements of Theorem \ref{thm:convergence_MFG} and Theorem \ref{thm:convergence_MFG-nondeg}
    hold if we replace $(P,d_0)$ by $(P_{1},d_1)$ in part (ii).
\end{corollary}
Note that the number of moments of $m$ is determined by the number of moments of  $1_{|z|>1}\nu$ (and $m_0$), see e.g. the discussion in section 2.3 in \cite{ersland2020classical}. Moreover, if $1_{|z|>1}\nu$ has at most $\alpha$ finite moments, then $\mathcal Lu$ is well-defined only if $u$ has at most order $\alpha$ growth at infinity. Hence in the nonlocal case there is "duality" between the moments of $m$ and the growth of $u$. Note that $um$ will always be integrable which is natural since then e.g. $Eu(X_t,t) =\int u(x,t)m(dx,t)$ is finite. 

In our case we assume no moments and have to work with bounded solutions $u$.

\subsection*{On moments and weak compactness in $L^p$ in the degenerate case}
Previous results for Semi-Lagrangian schemes in the first order
and the degenerate second order case \cite{carliniSilva2014semi1st, carlini2015semi} 
cover the case $m_{0} \in P_{1} ( \mathbb{R} ) \cap L^{\infty} ( \mathbb{R} )$, 
which means that $m_{0}$ has finite first-moments.
Our results assume
$m_{0} \in P ( \mathbb{R} ) \cap L^{p} ( \mathbb{R} )$, for $p \in ( 1,\infty ]$, and hence
no moment bounds and possibly unbounded $m_{0}$. When $p < \infty$ we
have weak compactness in $L^{1}$ instead of weak-* compactness in $L^{\infty}$. 

Since our results in the degenerate case allows for $\mathcal L=0$, they immediately give an extension to this $P  \cap L^{p}$ setting for the convergence results for first order problems of \cite{carliniSilva2014semi1st}. Moreover, the same conditions, arguments, and results easily also holds in the local diffusive case considered in \cite{carlini2015semi}. 

\section{On the SL scheme for the HJB equation} \label{sec:HJB}

We prove results for the numerical approximation of the HJB equation, including monotonicity, consistency, and different uniform a priori stability and regularity estimates. Using the ``half-relaxed" limit method \cite{barles1991convergence}, we then show convergence in the form of $v_{\rho_{n}, h_{n}} [ \mu_{n} ] ( t_{n},x_{n} ) \rightarrow v [ \mu ] ( t,x )$, where $v [ \mu ]$ is the (viscosity) solution of the continuous HJB equation. Let $B ( \mathcal{G}_{\rho})$ be the set of all bounded functions defined on  $\mathcal{G}_{\rho}$. 

\begin{thm}\label{thm:schm_prop}
    Assume \ref{A0}, \ref{L1}, $\rho,h,r >0$, $\mu \in  C ( [0,T], P ( \rd ))$, and let $S_{\rho,h,r} [ \mu ]$ denote the scheme defined in \eqref{schme_HJ}.
        \begin{itemize} 

            \item[(i)] (Bounded control) If
                $\phi \in \text{Lip} ( \mathbb{R}^d )$, $S_{\rho,h,r} [ \mu ] ( \phi,i,k )$ 
                 has a minimal control and $|\alpha| \leq K$ 
                where $K$ only depends on  $\| D_{x} \phi  \|_{0} $ and 
                the growth of $L$ as $|x| \rightarrow \infty$. \medskip
       \item[(ii)] (Monotonicity) For all $v,w \in B ( \mathcal{G}_{\rho} ) $ with $v\leq w$ we have,
            \[
                S_{\rho,h,r} [\mu] (v,i,k) \leq S_{\rho,h,r} [\mu] (w,i,k) \text{ for all } i\in\mathcal{G}_{\rho}, \ k= 0,\ldots, N-1.
            \]
        \item[(iii)] (Commutation by constant) For every $c\in \R$ and $w \in B (\mathcal{G}_{\rho} )$,
            \[
                S_{\rho,h,r} [\mu] (w+c,i,k) = S_{\rho,h,r} [\mu] (w,i,k) +c \text{ for all } i\in\mathcal{G}_{\rho}, \ k= 0,\ldots, N-1.
            \]
        \end{itemize}
Assume also \ref{A2} and \ref{F2}.
        \begin{itemize}
        \item[(iv)] (Consistency) Let $\rho_{n},h_{n}, r_{n} \xrightarrow{n \to \infty} 0$ under CFL conditions $\frac{\rho_{n}^{2}}{h_{n}},\frac{h_{n}}{ r_{n}^{\sigma}} = o ( 1 )$, 
            grid points $(t_{k_{n}}, x_{i_{n}}) \to (t,x)$, and  
            $\mu_{n},\mu \in C ( [0,T]; P (\R^d))$ such that 
            $\mu_{n} \to \mu$.  
            Then, for every $\phi \in C_{c}^{\infty} (\R^d \times [0,T))$,  
            \begin{align*}
                \lim_{n \to \infty }\frac{1}{h_{n}} \big[ \phi ( t_{k_{n}+1}, x_{i_{n}}) - & S_{ \rho_{n}, h_{n}, r_n} [ \mu_{n} ] ( \phi_{\cdot,k_{n}+1}, i_{n}, k_{n} ) \big] \\
            = &  - \partial_{t} \phi ( t,x ) - \inf_{\alpha \in \mathbb{R}^d } \big[  L ( x, \alpha ) - D \phi \cdot \alpha  \big]  - \mathcal{L} \phi ( x )  - F ( x, \mu ( t )  ).
            \end{align*}
    \end{itemize}
\end{thm}

\begin{proof}
    
 \noindent   (i) Since 
\begin{align*}
    h ( \alpha ) := \frac{e^{-h \lambda_{r}}}{2d} \sum_{m=1}^{d}  I [ \phi ] ( x_{i} + h \bar{b} ( \alpha ) + \sqrt{h} \sigma_{r}^{m}) + I [ \phi ] ( x_{i} + h \bar{b} ( \alpha ) - \sqrt{h} \sigma_{r}^{m})
\end{align*}
is Lipschitz in $\alpha$ (maximum linear growth at infinity), 
while  $L ( x, \alpha )$ is coercive (more than linear growth at infinity) by \ref{L1}, 
there exists a ball $B_{R}$, where $R$ depends on the Lipschitz constant 
of  $I [ \phi ]$ and the growth of $L$, 
such that the minimizing control $ \bar{\alpha}$ of $S_{\rho,h,r} [ \mu ] ( \phi,i,k )$ 
belongs to $ B_{R}$. \medskip

\noindent (ii) and (iii) Follows directly from the definition of the scheme. \medskip

\noindent (iv)  For ease of notation, we write $\rho,h,r,\mu$ instead of $\rho_{n}, h_{n}, r_{n}, \mu_{n}$. A $4$th order Taylor expansion of $\phi$ gives
    \begin{align*}
        \phi (x+h \bar{b}(\alpha)  \pm & \sqrt{hd} \sigma_{r}^{m})  =  \phi (x) + D \phi (x) \cdot (h \bar{b}(\alpha) \pm \sqrt{hd} \sigma_r^{m} )  \\
        & + \frac{hd}{2} ( \sigma_{r}^{m})^T  D^2 \phi (x) \sigma_r^{m} \pm  \sqrt{d}\,h^{\frac{3}{2}}  b(\alpha)^{T} D^{2} \phi ( x ) \sigma_{r}^{m} + \frac{h^{2}}{2} \bar{b} ( \alpha )^{T} D^{2} \phi ( x ) \bar{b}( \alpha ) \\[0.2cm]
        &  + \sum_{|\beta| = 3} \frac{D^{\beta} \phi ( x )}{\beta !} ( h \bar{b} ( \alpha ) \pm \sqrt{hd}\, \sigma_{r}^{m} )^{\beta} + \sum_{|\beta| = 4} \frac{D^{\beta} \phi ( \xi_{\pm} )}{\beta !} ( h \bar{b} ( \alpha ) \pm \sqrt{hd}\, \sigma_{r}^{m} )^{\beta},
    \end{align*}
    for some $\xi_{\pm} \in \mathbb{R}^{d}$. Using that $\bar{b} ( \alpha ) = - \alpha - \int_{ r \leq |z| \leq 1 } z \nu ( dz )   $, and by \ref{A2} $\int_{r \leq |z| \leq 1} z \nu ( dz ) = O ( r^{1-\sigma} )$, we get that 
    \begin{align}\label{err_bnd-consist00}
  &  \phi (x+h \bar{b}(\alpha) + \sqrt{hd} \sigma_{r}^{m}) + \phi (x+h \bar{b}(\alpha) - \sqrt{hd} \sigma_{r}^{m}) -2\phi (x) \\
   \notag =   -2 h D & \phi ( x ) \cdot \alpha - 2h \int_{ r < |z| <1 } D \phi ( x ) \cdot z \nu ( dz )
        + hd ( \sigma_{r}^{m})^T \cdot D^2\phi(x) \cdot \sigma_r^{m}   + \mathcal{O}\big(h^2r^{2-2 \sigma}\big).
    \end{align}
    We used that $\frac{h^{2}}{2} \bar{b} ( \alpha )^{T} D^{2} \phi ( x ) \bar{b}( \alpha )$ is of order $\mathcal{O} ( h^{2} r^{2-2 \sigma} )$,
    the $3$rd order terms are of order  $\mathcal{O} ( h^{3} r^{3- 3 \sigma} + h^{2} r^{1-\sigma})$, 
     and the $4$th order terms are of order 
     $( \sqrt{hd} \sigma_{r} )^{4} = \mathcal{O} (h^2 r^{4-2 \sigma}  )$. Then the
    error of the Taylor expansion is $O ( h^{2} r^{2- 2 \sigma} )$.
    Using Lemma \ref{lem:small_jump},
\begin{align} \label{consist_esti0}
    \notag &\phi ( x_{i} ) - S_{\rho,h,r} [ \mu ] ( \phi, i,k )\\
\notag  & = \  \phi ( x_{i} ) - \inf_{\alpha} \bigg[ h F ( x_{i}, \mu ( t_{k+1} ) ) + h L (x_{i}, \alpha )  + \frac{e^{-h \lambda_r}}{2d} \sum_{m=1}^{d} \Big(  2\phi ( x_{i} ) -  2h D \phi ( x_{i} ) \cdot \alpha  \\ 
\notag        & \hspace{2.5cm} + h d ( \sigma_{r}^{m} )^{T} D^{2} \phi ( x_{i} ) \sigma_{r}^{m} - 2 h \int_{ r < |z| < 1 } D \phi ( x_{i} ) \cdot z \nu ( dz ) \Big) \\
\notag        & \hspace{2.5cm} + \frac{1-e^{-h\lambda_r}}{\lambda_r} \int_{|z| > r}  \phi (x_{i} + z )  \nu (dz) + \mathcal{O} ( \rho^{2} ) + \mathcal{O}(h^2 r^{2-2\sigma} )\bigg] \\
        & = \  h F ( x_{i}, \mu ( t_{k+1} ) ) - \inf_{\alpha} \bigg[  h L ( x_{i}, \alpha ) - h e^{-h \lambda_{r}}   D \phi ( x_{i} ) \cdot \alpha \bigg] + ( 1 -  e^{-h \lambda_{r}}) \phi ( x_{i})  \\ 
\notag        &\quad  - h e^{-h \lambda_{r}} \Big(\mathcal{L}_{r} \phi ( x_{i} ) +\mathcal{O}(r^{3-\sigma})\Big)  +  h e^{-h \lambda_{r}} \int_{ r < |z|< 1 } D \phi ( x_{i} ) \cdot z \nu ( dz )  \\
\notag        & \quad -  \frac{1-e^{-h\lambda_r}}{\lambda_r} \int_{|z|> r}  \phi (x_{i} + z ) \nu (dz) + \mathcal{O} ( \rho^{2} + h^2 r^{2-2\sigma}).
    \end{align}

 \noindent Using that $\int_{|z|<r} |z|^2 \nu(dz)\leq K r^{2-\sigma}$ (by \ref{A2}), 
 for the small jump operator $\mathcal{L}_{r}$ (defined in \eqref{inner_outer_operator}) we have 
 \begin{align}\label{consist_esti1}
 |\mathcal{L}_r\phi(x_i) - e^{-h\lambda_r}\mathcal{L}_r\phi(x_i)|\leq h\lambda_r \, r^{2-\sigma}\|D^2\phi\|_{0}.
\end{align}    
Again, as $\int_{r<|z|<1} |z|\nu(dz)\leq K r^{1-\sigma}$ and $\int_{|z|>1} \nu(dz)\leq K$, 
for the long jump operator $\mathcal{L}^{r}$ (defined in \eqref{inner_outer_operator}) we have  that
\begin{align}\label{consist_esti2}
  \notag       \Big|\mathcal{L}^r\phi(x_i) & +e^{-h \lambda_{r}} \int_{ r <|z| < 1} D \phi ( x_{i} ) \cdot z \nu ( dz ) - \frac{1-e^{-h\lambda_r}}{h \lambda_r} \int_{|z| > r}  (\phi (x_{i} + z ) - \phi(x_i)) \nu (dz)\Big|\\
 \notag  & \leq    K (1-e^{-h \lambda_{r}}) r^{1-\sigma} \|D\phi\|_0  + K \Big(1-\frac{1-e^{-h\lambda_r}}{h \lambda_r}\Big) \Big( r^{1-\sigma} \|D\phi\|_0  + \|\phi\|_0 \Big) \\
         & \leq  K \Big( h \lambda_r r^{1-\sigma} \|D\phi\|_0 + h\lambda_r \|\phi\|_0\Big). 
\end{align} 

\noindent Recalling that $\mathcal{L} \phi ( x_{i} ) = \mathcal{L}_{r} \phi ( x_{i} ) + \mathcal{L}^{r} \phi ( x_{i} ) $,
combining \eqref{consist_esti0} with \eqref{consist_esti1} and \eqref{consist_esti2}, we find 
\begin{align*}
\phi ( x_{i} ) - S_{\rho,h,r} [ \mu ] ( \phi, i,k ) & = h F ( x_{i}, \mu ( t_{k+1} ) ) - h \inf_{\alpha} \bigg[   L ( x_{i}, \alpha ) -  D \phi ( x_{i} ) \cdot \alpha \bigg]  - h  \mathcal{L} \phi ( x_{i} )  \\[0.2cm] 
        &\hspace*{1cm} + \mathcal{O} \big( h^2\lambda_r + hr^{3-\sigma} + h^2 \lambda_r r^{1-\sigma} + \rho^2 + h^2 r^{2-2\sigma}\big).
\end{align*}
 As $|\lambda_r| \leq C r^{-\sigma}$, we have 
\begin{align*}
& \frac{\phi(t_{k},x_i)- \phi(t_{k+1},x_i)}{h} +  \frac{1}{h} \Big(\phi(t_{k+1},x_i)  - S_{ \rho, h} [ \mu ] ( \phi_{\cdot,k+1}, i, k )\Big) \\
& =  - \partial_t \phi(t_{k},x_i) -   \mathcal{L} \phi (t_{k+1}, x_{i} ) +  F ( x_{i}, \mu ( t_{k+1} ) )   -  \inf_{\alpha} \bigg[   L (x_{i}, \alpha)  -  D \phi ( t_{k+1},x_{i} ) \cdot \alpha \bigg]    \\
& \hspace{4cm} + \mathcal{O} \big(h+ h r^{-\sigma} + r^{3-\sigma} + h r^{1-2\sigma} + \frac{\rho^2}{h} + h r^{2-2\sigma} \big).
\end{align*}  
Hence the result follows by taking the limit $n \to \infty$ with $\frac{\rho_{n}^2}{h_{n}} ,\frac{h_{n}}{r_{n}^{\sigma}} = o ( 1 )$.
\end{proof}

\begin{thm} \label{hjb_discrete_comparison}
    (Comparison) Assume $\mu_1,\mu_2\in C([0,T],P(\rd))$, \ref{A0}, and \ref{L1}.
    Let $u_{\rho,h}[\mu_1]$ and ${u}_{\rho,h}[\mu_2] $ be defined by the scheme 
    \eqref{extnd_dscrt_schm_HJ} for $\mu=\mu_1,\mu_2$, respectively. Then,
    \begin{align*}
        \| u_{\rho,h}[\mu_1] - u_{\rho,h}[\mu_2]  \|_{0} \leq T \| F(\cdot,\mu_1) - F(\cdot,\mu_2)   \|_{0} + \| G(\cdot,\mu_1(T)) - G(\cdot,\mu_2(T))  \|_{0} .  
    \end{align*}
\end{thm}

\begin{proof}
Let $c^{\pm}_{m}(\alpha) :=  h \bar{b} (\alpha) \pm \sqrt{hd}
\sigma_r^m$, and note that
\begin{align}  \label{partition_unity2}
    I [u_{\cdot, k+1}[\mu_1]] ( x) - I[u _{\cdot,k+1}[\mu_2]]( x) = &  \sum_{p \in \Z^d} \beta_p (x) (u_{p,k+1}[\mu_1]- u _{p,k+1}[\mu_2] ) .
\end{align}
 By \eqref{schme_HJ} and the definition of $\inf$, for any $\epsilon>0$,
there is $\alpha_{\epsilon} \in \R^d$ such that  
\begin{align}
        u_{i,k}& [\mu_2]  \geq   \,h F  (  x_i,\mu_2(t_{k}) ) 
        + h L (x_{i}, \alpha_{\epsilon}) 
        +  \frac{e^{-h \lambda_r}}{2d} \sum_{m=1}^d \Big[ 
            I [ u_{\cdot, k+1}[\mu_2] ] (x_i + c^{+}_{m}(\alpha_{\epsilon})) \notag \\ 
    &+ I [ u_{\cdot,k+1}[\mu_2] ] (x_i+ c^{-}_{m}(\alpha_{\epsilon})) \Big] + \frac{1-e^{-h\lambda_r}}{\lambda_r} \int_{|z| \geq r} I [ u_{\cdot,k+1}[\mu_2] ] (x_i+z) \nu (dz) - \epsilon.
    \label{optimal_alpha2}
\end{align}
We then find, using \eqref{schme_HJ}, \eqref{partition_unity2}, 
\eqref{optimal_alpha2},
\begin{align*}
  u_{i,k}[\mu_1]  & - u_{i,k}[\mu_2]  \leq  h \big( F  (  x_i,\mu_1(t_{k}) ) - F  (  x_i,\mu_2(t_{k}) \big) + h ( L (x_{i}, \alpha_{\epsilon}) - L (x_{i}, \alpha_{\epsilon}) ) \\
     &\quad  +   \sum_{p \in \Z^d} \bigg[ \frac{e^{-h \lambda_r}}{2d} \sum_{m=1}^d \Big(\beta_p ( c^{+}_{m}(\alpha_{\epsilon}))  + \beta_p( c^{-}_{m}(\alpha_{\epsilon})) \Big)\big(u_{p+i,k+1}[\mu_1] - u_{p+i,k+1}[\mu_2] \big) \\
     &\quad  + \frac{1-e^{-h\lambda_r}}{\lambda_r} \int_{|z| \geq r}
    \beta_p (z) \big(u_{p+i,k+1}[\mu_1] - u_{p+i,k+1}[\mu_2]\big) \nu (dz) \bigg] + \epsilon\\
  &\leq h \| F(\cdot,\mu_1) - F(\cdot,\mu_2)   \|_{0} + c\sup_i |u_{i,k+1} - \tilde{u}_{i,k+1}| + \epsilon,
\end{align*}
where since  $\sum_p
\beta_p\equiv 1$, 
\begin{align*}
  c=\frac{e^{-h \lambda_r}}{2d} \sum_{m=1}^d\sum_{p \in \Z^d}  \Big(\beta_p ( c^{+}_{m}(\alpha_{\epsilon}))  + \beta_p( c^{-}_{i}(\alpha_{\epsilon})) \Big)+ \frac{1-e^{-h\lambda_r}}{\lambda_r} \int_{|z| \geq r}
  \sum_{p \in \Z^d}\beta_p (z) \nu (dz)=1.
\end{align*}
Since 
$|u_{i,N}[\mu_1] - u_{i,N}[\mu_2] | \leq \| G(\cdot,\mu_1(t_N)) - G(\cdot,\mu_2(t_N)) \|_0 $, a symmetry and
iteration argument shows that
\begin{align*} 
    \big|u_{i,k}[\mu_1] - u _{i,k}[\mu_2]\big| \leq ( N-k ) h  \| F(\cdot,\mu_1) - F(\cdot,\mu_2) \|_0 +  \| G(\cdot,\mu_1(t_N)) - G(\cdot,\mu_2(t_N)) \|_0 .
\end{align*}
The result then follows from interpolation and $T = Nh$.
\end{proof}

 The SL scheme is very stable in the sense that we have uniform in
$h,\rho,\mu$  boundedness, Lipschitz continuity, and
semi-concavity of the solutions $u_{i,k}$. 

\begin{lemma}\label{lem:aprrox_HJB_reg} 
    Let $\mu \in  C ( [0,T], P ( \rd ))$ and $u_{i,k}[\mu]$ be defined by the scheme \eqref{schme_HJ}.
  \begin{enumerate}[(a)]
      \item (Lipschitz continuity) Assume \ref{A0}, \ref{L2} and \ref{F2}. Then   
                     \begin{align*}
                         \frac{|u_{i,k} - u_{j,k} |}{|x_{i} - x_{j}|} \leq  ( L_{F} + L_{L} )(T-t_k) + L_G  , \quad i,j \in \mathbb{Z}^{d}, \ k \in \{ 0,1, \ldots N \}. 
            \end{align*}
        \item (Semi-concavity) Assume \ref{A0} , \ref{L3} and \ref{F3}.   Then 
            \begin{align*}
                \frac{u_{i+j,k} - 2 u_{i,k} + u_{i-j,k}}{ |x_j|^2} \leq  ( c_F +c_L ) (T-t_k) + c_G, \quad i,j \in \mathbb{Z}^{d}, \ k \in \{ 0,1, \ldots N \} .
            \end{align*}
        \item (Uniformly bounded) Assume \ref{A0}, \ref{L0}--\ref{L2},
  \ref{F1}, and \ref{F2}. Then   
                     \begin{align*}
                         |u_{i,k}| \leq  (C_{F}+C_{L}(K))(T-t_k) + C_G, \quad i,j \in \mathbb{Z}^{d}, \ k \in \{ 0,1, \ldots N \},
                     \end{align*}
                     where $K$ is defined in Theorem \ref{thm:schm_prop} (i). 
    \end{enumerate}
\end{lemma}

\begin{proof}\textbf{(a)} Note that since
    $\beta_{m} (x_{j}+x) = \beta_{m-j} (x)$,
\begin{align}  \label{partition_unity}
    I [u_{\cdot, k+1}] (x_{j} + x) - I[u_{\cdot,k+1}](x_i + x) = &  \sum_{p \in \Z^d} \beta_p (x) (u_{p+j,k+1} - u_{p+i,k+1} ) .
\end{align}
 %By similar computations as in Theorem \ref{hjb_discrete_comparison}, we then get
Then, by \ref{L2}, \ref{F2}, and similar
computations as in Theorem \ref{hjb_discrete_comparison}, we find that
%
% By \eqref{schme_HJ} and the definition of $\inf$, for any $\epsilon>0$,
%there is $\alpha_{\epsilon} \in \R^d$ such that  
%\begin{align}
%    \begin{split}
%    u_{i,k} \geq  & \,h F (x_i,m(t_{k+1}) ) + h L (x_{i}, \alpha_{\epsilon}) +  \frac{e^{-h \lambda_r}}{2d} \sum_{m=1}^d \Big[ I [ u_{\cdot, k+1} ] (x_i + c^{+}_{m}(\alpha_{\epsilon})) \\ 
%    &+ I [ u_{\cdot,k+1} ] (x_i+ c^{-}_{m}(\alpha_{\epsilon})) \Big] + \frac{1-e^{-h\lambda_r}}{\lambda_r} \int_{|z| \geq r} I [ u_{\cdot, k+1} ] (x_i+z) \nu (dz) - \epsilon.
%\end{split}
%    \label{optimal_alpha}
%\end{align}
%We then find, using \eqref{schme_HJ}, \eqref{partition_unity}, 
%\eqref{optimal_alpha}, \ref{L2}, and \ref{F2} that
\begin{align*}%\label{proof_lip_1}
   u_{j,k}   - u_{i,k} 
%  &\leq  h ( F (x_j,m(t_{k+1}) ) - F (x_i, m (t_{k+1})) + h ( L (\alpha_{\epsilon}, x_j) - L (\alpha_{\epsilon}, x_i) ) \\
%     &\quad  +   \sum_{p \in \Z^d} \bigg[ \frac{e^{-h \lambda_r}}{2d} \sum_{m=1}^d \Big(\beta_p ( c^{+}_{m}(\alpha_{\epsilon}))  + \beta_p( c^{-}_{i}(\alpha_{\epsilon})) \Big)(u_{p+j,k+1} - u_{p+i,k+1} ) \\
%     &\quad  + \frac{1-e^{-h\lambda_r}}{\lambda_r} \int_{|z| \geq r}
%    \beta_p (z) (u_{p+j,k+1} - u_{p+i,k+1}) \nu (dz) \bigg] + \epsilon\\
  \leq h(L_f+L_L) |x_i-x_j| + \sup_j |u_{i,k+1}-u_{j,k+1}| + \epsilon,
\end{align*}
%where since  $\sum_p
%\beta_p\equiv 1$, 
%\begin{align*}
%  c=\frac{e^{-h \lambda_r}}{2d} \sum_{m=1}^d\sum_{p \in \Z^d}  \Big(\beta_p ( c^{+}_{m}(\alpha_{\epsilon}))  + \beta_p( c^{-}_{i}(\alpha_{\epsilon})) \Big)+ \frac{1-e^{-h\lambda_r}}{\lambda_r} \int_{|z| \geq r}
%  \sum_{p \in \Z^d}\beta_p (z) \nu (dz)=1.
%\end{align*}
Since 
$|u_{i,N+1} - u_{j,N+1} | = | G (x_i, m(t_{N+1}) ) - G (x_j,
m(t_{N+1}) ) | \leq L_G |x_i - x_j|$ by \ref{F2}, the result follows
by iteration. \\[0.2cm]
%\begin{align*} 
%   |u_{j,k} - u_{i,k}| \leq \Big( (N-k)h (  L_F + L_L)  + L_G\Big)|x_i - x_j| .
%\end{align*}
\noindent \textbf{(b)} Similar to \eqref{partition_unity} we see
\begin{align*}
    &I [u_{\cdot, k+1}] (x_{i+j} + x) - 2 I[u_{\cdot,k+1}](x_i + x) + I [u_{\cdot, k+1}] (x_{i-j} + x) \\
  = &  \sum_{p \in \Z^d} \beta_p (x_i +x) (u_{p+j,k+1} - 2 u_{p,k+1} + u_{p-j, k+1} ) .
\end{align*}
Then, by \ref{L3}, \ref{F3}, and similar
computations as in Theorem \ref{hjb_discrete_comparison}, we find that
\begin{align*}%\label{dis_semicon_proof_1}
    u_{i+j,k} - 2 u_{i,k} + u_{i-j,k} \leq ( c_L + c_{F}  )h |x_j|^2  + \sup_i(u_{i+j,k+1} - 2 u_{i,k+1} + u_{i-j,k+1}). 
\end{align*}
Since $u_{i+j,N} - 2 u_{i,N} + u_{i-j,N}\leq c_G|x_j|^2$ by
\ref{F3}, the result follows by iteration.

  \smallskip

\noindent 
\textbf{(c)}  By part (a) and Theorem \ref{thm:schm_prop} (i),
$|\alpha|\leq K$, and then a direct calculation shows that
  \begin{align*}
- \sup_{|\alpha|\leq K}\Big( h (|F|+|L|) + \sup_j|u_{j,k+1}|  \Big)\leq  u_{i,k} \leq \sup_{|\alpha|\leq K}\Big( h (|F|+|L|) + \sup_j|u_{j,k+1}|  \Big).
    \end{align*}
The result follows from \ref{L1} and \ref{F1}. 
\end{proof}

\begin{thm} \label{HJ_convergence}
    (Convergence of the HJB   scheme)  
    Assume \ref{A0}, \ref{A2},  \ref{F1}, \ref{F2}, \ref{L2}, $\rho_{n},h_{n}, r_{n} \xrightarrow{n \to \infty}  0$  under CFL conditions
    $\frac{\rho_{n}^{2}}{h_{n}} ,\frac{h_{n}}{r_{n}^{\sigma}} = o ( 1 )$, 
    $\mu_n\rightarrow \mu$ in $C([0,T],P(\rd))$, 
    and $u_{\rho_{n},h_{n}}[\mu_n]$ is the solution of the scheme \eqref{schme_HJ} defined by \eqref{extnd_dscrt_schm_HJ}.  Then there is a continuous bounded function $u[\mu]$ such that $u_{\rho_{n},h_{n}}[\mu_n]\rightarrow u[\mu]$
     locally uniformly  in $\rd \times [ 0,T ]$, and  $u[\mu]$ is the viscosity solution of the HJB equation in \eqref{eqn:MFG} for $m=\mu$.
\end{thm}
\begin{proof}
    The result follows from the  Barles-Perthame-Souganidis relaxed limit method \cite{barles1991convergence},  using the monotonicity, consistency, and $L^\infty$-stability properties of the scheme (cf. Theorem \ref{thm:schm_prop} (ii), (iii), and Lemma \ref{lem:aprrox_HJB_reg} (c)), 
%    scheme given by , Lemma
%    \ref{lem:aprrox_HJB_reg} (a), 
and the strong comparison principle
    for the HJB equation in Proposition \ref{prop:viscosity_sol_HJB} (a). 
    %The standard details can
    We refer to the proof of \cite[Theorem 3.3]{carliniSilva2014semi1st} for a standard but more detailed proof in a similar case. 
\end{proof}

We recall that the continuous extensions $u_{\rho, h} [ \mu ] ( t,x )$ and 
$u_{\rho, h}^{\epsilon} [ \mu ] ( t,x )$ are defined in 
\eqref{extnd_dscrt_schm_HJ} and \eqref{cont_extnd_dscrt_schm_HJ}, 
respectively. 
The results of Lemma \ref{lem:aprrox_HJB_reg} transfers to $u_{\rho,h}^{\epsilon} [ \mu ] ( t,x )$.

\begin{lemma} \label{semiconcavity_eps}
    Let  $\mu \in  C ( [0,T], P ( \rd ))$ and $u_{\rho,h}^{\epsilon} [ \mu ]$ be given by \eqref{cont_extnd_dscrt_schm_HJ}.
\begin{enumerate}
    \item[(a)]  (Lipschitz continuity)  Assume \ref{A0}, \ref{L2} and \ref{F2}. 
         Then
 \begin{align*}
     \big|  u_{\rho, h}^{\epsilon} [ \mu ] ( t,x ) - u_{\rho, h}^{\epsilon} [ \mu ] ( t,y ) \big| \leq ( ( L_{L} + L_{F} )T + L_{G} ) |x-y|.
 \end{align*}
\item[(b)] (Approximate semiconcavity) 
        Assume \ref{A0}, \ref{L2},\ref{L3}, \ref{F2}, and \ref{F3}.
        Then  there exist a constant $c_1>0$, independent of $\rho, h, \epsilon$ and $\mu$, such that 
        \begin{align*}
            \qquad\qquad& u_{\rho, h}^{\epsilon} [ \mu ] ( t, x + y ) - 2 u_{\rho, h}^{\epsilon} [ \mu ] ( t,x ) + u_{\rho, h}^{\epsilon} [ \mu ] ( t,x - y ) \leq c_{1} (  |y|^{2} + \rho^{2} + \frac{\rho^2}{\epsilon} ), \ \mbox{and}  \\
             & \langle D u_{\rho, h}^{\epsilon} [ \mu ] (  t,y ) - D u_{\rho, h}^{\epsilon} [ \mu ] ( t,x ) , y-x \rangle \leq c_{1} \Big(|x-y|^{2} + \frac{\rho^2}{\epsilon^2}\Big).     
        \end{align*}
\item[(c)]   Assume $d=1$, \ref{A0}, \ref{L3}, and \ref{F3}. Then there exists a constant $c_{2} >0$, independent of $\rho, h, \epsilon$ and $\mu$, such that for each $i,j \in \zd$ and $k \in \mathcal  N_h$
        \begin{align*}
            \langle D u_{\rho, h}^{\epsilon} [ \mu ] (  t_k,x_j ) - D u_{\rho, h}^{\epsilon} [ \mu ] ( t_k,x_i ) , x_j-x_i \rangle \leq c_{2} |x_j-x_i|^{2}  .  
        \end{align*}
\end{enumerate}
\end{lemma}

\begin{proof}
    \noindent   (a) Since $u_{i,k}$ satisfies the discrete Lipschitz bound of Lemma \ref{lem:aprrox_HJB_reg}~(a), $u_{\rho,h} [ \mu ]$ is Lipschitz with same Lipschitz constant as  $u_{i,k}$ by properties of linear interpolation, and $u_{\rho,h}^{\epsilon} [ \mu ]$ is Lipschitz with same constant as $u_{\rho,h} [ \mu ]$ by properties of mollifiers (Lemma \ref{regularized_Lipschitz}).
     
    % Note that $u_{\rho,h}^{\epsilon} [ \mu ]$ is Lipschitz with same constant as $u_{\rho,h} [ \mu ]$ by properties of mollifiers (Lemma \ref{regularized_Lipschitz}), and $u_{\rho,h} [ \mu ]$ is Lipschitz with same constant as  $u_{i,k}$ (see Lemma \ref{lem:aprrox_HJB_reg}~(a)) by properties of linear interpolation.  
    \medskip

 \noindent (b)  For $i,j \in  \mathbb{Z}^{d}$ we have by Lemma \ref{lem:aprrox_HJB_reg}~(b),  $u_{i+j} + u_{i-j} - 2 u_{i} \leq c |x_{j}|^{2}.$
%  \begin{align*}
%  u_{i+j} + u_{i-j} - 2 u_{i} \leq c |x_{j}|^{2}.
%  \end{align*}
 Multiplying both sides by $\beta_{i} ( x )$, and summing over $\mathbb{Z}^{d}$, we get
 \begin{align*}
     u_{\rho,h} ( x + x_{j} ) + u_{\rho,h} ( x-x_{j} ) -2 u_{\rho,h} ( x ) \leq c |x_{j}|^{2}.
 \end{align*}
 Letting $x \to x - z$, multiplying by a positive mollifier $\rho_{\epsilon} ( z )$ and integrating, we get
 \begin{align*}
     u_{\rho,h}^{\epsilon} ( x + x_{j} ) + u_{\rho,h}^{\epsilon} ( x-x_{j} ) -2 u_{\rho,h}^{\epsilon} ( x ) \leq c |x_{j}|^{2}.
 \end{align*}
 We multiply both sides with $\beta_{j} ( y )$, and sum over $\mathbb{Z}^{d}$,
 \begin{align*}
     I [ u_{\rho,h}^{\epsilon} ] ( x+y ) + I [ u_{\rho,h}^{\epsilon} ] ( x-y ) - 2 I [ u_{\rho,h}^{\epsilon} ] ( x ) \leq c I [ |\cdot|^{2} ] ( y ) \leq  c (  |y|^{2} + \rho^{2} ).
 \end{align*}
By Lemma \ref{regularized_Lipschitz} and part (a), we have that $|I [ u_{\rho,h}^{\epsilon} ] ( \xi ) - u_{\rho,h}^{\epsilon} ( \xi )|\leq  K \| D^{2} u_{\rho,h}^{\epsilon}  \|_{0} \rho^{2} \leq  K \frac{\rho^{2}}{\epsilon}$,  where the Lipschitz bound  $K$ depends on the constants in \ref{L2} and \ref{F2}. Thus,
 \begin{align*}
      u_{\rho,h}^{\epsilon}  ( x+y ) + u_{\rho,h}^{\epsilon}  ( x-y ) - 2 u_{\rho,h}^{\epsilon}  ( x ) \leq  c (  |y|^{2} + \rho^{2} + \frac{\rho^{2}}{\epsilon} ).
 \end{align*}
 The second part of (b) then follows as in  \cite[Remark 6]{achdou2014numerical}.\medskip

 \noindent (c) The proof is given in \cite[Lemma 3.6]{carliniSilva2014semi1st}.
\end{proof}
\normalcolor

Under our assumptions, the continuous HJB equation
has a (viscosity) solution  $u ( t ) \in W^{1,\infty} ( \mathbb{R}^d )$, 
that is, the derivative exists almost everywhere \cite[Theorem 4.3]{ersland2020classical}. 
We have the following result for $Du_{\rho,h}^{\epsilon} [ \mu ]$.

\begin{thm} \label{thm:HJB-convergence-smoothsol}
    Assume \ref{A0}, \ref{A2}, \ref{L1}--\ref{L2}, \ref{F1}--\ref{F2},  $\rho_{n},h_{n}, r_{n}, \epsilon_n \xrightarrow{n \to \infty} 0$  under CFL conditions
    $\frac{\rho_{n}^{2}}{h_{n}} ,\frac{h_{n}}{r_{n}^{\sigma}} = o ( 1 )$, 
    and $\mu_n\rightarrow \mu$ in $C([0,T],P(\rd))$.
Let $u_{\rho_{n},h_{n}}^{\epsilon_n}[\mu_n]$ be defined by 
    \eqref{cont_extnd_dscrt_schm_HJ} and $u[\mu]$ the viscosity solution of the HJB equation in \eqref{eqn:MFG} for $m=\mu$. Then 
    
    \begin{itemize}
    \item[(i)]   $u_{\rho_{n}, h_{n}}^{\epsilon_{n}} [ \mu_{n} ] \rightarrow u [ \mu ]$
        locally uniformly,
    \item[(ii)] Assume also \ref{L3}, \ref{F3} and $\frac{\rho_n}{\epsilon_n}=o(1)$.  Then 
        $Du_{\rho_{n}, h_{n}}^{\epsilon_{n}} [ \mu_{n} ] ( t,x ) \rightarrow Du [ \mu ] ( t,x )$ whenever $Du ( t,x )$ exists, that is, the convergence is almost everywhere. 
    
    \item[(iii)] Assume also \ref{L3}, \ref{F3}, $\frac{\rho_n}{\epsilon_n}=o(1)$,  and  \ref{U1}. Then  $Du_{\rho_{n}, h_{n}}^{\epsilon_{n}} [ \mu_{n} ] \rightarrow Du [ \mu ]$ locally uniformly.
    \end{itemize}   
\end{thm}

\begin{proof}
    (i) This follows from the convergence result Theorem \ref{HJ_convergence} and Lemma \ref{regularized_Lipschitz}.   

    \noindent (ii) and (iii). We refer to \cite[Theorem $3.5$]{carliniSilva2014semi1st} and \cite[Proposition $5.1$]{MR3828859}. Estimates from Lemma \ref{semiconcavity_eps} are needed. For completeness we give the proof in Appendix \ref{app:gradlim}.
\end{proof}

\section{On the dual SL scheme for the FPK equation}\label{sec:FPK}

In this section we establish more properties of the discrete FPK equation \eqref{Fokker-Planck_discretized}, including tightness, equicontinuity in time, $L^1$-stability of solutions with respect to $\mu$, and $L^p$-bounds in dimension $d=1$.
%We prove the technical results on discrete scheme of FPK equation in this section.  % Theorem \ref{thm:thightness_m} guarantees that the sequence $\{m^{\epsilon}_{\rho, h}(t)\}$ is tight for any $t\in [0,T]$, and by Prokhorov Theorem, it is compact in the probability space $P(\rd)$. Then along with the equicontinuity result in time (c.f. Theorem \ref{thm:time_equicont_m}), the compactness of $\{m^{\epsilon}_{\rho, h}\}$ in $C([0,T], P(\rd))$ would follow. 
% Theorem  \ref{thm:thightness_m} and  \ref{thm:time_equicont_m} deal with the tightness and time equicontinuity results respectively, which guarantee, by the Prokhorov and Arzel\`a--Ascoli theorem, the compactness of $\{m^{\epsilon}_{\rho, h}\}_{\rho,h, \epsilon}$ in $C([0,T], P(\rd))$. 
% Furthermore, in Theorem \ref{thm:Lp_esti}, we establish uniform $L^p$ bound of the sequence $\{m^{\epsilon}_{\rho, h}\}$ in dimension $1$, which  plays a crucial role to show that the limit point is  actually a very weak solution of the FPK equation in degenerate cases.   
 To prove tightness we will use a result from \cite{Espen-Indra-Milosz-2020}.
% First we state an important Proposition, proved in \cite{Espen-Indra-Milosz-2020}, which is crucial for our tightness result. 
\begin{proposition}  \label{prop:tail-control-function}
Assume \ref{A0} and \ref{M1}. Then there exists a function $0\leq \Psi \in C^{2}(\rd)$ with $\|D\Psi\|_{0},$ $\|D^2\Psi\|_{0} < \infty$, and  $\displaystyle \lim_{|x|\rightarrow \infty} \Psi(x) = \infty$, such that 
\begin{align}\label{eq:tail-control-function}
\sup_{x\in \rd}\Big|\int_{|z|>1} \big(\Psi(x+z) - \Psi(z) \big)\nu(dz)\Big| < \infty \quad \mbox{and} \quad \int_{\rd} \Psi(x) \, m_0(dx)<\infty.
\end{align}
\end{proposition}

%\begin{proof}
%We refer to  \cite[Proposition 3.8 and Lemma 3.9 (ii)]{Espen-Indra-Milosz-2020} for the proof.
%\end{proof}

\begin{proof}
We use \cite[Lemma 4.9]{Espen-Indra-Milosz-2020} on the family of measures $\{\nu_1 , m_0\}$, where $\nu_1$ is defined in \eqref{lambda_r}, to get  
 a function $\Psi(x) = V_0(\sqrt{1+|x|^2})$ such that $V_0:[0,\infty)\rightarrow [0,\infty)$ is a non-decreasing sub-additive function, $\|V_0'\|_{0}, \|V_0''\|_0<\infty$, $\displaystyle \lim_{x\rightarrow \infty} V_0(x) = \infty$,  and 
\begin{align*}
\int_{\rd} \Psi(x) \, \mu(dx)<\infty \qquad \mbox{for} \qquad \mu\in \{\nu_1 , m_0\}. 
\end{align*}
We immediately get the result except for the first part of \eqref{eq:tail-control-function}. But this estimate follows from sub-additivity and $\nu_1$-integrability of $V_0$, see \cite[Lemma 4.13 (ii)]{Espen-Indra-Milosz-2020}.
\end{proof}

\begin{remark}
\textit{(a)} If  $\frac{d\nu}{dz}\leq \frac{C}{|z|^{d+\sigma_1}}$ for $|z|>1$ and $\int_{\rd} |x|^{\sigma_2} \, m_0(dx)< \infty$ for $\sigma_1,\sigma_2>0$, then  $\Psi(z)= \log(\sqrt{1+|z|^2})$ is a possible
explicit choice for the function in Proposition \ref{prop:tail-control-function}.
 \smallskip 

\noindent \textit{(b)} 
Since $\Psi \in C^2(\rd)$, the first part of   \eqref{eq:tail-control-function} is equivalent to  $\|\mathcal{L}\Psi\|_0 < \infty$ (see \cite[Lemma 4.13 (ii)]{Espen-Indra-Milosz-2020}).
%Instead of $\, \sup_{x\in \rd}\big|\int_{|z|>1} \big(\Psi(x+z) - \Psi(z) \big)\nu(dz)\big| < \infty$, we can equivalently write $\|\mathcal{L}\Psi\|_0 < \infty$  in Proposition \ref{prop:tail-control-function},  as $\psi \in C^2(\rd)$ (cf. \cite[Lemma 3.13 (ii)]{Espen-Indra-Milosz-2020}). 
%Note that, we want to emphasis the property of tails behavior of the respective measures,  by writing \eqref{eq:tail-control-function}.     
\end{remark} 

\begin{lemma}
Assume $\{\mu_\alpha\}_{\alpha\in A}\subset P(\rd)$ and there exists a function $0\leq \psi \in C(\rd)$ such that $\lim_{|x|\rightarrow \infty} \psi(x)=\infty$ and  $\sup_{\alpha}\int_{\rd} \psi(x) \mu_\alpha(dx) \leq C$. Then $\{\mu_\alpha\}_{\alpha}$ is tight. 
\end{lemma}
This result is classical and can be proved in a similar way as the Chebychev inequality. 

% \begin{proof}
% \be For every $\epsilon>0$, $K_{\epsilon}:= \{x: \psi(x)\leq C/\epsilon\}$ is compact since $\displaystyle \lim_{|x|\rightarrow \infty} \psi(x)=\infty$. Then  for any $\alpha \in A$, 
% \begin{align*}
% \mu_{\alpha}(K_\epsilon^c) = \int_{K_\epsilon^c} \mu_{\alpha}(dx) \leq \frac{\epsilon}{C} \int_{K_\epsilon^c} \psi(x) \mu_{\alpha}(dx) \leq \frac{\epsilon}{C} \int_{\rd} \psi(x) \mu_{\alpha} (dx) \leq \epsilon.
% \end{align*} 
% \be Hence $\{\mu_\alpha\}_{\alpha}$ is tight and the proof is complete.
% \end{proof}

\begin{thm}[Tightness]\label{thm:thightness_m}
    Assume \ref{A0}, \ref{A2}, \ref{L1}--\ref{L2}, \ref{F2}, \ref{H1}, \ref{M1},  the CFL condition $\frac{\rho^2} h,hr^{1-2\sigma}=\mathcal{O}(1)$,  $\mu \in  C ( [0,T], P ( \rd ))$, and $m^{\epsilon}_{\rho,h}[\mu]$ is defined by \eqref{exten_disc_measure}. Take $\Psi$ as in Proposition \ref{prop:tail-control-function}. Then there exists $C>0$, independent  of $\rho, h, \epsilon$  and $\mu$, such that
\begin{align*}
\int_{\R^d} \Psi(x) \, dm^{\epsilon}_{\rho,h}[\mu](t) \leq C \qquad \mbox{for any } \qquad t\in[0,T].
\end{align*}
\end{thm}

\begin{proof}
Essentially we start by
multiplying the scheme \eqref{Fokker-Planck_discretized} by $\Psi$ and integrating in space.
By the definition of $m^{\epsilon}_{\rho,h}=m^{\epsilon}_{\rho,h}[\mu]$ in \eqref{exten_disc_measure}  and \eqref{Fokker-Planck_discretized},
%(we omit the dependence on $\mu$), 
we find that
\begin{align*}
 \int_{\R^d} \Psi(x) dm^{\epsilon}_{\rho,h}(t_{k+1})  &  =  \frac{1}{\rho^d}\sum_{i \in \Z^d} m_{i,k+1} \int_{E_i} \Psi(x) dx \\ 
& =    \sum_{i \in \Z^d} \frac{1}{\rho^d} \int_{E_i} \Psi(x) dx  \sum_{j} m_{j,k}  \, \mathbf{B}_{\rho,h,r} [ H_{p} ( \cdot, Du_{\rho,h}^{\epsilon} ) ] ( i,j,k ).
 %&=  \sum_{i \in \Z^d} \frac{1}{\rho^d} \int_{E_i} \Psi(x) dx  \sum_{j} m_{j,k}\bigg( e^{-\lambda_r h} \Big[ \frac{1}{2d} \sum_{p=1}^d \beta_i (\Phi^{\epsilon,+}_{j,k,p}) )  + \beta_i(\Phi^{\epsilon,-}_{j,k,p} ) \Big]   \\
%& \hspace*{3cm} + \frac{1-e^{-\lambda_r h}}{\lambda_r} \int_{|z| \geq r} \beta_i (x_j+z) \nu (dz) \bigg),
 \end{align*}
By the definition of $\mathbf{B}_{\rho,h,r}$ in \eqref{fokker-planck-general-discretized} and interchanging the order of summation and integration, we have
\begin{align*}
 \int_{\R^d} \Psi(x) d\,m^{\epsilon}_{\rho,h}(t_{k+1}) 
=&   \sum_{j\in \Z^d} \frac{m_{j,k}}{\rho^{d}} \bigg[ \frac{e^{-\lambda_r h}}{2d}\sum_{p=1}^d\sum_{i\in \Z^d} \int_{E_i} \Psi(x) \big(\beta_i(\Phi^{\epsilon,+}_{j,k,p})+ \beta_i(\Phi^{\epsilon,-}_{j,k,p})\big) dx \\
& \hspace{1.7cm} +\frac{1-e^{-\lambda_r h}}{\lambda_r}\int_{|z|>r} \sum_{i\in \Z^d} \int_{E_i}\Psi(x) \beta_i(x_j  +z) dx \, \nu(dz) \bigg].
\end{align*}
Since $\Psi \in C^2(\rd)$, by properties of midpoint approximation  and linear/multilinear interpolation we have $\big| \frac{1}{\rho^d} \int_{E_i} \Psi(x) dx - \Psi(x_i)\big|  
+  \big|\Psi(x) - \sum_{i\in \Z^d} \beta_i(x) \Psi(x_i)\big|\leq \mathcal{O}(\rho^{2})$.
%\begin{align*}
%& \Big| \frac{1}{\rho^d} \int_{E_i} \Psi(x) dx - \Psi(x_i)\Big|  
%+  \Big|\Psi(x) - \sum_{i\in \Z^d} \beta_i(x) \Psi(x_i)\Big|\leq \mathcal{O}(\rho^{2}).
%\end{align*} 
Therefore
\begin{align}\label{esti_1}
  \int_{\R^d} \Psi(x) d\,m^{\epsilon}_{\rho,h}(t_{k+1})
\leq &  \sum_{j\in \Z^d} m_{j,k}\bigg[\frac{e^{-\lambda_r h}}{2d} \sum_{p=1}^d \big( \Psi\big(\Phi^{\epsilon,+}_{j,k,p}\big) + \Psi\big(\Phi^{\epsilon,-}_{j,k,p}\big)  \big) \\
&\notag  \hspace{1cm}+ \frac{1-e^{-\lambda_r h}}{\lambda_r}\int_{|z|>r} \Psi(x_j+z) \, \nu(dz)\bigg] + \mathcal{O}(\rho^{2}).
\end{align}

We estimate the terms on the right  hand side. %starting with the $\Phi$-terms. 
Let $\Phi^{\epsilon,\pm}_{j,k,p}= x_j\pm a^{\pm}_{h,j}$ where
\begin{align}\label{charectatistic_dif}
& a^{\pm}_{h,j}= h\,\Big( D_p H \big(x_j, Du^{\epsilon}_{\rho,h}(t_k,x_j)\big) + B_r^{\sigma}\Big)  \pm \sqrt{h} \sigma_r^p. 
% \\ &  \mbox{and,} \quad   a^{-}_{h,j}= h\,( Du^{\epsilon}_{\rho,h}(t_k,x_j)+ B_r^{\sigma}) - \sqrt{h} \sigma_r^p. 
\end{align}
By the fundamental theorem of Calculus, 
\begin{align}\label{esti_2}
 %\Psi\big(\Phi^{\epsilon,+}_{j,k,p}\big) + \Psi\big(\Phi^{\epsilon,-}_{j,k,p}\big) =
\Psi(x_j-a^{+}_{h,j}) + \Psi(x_j-a^{-}_{h,j})  = 2 \Psi(x_j) - (a^{+}_{h,j} + a^{-}_{h,j})\cdot D \Psi(x_j) +E_1
\end{align}
% where we denote
%\begin{align*}
%Err_1:=  & - \alpha\int_0^1 a^{+}_{h,j}\Big[(x_j- t a^{+}_{h,j}) \big(1+|x_j-t a^{+}_{h,j}|^2\big)^{\frac{\alpha-2}{2}}  -  x_j \big(1+|x_j|^2\big)^{\frac{\alpha-2}{2}}  \Big] dt \notag \\
% & \hspace{1cm} - \alpha\int_0^1 a^{-}_{h,j}\Big[(x_j- t a^{-}_{h,j}) \big(1+|x_j-t a^{-}_{h,j}|^2\big)^{\frac{\alpha-2}{2}}  -  x_j \big(1+|x_j|^2\big)^{\frac{\alpha-2}{2}} \Big] dt.
%\end{align*} 
where $a^{+}_{h,j} + a^{-}_{h,j} = 2 h\,\big( D_p H \big(x_j, Du^{\epsilon}_{\rho,h}(t_k,x_j)\big) + B_r^{\sigma}\big)$ and 
\begin{align*}
E_1=- \int_{0}^1 \Big[a^{+}_{h,j} \cdot \big(D\Psi(x_j-t a^{+}_{h,j}) - D \Psi(x_j)\big)  + a^{-}_{h,j} \cdot \big(D\Psi(x_j-t a^{-}_{h,j}) - D \Psi(x_j)\big)\Big] dt. 
\end{align*}
By Lemma \ref{semiconcavity_eps}~(a) and 
 \ref{H1}, we find that $\|D_p H(\cdot, Du_{\rho,h}^{\epsilon})\|_0\leq C_R$ with $R=( L_{L} + L_{F} )T + L_{G}+1$,
% As $H_p$ is continuous (see, \ref{H1}) we get a constant $K>0$ such that
% \begin{align}\label{esti_H_p}
%     \|H_p(\cdot, Du_{\rho,h}^{\epsilon})\|_0\leq K.
% \end{align}
and then that 
\begin{align*}
|E_1| 
& \leq  \|D^2 \Psi \|_{0} ( |a^{+}_{h,j}|^2 + |a^{-}_{h,j}|^2 ) 
%\leq   C \big(h\, C_R+ h |B_r^{\sigma}| + \sqrt{h}|\sigma_r^{p}|\big)^2 \\  & 
\leq  4 \|D^2 \Psi \|_{0} \big( h^2 (C_R^2+|B_r^{\sigma}|^2) + h|\sigma_r^{p}|^2 \big).
\end{align*}
%for some constant $c_1>0$ depending only on $d$. 
To estimate the nonlocal term, we write
%we find a constant $C_1>0$ % is independent of $\rho, h$ and $\epsilon$
% Next we estimate the nonlocal term. 
%By Proposition \ref{prop:tail-control-function} and since $\int_{|z|<1} |z|^2 \nu(dz)<\infty$, we find a constant $C_1>0$ % is independent of $\rho, h$ and $\epsilon$
\begin{align*}%\label{esti_err_5}
\notag &  \int_{|z|>r} \Psi(x_j+z) \, \nu(dz)  =   \int_{|z|>1} \Psi(x_j+z) \nu(dz)\\
&\quad +  \int_{r<|z|<1} \Big\{\Psi(x_j) + z\cdot D\Psi(x_j) 
%\\ \notag &  \qquad \qquad \qquad \qquad \qquad \qquad   
 + \int_0^1  z \cdot \Big[ D \Psi(x_j+tz) - D \Psi(x_j)\Big]dt \Big\} \, \nu(dz)  \\ 
\notag & \leq \Big| \int_{|z|>1} \big(\Psi(x_j+z) - \Psi(x_j) \big) \nu(dz)\Big|  + \lambda_r \Psi(x_j) + B_r^{\sigma} \cdot D\Psi(x_j) \\ 
\notag &  \hspace*{6.5cm} + \|D^2\Psi\|_{0} \int_{r<|z|<1} |z|^2 \nu(dz) \\
& \leq \,  \lambda_r \Psi(x_j) + B_r^{\sigma} \cdot D\Psi(x_j) + E_2,
\end{align*}
where $E_2$ is finite and independent of $\rho, h,\epsilon$ by Proposition \ref{prop:tail-control-function} and $\int_{|z|<1} |z|^2 \nu(dz)<\infty$.
Going back to \eqref{esti_1} and using the above estimates then leads to
%\eqref{esti_1} along with \eqref{esti_2}, \eqref{esti_err_5} and as $\|D\Psi\|_{0}<\infty$, we get that 
\begin{align*}
& \int_{\R^d} \Psi(x) d\,m^{\epsilon}_{\rho,h}(t_{k+1}) \\
& \leq  \sum_{j \in \Z^d} m_{j,k} \bigg[ \frac{e^{-\lambda_r h}}{2d} \sum_{p=1}^d\Big( \notag 2 \Psi(x_j) - 
2 h\,\big[ D_p H \big(x_j, Du^{\epsilon}_{\rho,h}(t_k,x_j)\big) + B_r^{\sigma}\big] \cdot D\Psi(x_j) + |E_1| \Big) \\
&  \hspace{2cm} + \frac{1-e^{-\lambda_r h}}{\lambda_r} \Big(\lambda_r \Psi(x_j)  + B_r^{\sigma} \cdot  D\Psi(x_j) + E_2 \Big) \bigg] + C   \rho^2  \\
& \leq   \sum_{j\in \Z^d} m_{j,k}\, \Psi(x_j)  +C \Big( h^2 \lambda_r |B_r^{\sigma}| + h^2 |B_r^{\sigma}|^2+ h + \rho^{2}\Big), 
\end{align*}
where we used 
$|-he^{-\lambda_rh} +\frac{1-e^{-\lambda_r h}}{\lambda_r}|\leq \frac{3}{2} \lambda_r h^2$ and $\frac{1-e^{-\lambda_r h}}{\lambda_r}\leq h$ to get the last inequality. 

With $A_{k+1} = \int_{\R^d} \Psi(x) d\,m^{\epsilon}_{\rho,h}(t_{k+1})$, the above estimate becomes
%then from previous estimate we have 
%\begin{align*}
$A_{k+1} \leq  A_k + E$ where $E=C (\lambda_r h^2 |B_r^{\sigma}| + h^2 |B_r^{\sigma}|^2 + h + \rho^{2})$. 
%\end{align*} 
 By iteration, $ |B^{\sigma}_r|^2\leq \lambda_r |B_r^{\sigma}| \leq C r^{1-2\sigma}$ (by \ref{A0}, \ref{A2}), and $k\leq N \leq \frac Ch$, we find that
\begin{align} \label{esti_tight_m_rho-h}
 A_{k+1} & \leq  \, A_0 +  (k+1) E 
 %\\ 
 %& \leq  A_0  + \frac{K}{h}\Big(h^2 r^{1-2\sigma}  + h + \rho^{2}\Big)  
\leq  A_0 + C\Big( h r^{1-2\sigma} +1 +\frac{\rho^{2}}{h} \Big) .
\end{align}
By assumption $\frac{\rho^{2}}{h},h r^{1-2\sigma}=\mathcal{O}(1)$, and by Proposition \ref{prop:tail-control-function}, $A_0 = \int_{\R^d} \Psi(x) d\,m_0<\infty$. Therefore
\begin{align*}
\int_{\R^d} \Psi(x) d\,m^{\epsilon}_{\rho,h}(t_k) \leq C \qquad \mbox{for} \qquad k=0,1,\dots,N ,
\end{align*} 
for some constant $C>0$ independent of $\rho,h,\epsilon, \mu$, and hence by \eqref{exten_disc_measure} the result follows for $t\in[0,T]$.
\end{proof}

\begin{thm}[Equicontinuity in time]\label{thm:time_equicont_m}
    Assume \ref{A0}, \ref{A2}, \ref{L1}--\ref{L2},  \ref{F2}, \ref{H1}, \ref{M1}, $\mu \in  C ( [0,T], P ( \rd ))$, and  $m^{\epsilon}_{\rho,h}[\mu]$ is defined by \eqref{exten_disc_measure}.  Let  $\frac{\rho^{2}}{h},\frac h{r^{\sigma}}=\mathcal{O}(1)$ if $\sigma\in (0,1)$, or  $\frac{\rho^{2}}{h},h r^{1-2\sigma}=\mathcal{O}(1)$ if $\sigma\in (1,2)$. Then there exists a constant $C_0>0$, independent of $\rho, h, \epsilon$ and $\mu$, such that for any $t_1,t_2 \in [0,T]$,
\begin{align*}
d_0 (m^\epsilon_{\rho,h}[\mu](t_1),m^\epsilon_{\rho,h}[\mu](t_2)) \leq C_0 \sqrt{|t_1-t_2|}. 
\end{align*}
\end{thm}
\begin{proof}
We start by the case $\sigma>1$. For $\delta>0$, let $\phi_{\delta}:= \phi*\rho_{\delta}$ for $\rho_{\delta}$ defined just before Lemma \ref{regularized_Lipschitz}. With $m^\epsilon_{\rho,h}=m^\epsilon_{\rho,h}[\mu]$ we first note that
\begin{align} 
  & d_0  (m^\epsilon_{\rho,h}(t_1), m^\epsilon_{\rho,h}(t_2))=  \sup_{\phi\in \Lip_{1,1}}\int_{\rd} \phi(x) (m^\epsilon_{\rho,h}(t_1)-m^\epsilon_{\rho,h}(t_2))dx \nonumber\\
 &=  \sup_{\phi\in \Lip_{1,1}}\Big\{\int_{\rd}(\phi-\phi_\delta)(m^\epsilon_{\rho,h}(t_1)-m^\epsilon_{\rho,h}(t_2))dx + \int_{\rd} \phi_\delta \,  (m^\epsilon_{\rho,h}(t_1)-m^\epsilon_{\rho,h}(t_2))dx\Big\} \nonumber\\
 & \leq \, 2 \delta \| D\phi \|_0  + \sup_{\phi\in \Lip_{1,1}}\int_{\rd} \phi_\delta \, (m^\epsilon_{\rho,h}(t_1)-m^\epsilon_{\rho,h}(t_2))dx, \label{regular_KR_dist}
\end{align}
where Lemma \ref{regularized_Lipschitz} was used to estimate the $\phi-\phi_\delta$ term and $\int m^\epsilon_{\rho,h}dx=1$. Since $m^\epsilon_{\rho,h}$ and $\int_{\R^d} \phi_{\delta}(x) m^\epsilon_{\rho,h}(t,x)dx$ are affine on each interval $[t_k,t_{k+1}]$,  $\int_{\R^d} \phi_{\delta}(x) \, m^\epsilon_{\rho,h}(\cdot,x)dx\in W^{1,\infty}[0,T]$ and 
\begin{align*}
\Big\|\frac{d}{dt} \int_{\R^d} \phi_{\delta}(x) \, m^\epsilon_{\rho,h}(\cdot,x) dx\Big\|_{0} \leq \sup_k |I_k|.
\end{align*}
where $I_k= \int_{\R^d} \phi_{\delta}(x) \, \frac{m^\epsilon_{\rho,h}(t_{k+1},x)-m^\epsilon_{\rho,h}(t_k,x)}{h}dx$.
It follows that
\begin{align}\label{diff-m}
 &\int_{\R^d} \phi_{\delta} \, (m^{\epsilon}_{\rho,h}(t_1,x)- m^{\epsilon}_{\rho,h}(t_2,x))dx
 %=  \int^{t_2}_{t_1}\frac{d}{ds}\Big[\int_{\R^d}\phi_{\delta}(x)m^{\epsilon}_{\rho,h}(s)dx \Big]ds 
  \leq  |t_1-t_2| \sup_k |I_k|. 
\end{align}

Let us estimate $I_k$.
By \eqref{exten_disc_measure}, \eqref{Fokker-Planck_discretized}, \eqref{fokker-planck-general-discretized}, the  midpoint quadrature approximation error bound, and the linear/multi-linear interpolation error bound, we have
\begin{align*}
I_k 
=&\frac{1}{h} \sum_i  \frac{1}{\rho^d}  \int_{E_i} \phi_{\delta}(x)\, dx[m_{i,k+1} - m_{i,k}]\\
%=& \frac{1}{h} \sum_i \frac{1}{\rho^d} \int_{E_i} \phi_{\delta}(x)\Big[ \sum_j m_{j,k}\Big\{ \frac{e^{-\lambda_r h}}{2d} \sum^d_{p=1} \beta_i (\Phi^+_{j,k}) + \beta_i (\Phi^{\epsilon,-}_{j,k,p}) \\
% & \hspace{4cm}+ \frac{1-e^{-\lambda_r h}}{\lambda_r}\int_{| z | > r} \beta_i (x_j + z) \nu(dz)\Big\} - m_{i,k} \Big] dx \\
% =&\frac{1}{h\rho^d} \sum_{j,i}  \bigg(\int_{E_i} \phi_{\delta}(x) dx \bigg) \frac{m_{j,k}}{\rho^d}   \, \mathbf{B}_{\rho,h,r} [ H_{p} ( \cdot, Du_{\rho,h}^{\epsilon} ) ] ( i,j,k ) 
%  - \sum_j \bigg(\int_{E_j} \phi_{\delta}(x) dx \bigg) \frac{m_{j,k}}{\rho^d} \\
  =& \frac{1}{h\rho^d} \sum_{j,i}  \Big(\int_{E_i} \phi_{\delta}(x) dx \Big)\Big[ m_{j,k}   \, \mathbf{B}_{\rho,h,r} [ H_{p} ( \cdot, Du_{\rho,h}^{\epsilon} ) ] ( i,j,k ) 
 - m_{i,k}\, \delta_{i,j}\Big]  \\
=& \frac{1}{h} \sum_{j} m_{j,k} \Big[ \sum_i \phi_{\delta}(x_i)\mathbf{B}_{\rho,h,r} [ H_{p} ( \cdot, Du_{\rho,h}^{\epsilon} ) ] ( i,j,k ) - \phi_{\delta}(x_j) + C\|D^2 \phi_{\delta}\|_{0}\rho^2 \Big]      \\
=& \frac{1}{h} \sum _j m_{j,k}\Big[ \frac{e^{-\lambda_r h}}{2d} \Big( \sum^d_{p=1}\phi_{\delta}(\Phi^{\epsilon,+}_{j,k,p}) + \phi_{\delta} (\Phi^{\epsilon,-}_{j,k,p}) -2\phi_{\delta}(x_j)\Big) \\
 & \hspace{2cm}+ \frac{1-e^{-\lambda_r h}}{\lambda_r}\int_{| z | > r} \big(\phi_{\delta}(x_j + z)-\phi_{\delta}(x_j) \big)\nu(dz) + C\|D^2 \phi_{\delta}\|_{0}\rho^2 \Big] .
\end{align*}
Since $\Phi^{\epsilon,\pm}_{j,k,p}= x_j + a^{\pm}_{h,j}$ by~\eqref{charectatistic_dif}, a 2nd order Taylor's expansion gives us
\begin{align*}
\big|I_k   \big|
\leq  & \frac{1}{h}\sum_j m_{j,k}\bigg[e^{-\lambda_r h}\Big( (-h D_p H \big(x_j, Du^{\epsilon}_{\rho,h}[\mu](t_k,x_j)\big)  -hB_r^{\sigma})\cdot D\phi_{\delta}(x_j)  \\
 & \hspace*{1cm} +\frac{\|D^2\phi_{\delta}\|_{0}}{2d}\sum^d_{p=1}  \big(|a^{+}_{h,j}|^ 2+|a^{-}_{h,j}|^2\big)    
+ \frac{1-e^{-\lambda_r h}}{\lambda_r}\Big(B^\sigma_r \cdot D\phi_{\delta}(x_j)  \\
 & \hspace*{1cm}+ \|D^2\phi_{\delta}\|_{0}\int_{r< | z | <1}|z|^2 \nu(dz)  + 2\|\phi_{\delta}\|_{0} \int_{|z|>1} \nu(dz) + C\|D^2 \phi_{\delta}\|_{0}\rho^2 \Big) \bigg] \\
\leq  &\,  \frac{1}{h} \bigg[\Big(h  \| D_pH(\cdot,Du_{\rho,h}^{\epsilon}) \|_0  + h^2 \lambda_r | B^\sigma_r|\Big)\|D\phi_{\delta}\|_0 + c_3 h {\|  \phi_{\delta} \|}_0 
			\\ & \hspace{0.5cm} + c_1 \Big(h^2 {\| D_pH(\cdot,Du_{\rho,h}^{\epsilon}) \|}^2 + h^2 {| B^\sigma_r | }^2 + h{|\sigma_r|}^2 + h + \rho^2\Big ) {\| D^2 \phi_{\delta} \|}_0  \bigg]  \sum_j m_{j,k}.
\end{align*}
The above inequality follows since $(\frac{1-e^{-\lambda_r h}}{\lambda_r} -h e^{-h\lambda_r}) \leq h^2\lambda_r$ 
%(used for the $\lambda_r | B^\sigma_r|$-terms), 
 (used for the $B_r^{\sigma}\cdot D\phi_{\delta}$-terms), and $\int_{r<| z | <1} {| z | }^2  \nu(dz) + \int_{|z|>1} \nu(dz) \leq C$ independently of $r$  by \ref{A0} and \ref{A2}.  
By Lemma \ref{semiconcavity_eps}~(a) and 
 \ref{H1}, $\|D_pH(\cdot, Du_{\rho,h}^{\epsilon})\|_0\leq C_R$ with $R=( L_{L} + L_{F} )T + L_{G}+1$. Since $\sum m_{j,k} =1$, $\phi\in \textup{Lip}_{1,1}$, $\|D^2\phi_{\delta}\|_{0} \leq \frac{\|D\phi\|_{0}}{\delta}$, and $ |B^{\sigma}_r|^2\leq \lambda_r |B_r^{\sigma}| \leq K r^{1-2\sigma}$ (by \ref{A0}, \ref{A2}), we get that 
\begin{align*} %\label{esti_equicont_m_rho-h}
|I_k| 
% &\leq   (1+h  r ^{1-2\sigma} ) \| D\phi \|_{0} +  \|\phi\|_{0}   + ( h + h  r ^{1-2\sigma}   + \rho^2/h )   \|D^2 \phi_{\delta}\|_0  \\ 
%\notag &
\leq  C(1+h r ^{1-2 \sigma} )  + C\big(1 +   h +h r ^{1-2 \sigma}  + \frac{\rho^2}{h}\big) \frac{1}{\delta}.
%\\\leq & c_1(t_1 - t_2 ) + c_2\delta + c_3 (t_1 - t_2 )(1+hr^{1-2\delta} + \frac{\rho}{h} ) + c_n(t_1 - t_2 )\frac{1}{\delta} (1+r^{2-\delta} + h+ hr^{2(1-\sigma)} ).
\end{align*}

To conclude the proof in the case $\sigma>1$, we go back to \eqref{regular_KR_dist} and \eqref{diff-m}. In view of the above estimate on $I_k$ and the assumption that $\frac{\rho^2}{h},h r^{1-2\sigma}=\mathcal{O}(1)$, we find that
\begin{align*}% \label{esti_equicont_m_rho-h}
d_0(m^{\epsilon}_{\rho,h}(t_1),m^{\epsilon}_{\rho,h}(t_2))
% &\leq  2\delta + C|t_1 - t_2 | \Big((1+h r ^{1-2 \sigma} )  + \big(\be 1 +  h +h r ^{1-2 \sigma}  + \frac{\rho^2}{h}\big) \frac{1}{\delta}\Big)\\
&\leq 2\delta + C |t_1-t_2|\Big(1+\frac{1}{\delta}\Big). 
\end{align*}
Finally taking $ \delta = \sqrt{|t_1-t_2|}$ we get 
$d_0(m^{\epsilon}_{\rho,h}(t_1),m^{\epsilon}_{\rho,h}(t_2)) \leq C\sqrt{|t_1-t_2|}.$ \medskip

When $\sigma<1$, we find that $|B_r^{\sigma}|\leq C$ and hence that
\begin{align*} 
|I_k| \leq  C(1+h r ^{-\sigma} )  + C\big( 1 + h   +h r ^{-\sigma} + \frac{\rho^2}{h}\big) \frac{1}{\delta}.
\end{align*}
By assumption $h r^{-\sigma}+ \frac{\rho^2}{h}= \mathcal{O}(1)$, so again we find that
\begin{align*}
d_0(m^{\epsilon}_{\rho,h}(t_1),m^{\epsilon}_{\rho,h}(t_2)) \leq 2\delta + C |t_1-t_2|\Big(1+\frac{1}{\delta}\Big), 
\end{align*}
and can conclude as before.
% Again by taking $\delta= \sqrt{|t_1-t_2|}$ we get $d_0(m^{\epsilon}_{\rho,h}(t_1),m^{\epsilon}_{\rho,h}(t_2)) \leq C\sqrt{|t_1-t_2|}. $
\end{proof} 

We also need a $L^1$-stability result for $m^{\epsilon}_{\rho,h}[\mu]$ with respect to variations in $\mu$.

\begin{lemma}[$L^{1}$-stability] \label{L1_stability_fp}
Assume \ref{A0}, \ref{H1}, and $m^{\epsilon}_{\rho,h}[\mu]$ is defined by \eqref{exten_disc_measure}.
    Then for $\mu_1,\mu_2 \in C([0,T],P(\rd))$,
    \begin{align*}
      &  \sup_{t \in  [ 0,T ]} \| m^{\epsilon}_{\rho,h}[\mu_1] ( t, \cdot ) - m^{\epsilon}_{\rho,h}[\mu_2] ( t,\cdot )   \|_{L^{1} ( \R^d )}  \\
        &\leq \frac{cKT}{\rho} e^{- h \lambda_{r}} \big\| D_p H(\cdot, Du^{\epsilon}_{\rho,h}[\mu_1]) - D_p H(\cdot, Du^{\epsilon}_{\rho,h}[\mu_2])  \big\|_{0} .
    \end{align*}
\end{lemma}

\begin{proof}
    Let  $\alpha=D_p H(\cdot, Du^{\epsilon}_{\rho,h}[\mu_1])$, $\tilde \alpha=D_p H(\cdot, Du^{\epsilon}_{\rho,h}[\mu_2])$, 
    %$m_{i,k}^j=m_{i,k}[\mu_j]$ for $j=\{1,2\}$. 
    $m_{j,k}=m_{j,k}[\mu_1]$, and $\tilde m_{j,k}=m_{j,k}[\mu_2]$.
    By \eqref{fokker-planck-general-discretized} and Lemma \ref{mass_conservation}, 
    $\mathbf{B}_{\rho,h,r} [ \alpha ] ( i,j,k ) \geq 0$ and  $m_{j,k} \geq 0 
    $, so that
    \begin{align*}
        & \sum_{i} \big|m_{i,k+1} - \tilde{m}_{i,k+1} \big|  =  
    \sum_{i} \big| \sum_{j} ( m_{j,k}  \, \mathbf{B}_{\rho,h,r} [ \alpha ] ( i,j,k )  - \tilde m_{j,k}  \, \mathbf{B}_{\rho,h,r} [ \tilde{\alpha}  ] ( i,j,k )    ) \big| \\
      & \leq  \sum_{i} \sum_{j} \Big(  m_{j,k} \big| \mathbf{B}_{\rho,h,r} [ \alpha ] ( i,j,k ) 
    - \mathbf{B}_{\rho,h,r} [ \tilde{\alpha}  ] ( i,j,k )    \big|  +  \big| m_{j,k} - \tilde{m}_{j,k}  \big| \mathbf{B}_{\rho,h,r} [ \tilde{\alpha}  ] ( i,j,k ) \Big)  .
    \end{align*}
    Since $\sum_{i} \mathbf{B}_{\rho,h,r} [ \tilde{\alpha}  ] ( i,j,k ) = 1$ (follows from $\sum_i\beta_i=1$ and \eqref{fokker-planck-general-discretized}), 
    \begin{align*}
  \sum_{i} \sum_{j}  \big| m_{j,k} - \tilde{m}_{j,k}  \big| \mathbf{B}_{\rho,h,r} [ \tilde{\alpha}  ] ( i,j,k ) = \sum_{j} \big| m_{j,k} - \tilde{m}_{j,k} \big|. 
    \end{align*}
   Moreover, since only a finite number $K_{d}$ of $\beta_{i}$'s are non-zero at any given point,
    %    at the points $\Phi_{j,k,p}^{\pm} [\mu_1]$ %    and $\Phi_{j,k,p}^{ \pm} [\mu_2]$ (defined % in \eqref{defn_characteristics}), 
    % $depending only on the dimension $d$,
    $\beta_{i}$ is Lipschitz with constant $\frac{c}{\rho}$, and $\sum_{j} m_{j,k}  = 1$ by Lemma \ref{mass_conservation}, by the definitions of $\mathbf{B}_{\rho,h,r}$ \eqref{fokker-planck-general-discretized} and $\Phi_{j,k,p}^{ \pm}$  \eqref{defn_characteristics},
    \begin{align*}
       & \sum_{i} \sum_{j}  m_{j,k} \big| \mathbf{B}_{\rho,h,r} [ \alpha ] ( i,j,k ) 
    - \mathbf{B}_{\rho,h,r} [ \tilde{\alpha}  ] ( i,j,k )    \big|  \\
    &\leq  \sum_{j} m_{j,k} \frac{e^{-h \lambda_{r}}}{2d} \sum_{p=1}^{d} \sum_i   \, \big| \beta_{i} ( \Phi_{j,k,p}^{+} [\mu_1] ) - \beta_{i} ( \Phi_{j,k,p}^{+} [\mu_2] )  \\
   & \hspace*{2cm}+ \beta_{i} ( \Phi_{j,k,p}^{-} [\mu_1] ) - \beta_{i} (  \Phi_{j,k,p}^{-} [\mu_2] )  \big|  \leq K_{d} \frac{c h  e^{-h \lambda_{r}}  }{\rho}  \| \alpha - \tilde{\alpha}   \|_{0} .
    \end{align*}
  An iteration then shows that 
     \begin{align*}
        \sum_{i} \big|  m_{i,k+1} - \tilde{m}_{i,k+1} \big| \leq    \sum_{i} \big| m_{i,0} - \tilde{m}_{i,0}  \big| + \frac{cK_{d}T}{\rho} e^{- h \lambda_{r}} \| \alpha - \tilde{\alpha}   \|_{0} .
    \end{align*}
Since $ m_{i,0} = \tilde{m}_{i,0}=\int_{E_i}m_0 \, dx$, the result follows by interpolation.
\end{proof}

We end this section by  a uniform $L^p$-bound on $m^{\epsilon}_{\rho,h}$ in dimension $d=1$. 
\begin{thm}[$L^p$ bounds]\label{thm:Lp_esti}
 Assume $d=1$, \ref{A0}, \ref{A2}, \ref{L1}, \ref{L3}, \ref{F3}, \ref{H2}, \ref{M11},  $\mu \in  C ( [0,T], P ( \rd ))$, and  $m^{\epsilon}_{\rho,h}[\mu]$ be defined by \eqref{exten_disc_measure}. Then $m_{\rho,h}^{\epsilon}[\mu] \in L^p(\R)$ and there exist a constant $K>0$ independent of $\epsilon,h, \rho$ and $\mu$ such that
\begin{align*}
\|m^{\epsilon}_{\rho, h}[\mu](\cdot, t)\|_{L^{p}(\R)} \leq e^{KT} \|m_{0}\|_{L^{p}(\R)}. 
\end{align*} 
\end{thm}\medskip

To prove the theorem we need few technical lemmas.

\begin{lemma} \label{lem:dimone-semicon}
Assume $d=1$, \ref{A0}, \ref{A2}, \ref{L1}, \ref{L3}, \ref{F3}, and \ref{H2}. There exists a constant $c_0>0$ independent of $\rho,h,\epsilon,\mu$ such that
 $$\Big(D_pH\big(x_j,Du^{\epsilon}_{\rho,h}(t_k,x_j)\big)- D_p H \big(x_i, Du^{\epsilon}_{\rho,h}(t_k,x_i)\big)\Big) (x_j-x_i) \leq c_0 |x_j-x_i|^2. $$
\end{lemma} 
\begin{proof}
By \ref{L1} and \ref{H2} for $R = (  ( L_{F} + L_{L} )T + L_G)  +1$ we have 
 \begin{align*}
&  \Big(D_pH\big(x_j,Du^{\epsilon}_{\rho,h}(t_k,x_j)\big)- D_p H \big(x_i, Du^{\epsilon}_{\rho,h}(t_k,x_i)\big)\Big) (x_j-x_i) \\
& = (x_j-x_i)\int_0^1 \frac{d}{dt} \Big( D_pH\big(x_j,t \, Du^{\epsilon}_{\rho,h}(t_k,x_j) + (1-t)Du^{\epsilon}_{\rho,h}(t_k,x_i)\big) \Big) dt  \\
& \quad \ + (x_j-x_i) \Big(D_pH\big(x_j,Du^{\epsilon}_{\rho,h}(t_k,x_i)\big)- D_p H \big(x_i, Du^{\epsilon}_{\rho,h}(t_k,x_i)\big)\Big) \\
&   =  (x_j-x_i)\int_0^1 D_{pp} H\Big(x_j,t \, Du^{\epsilon}_{\rho,h}(t_k,x_j)  \\ 
& \hspace*{4cm}  + (1-t)Du^{\epsilon}_{\rho,h}(t_k,x_i)\Big) \big(Du^{\epsilon}_{\rho,h}(t_k,x_j) - Du^{\epsilon}_{\rho,h}(t_k,x_i)\big)  dt \\
& \quad \ + (x_j-x_i) \Big(D_pH\big(x_j,Du^{\epsilon}_{\rho,h}(t_k,x_i)\big)- D_p H \big(x_i, Du^{\epsilon}_{\rho,h}(t_k,x_i)\big)\Big)  \\
& \leq C_R \, c_2 |x_j-x_i|^2 +C_R |x_j-x_i|^2,
 \end{align*}
 where the last inequality follows from convexity of $H$ (since $L$ is convex by \ref{L1}), semiconcavity of $u^{\epsilon}_{\rho,h}$ in Lemma \ref{semiconcavity_eps} (c), and regularity of $H$ in \ref{H2}.
 %, where $C_R$ is given in \ref{H2}. 
% This completes the proof.
\end{proof}
 
 \begin{lemma} \label{lem:bound_char}
Assume $d=1$, \ref{A0}, \ref{A2}, \ref{L1}, \ref{L3}, \ref{F3}, \ref{H2},  $\mu \in  C ( [0,T], P ( \rd ))$, and  let $\Phi^{\epsilon,\pm}_{j,k}[\mu]$ be defined in \eqref{defn_characteristics}. There exist a constant $K_0>0$ independent of $\epsilon, \rho, h,\mu$, such that for all $i \in \Z$ and $k=\mathcal{N}_h$, 
 \begin{align*}
 \max\Big\{ \sum_{j\in \Z}\beta_i(\Phi^{\epsilon,+}_{j,k})[\mu],\sum_{j\in \Z} \beta_i(\Phi^{\epsilon,-}_{j,k})[\mu]  \Big\} \leq 1+K_0 h. 
 \end{align*}
 \end{lemma}
The proof of this result is similar to the proof of \cite[Lemma 3.8]{carliniSilva2014semi1st} -- 
a slightly expanded proof is given in Appendix \ref{app:pf}. A similar result holds for the integral-term:
\begin{lemma} \label{lem:bound_jump_char}
Assume $d=1$.  Then we have 
\begin{align*}
\frac{1}{\lambda_r} \sum_{j\in \Z} \int_{|z|>r} \beta_i(x_j+z) \nu(dz) =1. 
\end{align*}
\end{lemma}

\begin{proof}
By \eqref{lambda_r} and properties of the basis functions $\beta_j$ we have
\begin{align*}
    \frac{1}{\lambda_r} \sum_{j\in \Z} \int_{|z|>r} \beta_i(x_j+z) \nu(dz) = 
    \frac{1}{\lambda_r} \int_{|z|>r} \sum_{j\in \Z} \beta_{i-j}(z) \nu(dz) =\frac{1}{\lambda_r} \int_{|z|>r} \nu(dz) =1.  \qquad \quad \qedhere 
\end{align*}
\end{proof} 

\medskip

% \BC
% \begin{proof}
% From the definition of (hat) basis function we have that  $\beta_i(x_j+z)\neq 0$ only when $|x_i-(x_j+z)|\leq \rho$ and
% \begin{align*}
% \beta_i(x_j+z)=\begin{cases} \frac{x_{i-j+1} -z}{\rho} \quad \mbox{for } \, z\in (x_{i-j}, x_{i-j+1}), \\
% \frac{z- x_{i-j-1}}{\rho} \quad \mbox{for } \, z\in (x_{i-j-1}, x_{i-j}), \\
% 0 \quad \mbox{otherwise}.
% \end{cases}
% \end{align*}
% Therefore, we get 
% \begin{align*}
% \sum_{j\in \Z}\int_{|z|>r} \beta_i(x_j+z) \nu(dz) =& \sum_{j\in \Z} \int_{z \in (x_{i-j-1}, x_{i-j})\cap (-r,r)^c} \frac{z- x_{i-j-1}}{\rho} \nu(dz) \\
% & \hspace{1.5cm}+ \sum_{j\in \Z}\int_{z \in (x_{i-j}, x_{i-j+1})\cap (-r,r)^c} \frac{x_{i-j+1}-z}{\rho} \nu(dz)\\
%   =& \sum_{j\in \Z} \int_{z \in (x_{i-j}, x_{i-j+1})\cap (-r,r)^c} \Big(\frac{z- x_{i-j}}{\rho} + \frac{x_{i-j+1}-z}{\rho} \Big)\nu(dz) \\
% = & \sum_{j \in \Z} \int_{z \in (x_{i-j}, x_{i-j+1})\cap (-r,r)^c} \nu(dz) = \int_{|z|>r} \nu(dz). 
% \end{align*}
% This completes the proof.
% \end{proof}
% 

\begin{proof}[Proof of Theorem \ref{thm:Lp_esti}]
By definition of $m^{\epsilon}_{\rho, h}$ in \eqref{exten_disc_measure} and the scheme \eqref{Fokker-Planck_discretized},
\begin{align*}
&\int_{\R} ( m^{\epsilon}_{\rho, h}(x, t_{k+1}))^p dx 
  =  \int_{\R} \Big( \frac{1}{\rho} \sum_{i} m_{i,k+1} \mathbbm{1}_{E_i}(x)\Big)^p dx 
\\
& =  \frac{1}{\rho^{p-1}} \sum_{i\in \Z} (m_{i,k+1})^p  = \frac{1}{\rho^{p-1}} \sum_{i} \Big( \sum_{j} m_{j,k} \, \mathbf{B}_{\rho,h,r} ( i,j,k )\Big)^p, 
\end{align*} 
where $\mathbf{B}_{\rho,h,r} =\mathbf{B}_{\rho,h,r} [ H_{p} ( \cdot, Du_{\rho,h}^{\epsilon}[\mu] ) ]$ is defined in \eqref{fokker-planck-general-discretized}. 
%Since $\sum_{i} \beta_i =1$, we have $\sum_{i}\mathbf{B}_{\rho,h,r} ( i,j,k )=1$.
By  Jensen's inequality we have 
\begin{align*}
\sum_{i\in \Z} \Big( \sum_{j} m_{j,k}\, \mathbf{B}_{\rho,h,r} ( i,j,k )\Big)^p
\leq \sum_{i\in \Z} \Big(\sum_{p\in \Z}\mathbf{B}_{\rho,h,r} ( i,p,k )\Big)^{p-1} \Big(\sum_{j} \big(m_{j,k}\big)^p \, \mathbf{B}_{\rho,h,r} ( i,j,k ) \Big),
\end{align*}
and by Lemma \ref{lem:bound_char} and  \ref{lem:bound_jump_char},
$$\sum_{p\in \Z}\mathbf{B}_{\rho,h,r} ( i,p,k ) \leq 1+ K_0h,$$
where $K_0$ is independent of  $i, \rho, h, \epsilon$ and $\mu$. Since $\sum_{i} \mathbf{B}_{\rho,h,r} ( i,p,k ) =1$ (follows from $\sum_{i} \beta_i =1$), we find that
\begin{align*}
\sum_{i\in \Z} (m_{i,k+1})^p  
%\leq (1+K_0h)^{p-1} \sum_{i\in \Z}\sum_{j} \big(m_{j,k}\big)^p \, \mathbf{B}_{\rho,h,r} ( i,j,k ) \\
 &\leq (1+K_0h)^{p-1} \sum_{j}\big(m_{j,k}\big)^p \sum_{i} \mathbf{B}_{\rho,h,r} ( i,j,k ) \\
 &\leq \rho^{p-1} \|m^{\epsilon}_{\rho,h}(t_k,\cdot)\|^p_{L^p(\R)} (1+ K_0h)^{p-1}.
 \end{align*} 
By iteration and $\|m^{\epsilon}_{\rho, h}(\cdot, t_0)\|_{L^{p}} = \|m_0\|_{L^{p}}$, $\|m^{\epsilon}_{\rho,h}(t_{k+1},\cdot)\|_{L^p} \leq e^{K_0T(1-\frac 1p)} \|m_0\|_{L^p}$, and the result follows for $p\in[1,\infty)$.

The proof of $p=\infty$ is simpler, and in view of Lemma \ref{lem:bound_char} and \ref{lem:bound_jump_char}, the proof follows as in  \cite{carlini2015semi} for 2nd order case. 
\end{proof}

\section{Proof of convergence -- Theorem  \ref{thm:convergence_MFG} and \ref{thm:convergence_MFG-nondeg}} \label{sec:proof-main}

The main structure of the proofs are similar, so we present the proofs together.  We proceed by several steps. \smallskip

\noindent \textbf{Step 1.} (Compactness of $m^{\epsilon_n}_{\rho_n,h_n}$) \ 
 In view of Theorem \ref{thm:thightness_m} and  \ref{thm:time_equicont_m},  $m^{\epsilon}_{\rho,h}$ is precompact in $C([0,T], P(\rd))$ by the Prokhorov and Arzel\`a-Ascoli Theorem.   Hence there exist a subsequence $\{m^{\epsilon_n}_{\rho_n,h_n}\}$ and $m$ in $ C([0,T], P(\rd))$ such that 
\begin{align*}
m^{\epsilon_n}_{\rho_n,h_n} \rightarrow m \quad \mbox{in} \quad C([0,T],P(\rd)).
\end{align*} 
This proves Theorem \ref{thm:convergence_MFG-nondeg} (a) (ii) and the first part of Theorem \ref{thm:convergence_MFG}~(a) (ii).\smallskip

If \ref{M11} holds with $p=\infty$, then Theorem \ref{thm:Lp_esti} and Helly's weak $*$ compactness theorem imply that $\{m^{\epsilon}_{\rho,h}\}$ is weak $*$ precompact  in $L^{\infty}([0,T]\times \R)$ and there is a subsequence $\{m^{\epsilon_n}_{\rho_n,h_n}\}$ and function $m$  such that $m^{\epsilon_n}_{\rho_n,h_n} \overset{\ast}{\rightharpoonup} m$ in $L^{\infty}([0,T]\times \R)$. If \ref{M11} holds with  $p \in (1,\infty)$, then $\{m^{\epsilon}_{\rho,h}\}$ is equiintegrable in $[0,T]\times \R$ by Theorem 
\ref{thm:thightness_m} and \ref{thm:Lp_esti} and  de la Vall\'ee Poussin's theorem. By Dunford-Pettis' theorem, it is then weakly precompact in $L^1([0,T]\times \R)$ and there exists a subsequence $\{m^{\epsilon_n}_{\rho_n,h_n}\}$ and function $m$ such that $m^{\epsilon_n}_{\rho_n,h_n} \rightharpoonup  m$ in $L^{1}([0,T]\times \R)$. 
The second part  of Theorem \ref{thm:convergence_MFG}~(a) (ii) follows. \medskip

\noindent \textbf{Step 2.} (Compactness and limit points for $u_{\rho_n,h_n}$) \ Part (i)  and limit points $u$ as viscosity solutions in part (iii)  of both Theorem \ref{thm:convergence_MFG} and \ref{thm:convergence_MFG-nondeg} follow from step 1 and Theorem  \ref{thm:HJB-convergence-smoothsol}~(i). \medskip

\noindent \textbf{Step 3.} (Consistency for $m^{\epsilon_n}_{\rho_n,h_n}$) \ Let   $(u,m)$ be a limit point of $\{(u^{\epsilon_n}_{\rho_n,h_n},m^{\epsilon_n}_{\rho_n,h_n})\}_n$. Then by step 2, $u$ is a viscosity solution of the HJB equation in \eqref{eqn:MFG}.
% By Theorem \ref{thm:HJB-convergence-smoothsol} it immediately follows that $u$ is a viscosity solution of the HJB equation in \eqref{eqn:MFG}.
We now   show that $m$ is a very weak solution of the FPK equation in \eqref{eqn:MFG} with $u$ as the input data, i.e. $m$ satisfies \eqref{dist_sol_FPK} 
for $t\in[0,T]$ and $\phi \in C_c^\infty(\rd)$.  In the rest of the proof we use  $\rho, h, r, \epsilon$ instead of $\rho_{n}, h_{n}, r_{n}, \epsilon_n$ to simplify. We also let 
$\widehat{\widehat{m}} = m^{\epsilon_n}_{\rho_n,h_n}$, $w=u_{\rho_n,h_n}^{\epsilon_n}[\widehat{\widehat{m}}]$, and take
$t_n=\big[\frac{t}{h_n}\big]h_n$. Then we note that
\begin{align*}
\int_{\rd}\phi(x)d\widehat{\widehat{m}}(t_n)(x)=\int_{\rd}\phi(0)dm_0(x) + \sum_{k=0}^{n-1}\int_{\rd}\phi(x)d[\widehat{\widehat{m}}(t_{k+1})-\widehat{\widehat{m}}(t_k)],
\end{align*}
so to prove \eqref{dist_sol_FPK}, we must estimate the sum on the right. 

By the midpoint approximation and \eqref{exten_disc_measure}, the scheme \eqref{Fokker-Planck_discretized}, %\eqref{fokker-planck-general-discretized}, 
  and \eqref{fokker-planck-general-discretized} combined with linear/multilinear interpolation, and finally midpoint approximation again, we find that
\begin{align*}
& \int_{\rd} \phi(x) d\widehat{\widehat{m}}(t_{k+1}) =  \frac{1}{\rho^d} \sum_{i\in \zd} m_{i,k+1}  \int_{E_i} \phi(x)dx = \sum_{i} m_{i,k+1} \phi(x_i) + \mathcal{O} (\rho^2) \\
& = \sum_{i} \phi(x_i) \sum_{j} m_{j,k} \, \mathbf{B}_{\rho,h,r} [ H_{p} ( \cdot, Dw ) ] ( i,j,k ) + \mathcal{O} (\rho^2)\\
& =  \sum_j m_{j,k} \Big( \frac{e^{-\lambda_r h}}{2d}\sum_{p=1}^d  [ \phi(\Phi_{j,k,p}^{\epsilon,+})+ \phi(\Phi_{j,k,p}^{\epsilon,-})  ]
+ \frac{1-e^{-\lambda_r h}}{\lambda_r} \int_{|z|>r} \phi(x_j+z) \nu(dz) \Big)
+ \mathcal{O} (\rho^2)\\
&=\sum_j \frac{m_{j,k}}{\rho^d} \int_{E_j} \Big( \frac{e^{-\lambda_r h}}{2d}\sum_{p=1}^d [\phi(\Phi_{k,p}^{\epsilon,+})(x)+ \phi(\Phi_{k,p}^{\epsilon,-})(x)] \\
&\hspace{2.7cm}  + \frac{1-e^{-\lambda_r h}}{\lambda_r} \int_{|z|>r} \phi(x+z) \nu(dz) \Big) dx +  \mathcal{O} (\rho^2) +E_\Phi+E_\nu,
\end{align*}
where $\Phi_{j,k,p}^{\epsilon,\pm}$ is defined in \eqref{defn_characteristics}, $\Phi^{\epsilon,\pm}_{k,p}(x) = x - h\,\big( 
  H_{p} ( x, D w (t_k,x) ) + B_r^{\sigma} \big)
 \pm \sqrt{hd} \sigma_r^p$, and $E_\Phi+E_\nu$ is the error of the last midpoint approximation. 
  % x- h H_p(x,Du^{\epsilon_n}_{\rho_n,h_n}(t_k,x) + B_r^{\sigma}) \pm \sqrt{hd}\sigma^p_r.$$ 
 Since $\phi$ is  smooth, $u_{\rho,h}$ uniformly Lipschitz (Lemma \ref{semiconcavity_eps} (a)), $\|D^2w\|_{0}\leq \frac{C \|Du_{\rho,h}\|_{0}}{\epsilon}$, and by assumption \ref{H2},
\begin{align*}
& \Big| \phi(\Phi^{\epsilon,\pm}_{j,k,p}) - \frac{1}{\rho^d} \int_{E_j} \phi(\Phi^{\epsilon,\pm}_{k,p})(x) dx \Big| \\
& \leq \frac{\|D\phi\|_0}{\rho^d} \int_{E_j} |x-x_j| dx + \frac{h\|D\phi\|_0}{\rho^d} \int_{E_j} \big|D_p H(x_j, Dw(t_k,x_j)) - D_p H(x, Dw(t_k,x))\big| dx \\
& \leq  K \rho \big(1+ h (\| H_{pp}\|_{0}\| D^2 w\|_0 + \|H_{px}\|_0)\big)
\leq K\rho \big(1+ \frac{h}{\epsilon} \|D u_{\rho,h} \|_0\big),
\end{align*}
and hence $E_\Phi=\mathcal{O}(\frac{h\rho}{\epsilon})$. Similarly, $E_\nu=\mathcal{O}(h\rho^2\lambda_r)=\mathcal{O}(\frac{h\rho^2}{r^{\sigma}})$.

From the above estimates, we find that
\begin{align*}
&\int_{\rd} \phi(x) d\big(\widehat{\widehat{m}}(t_{k+1}) - \widehat{\widehat{m}}(t_k)\big)(x) 
% = \sum_j \frac{m_{j,k}}{\rho^d} \int_{E_j} \bigg( \frac{e^{-\lambda_r h}}{2d}\Big[\sum_{p=1}^d \phi(\Phi_{k,p}^{\epsilon,+})(x)+ \phi(\Phi_{k,p}^{\epsilon,-})(x)\Big] \\
%  & \qquad \qquad  + \frac{1-e^{-\lambda_r h}}{\lambda_r} \int_{|z|>r} \phi(x+z) \nu(dz) \bigg) dx  +  \mathcal{O} (\rho^2 + \frac{h\rho}{\epsilon_n})  - \int_{\rd} \phi(x) d\widehat{\widehat{m}}(t_k)(x)\\
 =  \int_{\rd} \Big( \frac{e^{-\lambda_r h}}{2d}\sum_{p=1}^d [\phi(\Phi_{k,p}^{\epsilon,+})(x)+ \phi(\Phi_{k,p}^{\epsilon,-})(x)- 2\phi(x)]  \\
 & \qquad \quad   + \frac{1-e^{-\lambda_r h}}{\lambda_r} \int_{|z|>r} \big(\phi(x+z) - \phi(x)\big) \nu(dz) \Big) d\widehat{\widehat{m}}(t_k)(x) +  \mathcal{O} \big(\rho^2 + \frac{h\rho}{\epsilon} +\frac{h\rho^2}{r^{\sigma}}\big).
\end{align*}
By a similar argument as in \eqref{err_bnd-consist00} and using Lemma \ref{lem:small_jump},
\begin{align*}
\phi(\Phi_{k,p}^{\epsilon,+})(x)+ \phi(\Phi_{k,p}^{\epsilon,-})(x)- 2\phi(x) = & - 2h \Big(D\phi(x)\cdot D_p H(x, Dw(t_k,x)) + B_r^{\sigma}\cdot D\phi(x)\Big)\\
& \, + 2h \mathcal{L}_r[\phi](x) + \mathcal{O}(h^2r^{2-2\sigma} + hr^{3-\sigma} ).
\end{align*}
Hence using \eqref{consist_esti1} and \eqref{consist_esti2} we have 
\begin{align*}
&\int_{\rd} \phi(x)d(\widehat{\widehat{m}}(t_{k+1}) -\widehat{\widehat{m}}(t_k))(x)  \\
& = h\int_{\rd} \big[- D\phi(x)\cdot D_p H(x, Dw(t_k,x)) + \mathcal{L}_r[\phi](x)+\mathcal{L}^r[\phi](x)\big] d\widehat{\widehat{m}}(t_k)(x) \\
  & \quad+\mathcal{O}(h^2 r^{-\sigma} + h^2 r^{1-2\sigma} + h^2 r^{2-2\sigma} ) +  \mathcal{O} (\rho^2 + \frac{h\rho}{\epsilon} +\frac{h\rho^2}{r^{\sigma}} +h^2r^{2-2\sigma} + hr^{3-\sigma}).
\end{align*}
Summing from $k=0$ to $k=n-1$ and approximating sums by integrals, we obtain
\begin{align} \label{eq:proof_main_1}
\begin{split}
  &\int_{\rd} \phi(x)d\widehat{\widehat{m}}(t_n)(x)-\int_{\rd}\phi(x) d\widehat{\widehat{m}}(t_0)\\
&= h\sum_{k=0}^{n-1} \int_{\rd} \big[- D\phi(x)\cdot D_p H(x,Dw(t_k,x)) + \mathcal{L}[\phi](x)\big] d\widehat{\widehat{m}}(t_k)(x) \\
 & \qquad \qquad + n \,\mathcal{O} (\rho^2 + \frac{h\rho}{\epsilon} +\frac{h\rho^2}{r^{\sigma}} +h^2r^{-\sigma} + hr^{3-\sigma})\\
 & = \int_{\rd}\int_0^{t_n}  \big[- D\phi(x)\cdot D_p H(x, Dw(s, x)) + \mathcal{L}[\phi](x)\big] d\widehat{\widehat{m}}(s)(x) \, ds\\
& \qquad \qquad + \mathcal{O} \Big(\frac{\rho^2}{h} + \frac{\rho}{\epsilon}+ \frac{\rho^2+h}{r^{\sigma}}+ r^{3-\sigma}\Big) + E,
%=&  h\sum_{k=0}^{n-1} \int_{\R} \Big(- D\phi(x)\cdot H_p(x,Du^{\epsilon_n}_{\rho_n,h_n}(x,t_k)) + \mathcal{L}[\phi](x)\Big) d\widehat{\widehat{m}}(t_k)(x) +\mathcal{O} (\frac{\rho^2}{h} + \frac{\rho}{\epsilon_n}+hr^{-2\sigma} + h^{\frac{1}{2}} \sigma_r^3+ r^{3-\sigma}).
\end{split}
\end{align}
where $E$ is Riemann sum approximation error. Let $I_k(x):= - D\phi(x)\cdot D_p H(x, Dw(t_k,x))$ $+ \mathcal{L}[\phi](x)$ and use time-continuity $\widehat{\widehat{m}}$ in the $d_0$-metric (Theorem \ref{thm:time_equicont_m}), %the fact 
that $w(\cdot,x)$ is constant on $[t_k, t_{k+1})$, \ref{H1}, \ref{H2} and  $\|D^2w\|_{0}\leq \frac{C \|Du_{\rho,h}\|_{0}}{\epsilon}$, to conclude that  for $s\in [t_k, t_{k+1})$  
\begin{align*}
& \int_{t_k}^{t_{k+1}} \int_{\rd}  I_k(x) d\big(\widehat{\widehat{m}}(t_k) - \widehat{\widehat{m}}(s)\big)(x) ds \leq h \big(\|I_k\|_{0} + \|DI_k\|_{0}\big) C_0 \sup_{s\in[t_k,t_{k+1})} \sqrt{s-t_k} \\
&\leq K h \Big( 1 + \|Dw\|_{0} + \|D^2w\|_{0} \Big)   \sqrt{h} \leq K h \Big(1+ \frac{1}{\epsilon}\Big) \sqrt{h}.
\end{align*}
Summing over $k$, we have $E= \big| \sum_{k=0}^{n-1}\int_{t_k}^{t_{k+1}} \int_{\rd}  I_k(x) d\big(\widehat{\widehat{m}}(t_k) - \widehat{\widehat{m}}(s)\big)(x) ds\big| = \mathcal O(\frac{\sqrt h}{\epsilon})$.
%\begin{align*}
%&h\sum_{k=0}^{n-1} \int_{\R} \Big(- D\phi(x)\cdot H_p(x,Du^{\epsilon_n}_{\rho_n,h_n}(x))-\mathcal{L}[\phi](x)\Big) d\widehat{\widehat{m}}(t_k)(x) \\
%=&  \sum_{k=0}^{n-1} \int_{\R} \int_{t_k}^{t_{k+1}} \Big(- D\phi(x)\cdot H_p(x,Du^{\epsilon_n}_{\rho_n,h_n}(x))-\mathcal{L}[\phi](x)\Big) d\widehat{\widehat{m}}(s)(x) \\
%& \hspace{5cm} + \sum_{k=0}^{n-1}h\sup_{s\in[t_k,t_{k+1})}\mathcal{O}\Big(\frac{\sqrt{s-t_k}}{\epsilon_n}\Big).
%\end{align*} 
% Hence from \eqref{eq:proof_main_1} we have
% \begin{align*}
% &\int_{\rd} \phi(x)d\widehat{\widehat{m}}(t_n)(x)-\int_{\rd}\phi(x) d\widehat{\widehat{m}}(t_0)\\
% =& \int_{\rd}\int_0^{t_n}  \Big(- D\phi(x)\cdot H_p(x,Du^{\epsilon_n}_{\rho_n,h_n}(s, x)) + \mathcal{L}[\phi](x)\Big) d\widehat{\widehat{m}}(s)(x) \, ds\\
% & \hspace{6cm}+ \mathcal{O} \Big(\frac{\rho^2}{h} + \frac{\rho}{\epsilon_n}+hr^{-\sigma}+ r^{3-\sigma} + \frac{\sqrt{h}}{\epsilon_n}\Big).
% \end{align*}

Since $\widehat{\widehat{m}}$ converges to $m$ in $C([0,T], P(\rd))$ and $\phi\in C_c^{\infty}(\rd)$ implies $\mathcal{L}[\phi] \in C_b(\rd)$, we have 
\begin{align}\label{conv_m_n}
\int_{\rd}\int_0^{t_n}  \mathcal{L}[\phi](x) d\widehat{\widehat{m}}(s)(x) \xrightarrow{n\rightarrow\infty}\int_{\rd}\int_0^{t}  \mathcal{L}[\phi](x) d{m}(s)(x).
\end{align} 
It now remains to show convergence of the $D_p H$-term and pass to the limit in \eqref{eq:proof_main_1} to get that $m$ is a very weak solution satisfying \eqref{dist_sol_FPK}.
\medskip

\noindent \textbf{Step 4} \textit{(Proof of Theorem \ref{thm:convergence_MFG} (a) (iii))}\textbf{.}\quad  
Now $d=1$ and part (ii) of Theorem \ref{thm:convergence_MFG} (a) implies that  $\widehat{\widehat{m}} \overset{\ast}{\rightharpoonup} m$ in $L^{\infty}([0,t]\times \R)$ if $m_0 \in L^{\infty}(\R)$, or $\widehat{\widehat{m}}\rightharpoonup m$ in $L^{1}([0,t]\times \R)$ if $m_0 \in L^{p}(\R)$ for $p \in (1,\infty)$. We also have  $Dw(t,x)=Du_{\rho,h}^{\epsilon}(t,x)\rightarrow Du(t,x)$ almost everywhere in $[0,T]\times \R$ by Theorem \ref{thm:HJB-convergence-smoothsol}~(ii). Since $D\phi \in C_c^{\infty}(\R)$ and $D_p H(\cdot,Dw)$ uniformly bounded, by the triangle inequality and the  dominated convergence Theorem we find that 
\begin{align*}%\label{conv_v_h}
\int_{\R} \int_{0}^{t_n}   D\phi(x)\cdot & D_p H(x,Dw(s,x)) \,   d\widehat{\widehat{m}}(s)(x) \\ 
& \longrightarrow \int_{\R} \int_0^t D\phi(x)\cdot D_p H(x,Du(s,x)) \, d{m}(s)(x).
\end{align*}
 Then by passing to the limit in \eqref{eq:proof_main_1} using the above limit,  \eqref{conv_m_n}, and the CFL conditions $\frac{\rho^2}{h},\frac{h}{r^{\sigma}},\frac{\sqrt{h}}{\epsilon}=\mathit{o}(1)$ (note that $\rho^2\leq h$ for large $n$), we see that \eqref{dist_sol_FPK} holds and ${m}$ is a very weak solution of the FPK equation. This completes the proof of Theorem \ref{thm:convergence_MFG} (a) (iii). \medskip

\noindent \textbf{Step 5}\textit{(Proof of Theorem \ref{thm:convergence_MFG-nondeg}(iii))}\textbf{.} \quad Now \ref{U1} holds and  $Dw=Du_{\rho,h}^{\epsilon}\rightarrow Du$ locally uniformly by Theorem \ref{thm:HJB-convergence-smoothsol}~(iii).  
Since $D \phi \in C^{\infty}_c(\rd)$ and $\int_{\rd} d\widehat{\widehat{m}}(s)(x) =1$, by continuity and uniform boundedness of $D_p H(\cdot, Dw)$, it follows that 
\begin{align} \label{eq:main_proof_2}
\begin{split}
 & \Big|\int_{\rd} \int_{0}^{t_n}  D\phi(x)\cdot D_p H(x,Dw(s,x)) \, d\widehat{\widehat{m}}(s)(x) \\
 & \hspace*{3cm} - \int_{\rd} \int_0^{t_n} D\phi(x)\cdot D_p H(x,Du(s,x)) \, d{\widehat{\widehat{m}}}(s)(x)\Big|  \\
& \leq T \|D\phi\|_0 \|D_{pp} H\|_0 \|Dw - Du \|_{L^{\infty}(\textup{supp}(\phi))}  \int_{\rd} d\widehat{\widehat{m}}(s)(x)  \longrightarrow 0.
\end{split}
\end{align} 
Since $\widehat{\widehat{m}} \rightarrow m$ in $C([0,T], P(\rd))$ and $D\phi\cdot D_p H(\cdot,Du)(t) \in C_b(\rd)$ by \ref{U1}, we get
\begin{align*}%\label{eq:main_proof_3}
 \int_{\rd} \int_{0}^{t_n}  D\phi(x)\cdot & D_p H(x,Du(s,x)) \, d\widehat{\widehat{m}}(s)(x) \\ 
 & \longrightarrow \int_{\rd} \int_0^t D\phi(x)\cdot D_p H(x,Du(s,x)) \, d{m}(s)(x).
\end{align*}
Then by passing to the limit in \eqref{eq:proof_main_1} using the above limit,  \eqref{eq:main_proof_2}, \eqref{conv_m_n}, and the CFL conditions $\frac{\rho^2}{h},\frac{h}{r^{\sigma}},\frac{\sqrt{h}}{\epsilon}=\mathit{o}(1)$,  we see that \eqref{dist_sol_FPK} holds and ${m}$ is a  very weak solution of the FPK equation. This completes the proof of Theorem \ref{thm:convergence_MFG-nondeg}(iii).   

\section{Numerical examples} \label{sec:numerics}

For numerical experiments we look at

\begin{align}
 \begin{cases}
     -u_t - \sigma^2 \mathcal{L} u + \frac{1}{2} |u_x|^{2} = f ( t,x ) + K \ \phi_{ \delta } \ast m ( t,x ) \quad &\text{ in } (0,T)\times [ a,b ], \\
        m_t - \sigma^2 \mathcal{L}^{*} m - \text{div} (m u_x ) = 0 \quad &\text{ in } (0,T)\times [ a,b ], \\
        u (T,x) = G(x,m(T)), \qquad m(x,0) = m_0 (x) \quad &\text{ in } [ a,b ],
    \end{cases}
    \label{eqn:MFG_numerical}
\end{align}
where $a<b$ are real numbers, $\mathcal{L}$ is a diffusion operator, $\phi_{\delta} = \frac{1}{\delta \sqrt{2 \pi}} e^{-\frac{x^2}{2 \delta^2}} $, $K$ some real number, and $f$ is some bounded smooth function. We will specify these quantities in the examples below. 

\subsection*{Artificial boundary conditions}
Our schemes \eqref{schme_HJ} and \eqref{Fokker-Planck_discretized} for approximating \eqref{eqn:MFG_numerical} are posed in all of $\mathbb{R}$. To work in a bounded domain we impose (artificial) exterior conditions: 
\begin{enumerate}
    \item[(U1)] $u \equiv \| u_{0}  \|_{0} + T \cdot \| f  \|_{L^{\infty} ( ( 0,T ) \times ( a,b ) )}  $ in $(\mathbb{R} \setminus [ a,b ] ) \times [ 0,T ]$, 
        \medskip
    \item[(M1)]  $m \equiv 0$ in $(\mathbb{R} \setminus [ a,b ] ) \times [ 0,T ]$, and
        $m_{0}$ is compactly supported in $ [ a,b ]$.
        \medskip
\end{enumerate}
Condition (U1)  penalize being in $[a,b]^c$  ensuring that optimal controls  $\alpha$ in \eqref{schme_HJ} are such that $x_{i} - h \alpha \pm \sqrt{h} \sigma_{r} \in [a,b]$. 
Moreover, the contributions to non-local operators of $u$ from $[a,b ]^c$ will be small away from the boundary.
Condition (M1) ensures that the mass of $m$ is essentially contained in $[a,b]$ up to some finite time (but some mass will leak out due to nonlocal effects), and there is no contribution from $[ a,b ]^c$ when we compute non-local operators of $m$.  We will present numerical results from a region of interest that is far away from the boundary of $[a,b ]$, and where the influence of the (artificial) exterior data is expected to be negligible.

\subsection*{Evaluating the integrals}
To implement the scheme, we need to evaluate the integral 
\begin{align*}
    \int_{|z| \geq r} I [ f ] ( x_{i} + z ) \nu ( dz ) = \sum_{j \in \mathbb{Z}} f [ x_{i} ] \omega_{j-i, \nu} ,  
\end{align*}
where 
\begin{align*}
    \omega_{j-i, \nu} = \int_{|z| \geq r } \beta_{j-i} ( z ) \nu ( dz ),
\end{align*}
see \eqref{linear_interpolation}. 
In addition, we need to compute the values of $\sigma_{r}, b_{r}$,  and $\lambda_{r}$ (see \eqref{sigma_r}, \eqref{B_r}, and \eqref{lambda_r}).
To compute the weights $\omega_{j-i, \nu}$ we use two different methods. 
For the fractional Laplacians, 
we use the explicit weights of \cite{huang2014numerical}, while for CGMY diffusions we calculate the weights numerically using the inbuilt integral function in MATLAB. 
When tested on the fractional Laplacian, the MATLAB integrator produced an error of less than $10^{-15}$. Below the quantities $\sigma_{r}, b_{r}, \lambda_{r}$ are computed explicitly, except in the CGMY case where we use numerical integration. 

\subsection*{Solving the coupled system}
We use a fixed point iteration scheme: (i) Let $\mu = m_0$, and solve for $u_{\rho,h}$ in \eqref{schme_HJ}--\eqref{extnd_dscrt_schm_HJ}. (ii)
With approximate optimal control $Du_{\rho,h}^{\epsilon}$ as in \eqref{alp_num}, we solve for 
$m_{\rho,h}^{\epsilon}$ in \eqref{Fokker-Planck_discretized}. (iii) Let
$\mu_{\text{new}} = ( m_{\rho,h}^{\epsilon} + \mu)/2$, and repeat the process with $\mu=\mu_{\text{new}}$. We continue until we have converged to a fixed point to within machine accuracy.

\begin{remark}
Instead $\mu_{\text{new}} = m_{\rho,h}^{\epsilon}$, we take $\mu_{\text{new}} =( m_{\rho,h}^{\epsilon} + \mu)/2$.  I.e. we use a fixed point iteration with some memory. This gives much faster convergence in our examples. 
\end{remark}

\begin{example}\label{ex1}
Problem \eqref{eqn:MFG_numerical} with $ [ 0,T ] \times [ a,b ] = [ 0,2 ] \times [ 0,1  ]$, $G = 0$, $f ( t,x ) = 5 (  x - 0.5(1 - \sin(2 \pi t)))^2$,  
$m_{0} ( x ) = C e^{-\frac{(x-0.5)^2}{0.1^2}} $, where $C$ 
is such that $\int_{a}^b m_{0} = 1$. 
Furthermore, 
in accordance with the CFL-conditions of 
Theorem \ref{thm:convergence_MFG},
we let $h = \rho = 0.005$, $r = h^{\frac{1}{2 s}}$,
$\epsilon = \sqrt{h} \approx 0.0707$, $\sigma=0.09$, $\delta = 0.4$, $K=1$. 

For the diffusions, we consider $\mathcal{L} = ( - \Delta )^{\frac{s}{2}}$ 
for $s =0.5, 1.5, 1.9$, 
$\mathcal{L} = \Delta$, and $\mathcal{L} \equiv 0$.
In figure \ref{fig:all_times} we plot
the different solutions at time $t=0.5$ and  $t=1.5$. 
\end{example}
\begin{figure}[ht!]
    \centering
    \begin{subfigure}[b]{0.49\textwidth}
        \includegraphics[width=\textwidth]{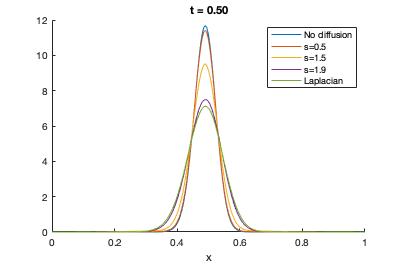}
        \caption{$t=0.5$}
        \label{fig:ex1m05}
    \end{subfigure}
    ~ %add desired spacing between images, e. g. ~, \quad, \qquad, \hfill etc. 
    \begin{subfigure}[b]{0.49\textwidth}
        \includegraphics[width=\textwidth]{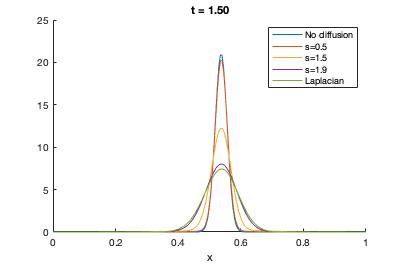}
        \caption{$t=1.5$}
        \label{fig:ex1u05}
    \end{subfigure}
    \caption{The solutions $m$ in Example 1.}
    \label{fig:all_times}
\end{figure}

In figure \ref{fig:s_15} we plot the solution with $s=1.5$ on 
the time interval $[ 0,2 ]$.

\begin{figure}[ht!]
    \centering
    \begin{subfigure}[b]{0.49\textwidth}
        \includegraphics[width=\textwidth]{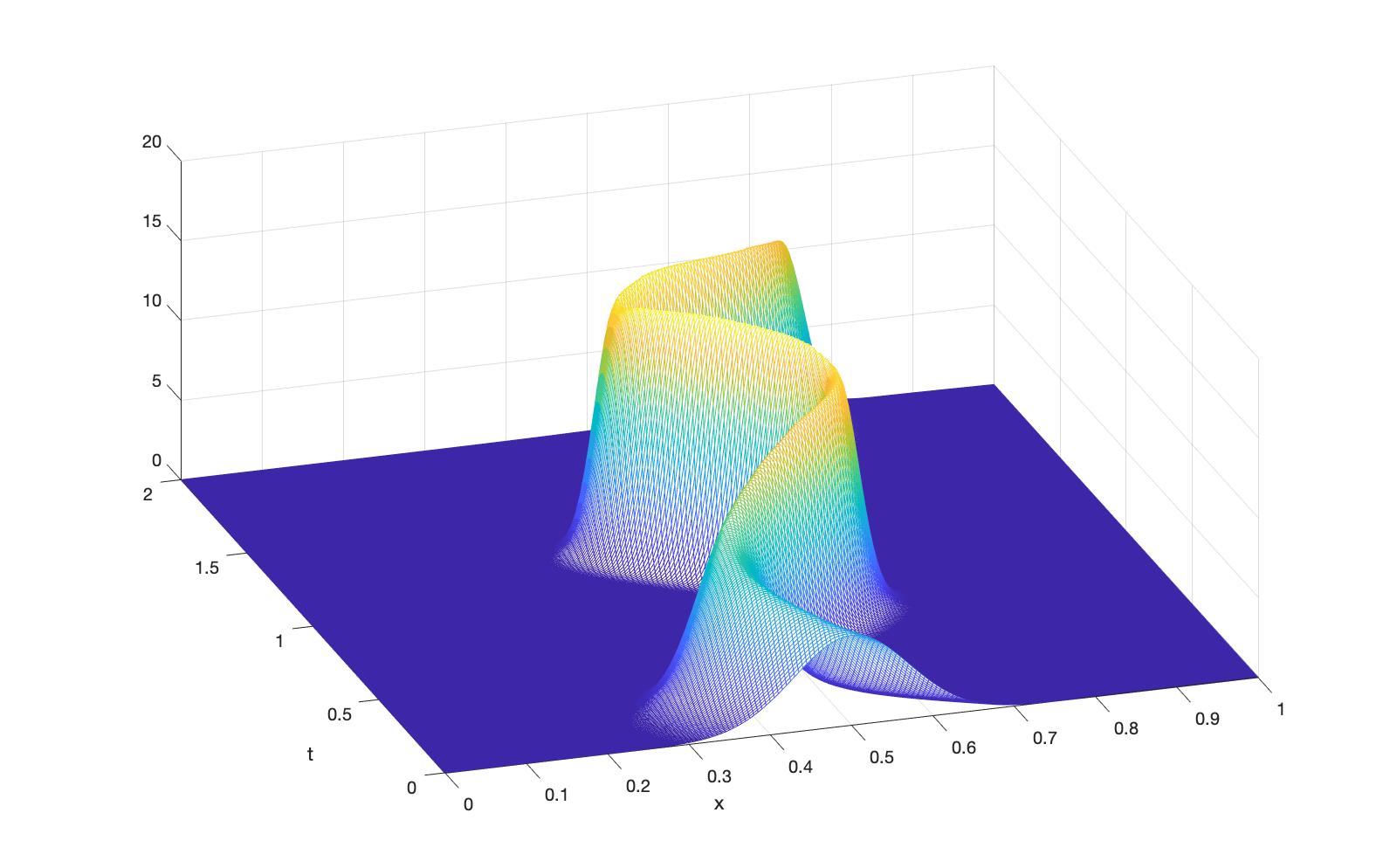}
        \caption{$m ( t,x )$}
        \label{fig:ex1m15}
    \end{subfigure}
    ~ %add desired spacing between images, e. g. ~, \quad, \qquad, \hfill etc. 
    \begin{subfigure}[b]{0.49\textwidth}
        \includegraphics[width=\textwidth]{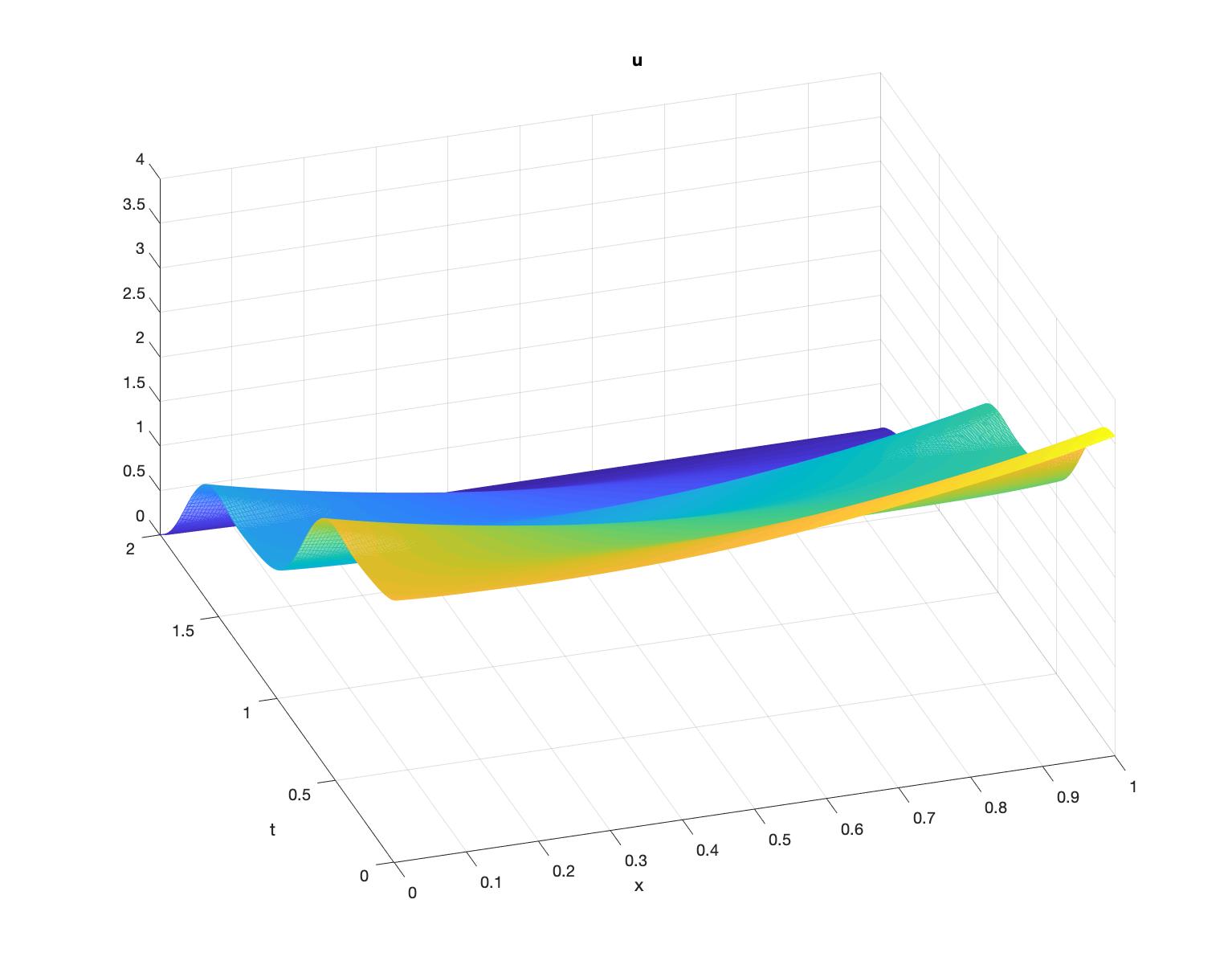}
        \caption{$u(t,x)$}
        \label{fig:ex1u15}
    \end{subfigure}
    \caption{Solution $m$ and $u$ in Example 1 with diffusion parameter $s=1.5$} 
    \label{fig:s_15}
\end{figure}

\begin{example}\label{ex2}
Problem \eqref{eqn:MFG_numerical} with the same cost functions as in Example 1, but  different diffusions with parameter
$s=1.5$:

\begin{enumerate}
\item $\mathcal{L} = \sigma^2 ( - \Delta )^{\frac{s}{2}},$
\item $\mathcal{L} =  \sigma^2 C_{d,s} \int_{\mathbb{R}} [   u ( x+y ) - u ( x ) - Du ( x ) \cdot y \mathbbm{1}_{|y|<1} ] \, \mathbbm{1}_{ [ 0, + \infty )}    \, \frac{dy}{|y|^{1+s}}$,  
\item $ \mathcal{L} = \sigma^2 C_{d,s} \int_{\mathbb{R}} [   u ( x+y ) - u ( x ) - Du ( x ) \cdot y \mathbbm{1}_{|y|<1} ] \, \mathbbm{1}_{ [ -0.5,0.5 ]^{c}} \, \frac{dy}{|y|^{1+s}}$,
\item $\mathcal{L} = \sigma^2 C_{d,s} \int_{\mathbb{R}} [   u ( x+y ) - u ( x ) - Du ( x ) \cdot y \mathbbm{1}_{|y|<1} ] \, e^{-10 y^{-} -y^{+}} \,  \frac{dy}{|y|^{1+s}}  $,
\end{enumerate}
where $C_{d,s}$ is the normalizing constant for the fractional Laplacian (see \cite{huang2014numerical}).
Case (i) is the reference solution, a symmetric and uniformly elliptic operator.  Case (ii) is non-symmetric and non-degenerate,  case (iii) is symmetric and degenerate, and case (iv) is a CGMY-diffusion (see e.g. \cite{tankov2003financial}). We have plotted $m$ at $t=0.5$ and  $t=1.5$ in Figure \ref{fig:deg_diff}.  
\end{example}

\begin{figure}[ht!]
    \centering
    \begin{subfigure}[b]{0.49\textwidth}
        \includegraphics[width=\textwidth]{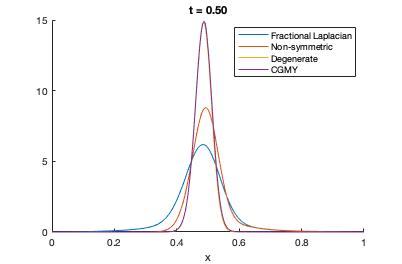}
        \caption{$t=0.5$}
        \label{fig:deg_t05}
    \end{subfigure}
    ~ %add desired spacing between images, e. g. ~, \quad, \qquad, \hfill etc. 
    \begin{subfigure}[b]{0.49\textwidth}
        \includegraphics[width=\textwidth]{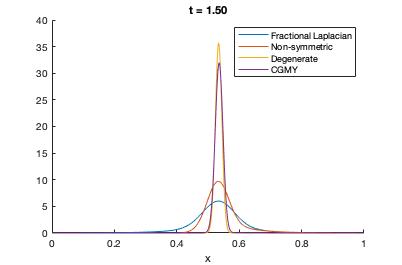}
        \caption{$t=1.5$}
        \label{fig:deg_t15}
    \end{subfigure}
    ~ %add desired spacing between images, e. g. ~, \quad, \qquad, \hfill etc. 
    \caption{The solutions $m$ in Example 2}\label{fig:deg_diff}
\end{figure}

\begin{example}\label{ex3}
(Long time behaviour).
Under certain conditions (see e.g. \cite{cardaliaguet2012long, cardaliaguet2013long}), 
the solution of time dependent
MFG systems will quickly converge to the solution of the corresponding
stationary ergodic MFG system, as the time horizon $T$ increases. 
We check numerically that this is also the case for nonlocal diffusions. 
In \eqref{eqn:MFG_numerical}, we take $\mathcal{L} = ( -\Delta )^{\frac{s}{2}}$, with $s=1.5$, $ [ 0,T ] \times [ a,b ] = [ 0,10 ] \times [ -1,2 ]$, 
$G ( x ) = ( x-2 )^{2}$,
 $f ( t,x ) = x^2$, and $m_{0} ( x ) = \mathbbm{1}_{ [ 1,2 ] } ( x )$.
We expect (from the cost functions $f$ and $G$) that
the solution $m$ will approach the line $x=0$ 
quite fast, and then travel along this line, 
until it goes towards the point $x=2$ in the very end. 
Our numerical simulations shows that this is the case 
also for nonlocal diffusions. 
Here we have considered the cases $K=0$ (no coupling in the $u$ equation) and $K = 0.4$ (some coupling).
The parameters used in the simulations are 
$h = \rho = 0.01$, $\epsilon = \sqrt{h}$, $r = h^{1 / 2s}$,
and the results are shown in Figure \ref{fig:long_time}.
\end{example}
\begin{figure}[ht!]
    \centering
    \begin{subfigure}[b]{0.49\textwidth}
        \includegraphics[width=\textwidth]{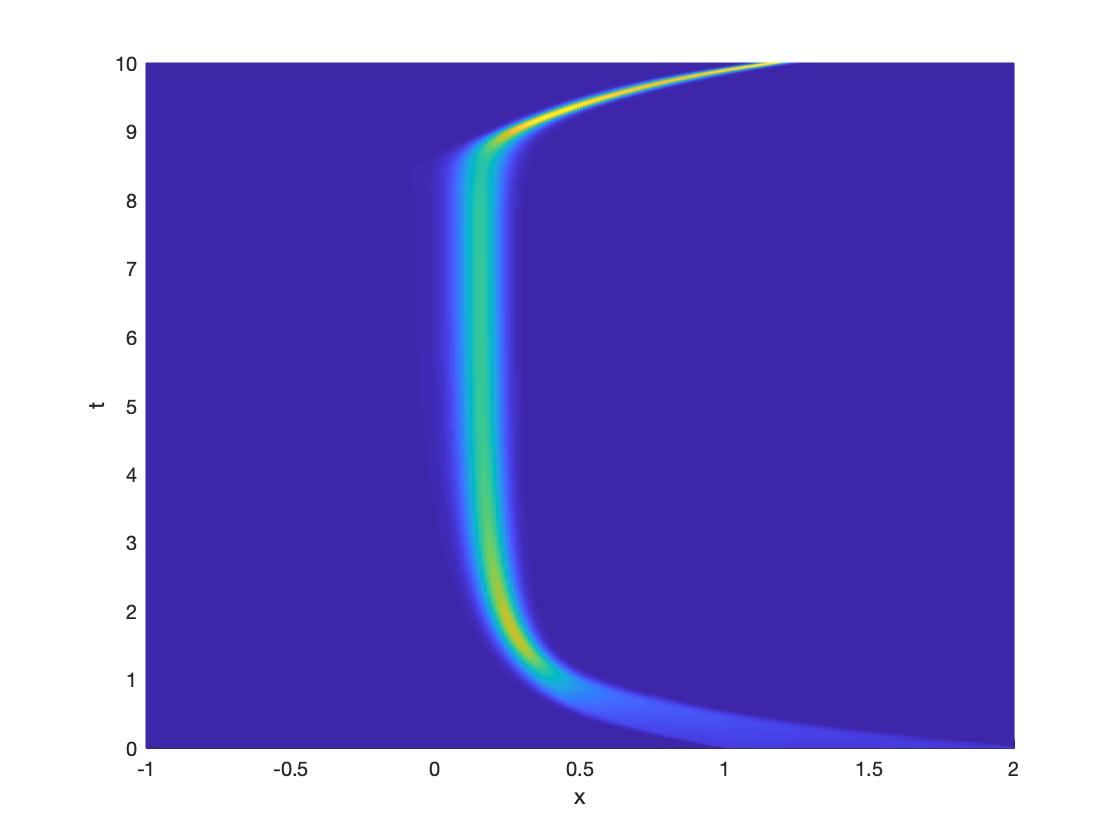}
        \caption{$f$}
        \label{fig:ex3m15}
    \end{subfigure}
    ~ %add desired spacing between images, e. g. ~, \quad, \qquad, \hfill etc. 
    \begin{subfigure}[b]{0.49\textwidth}
        \includegraphics[width=\textwidth]{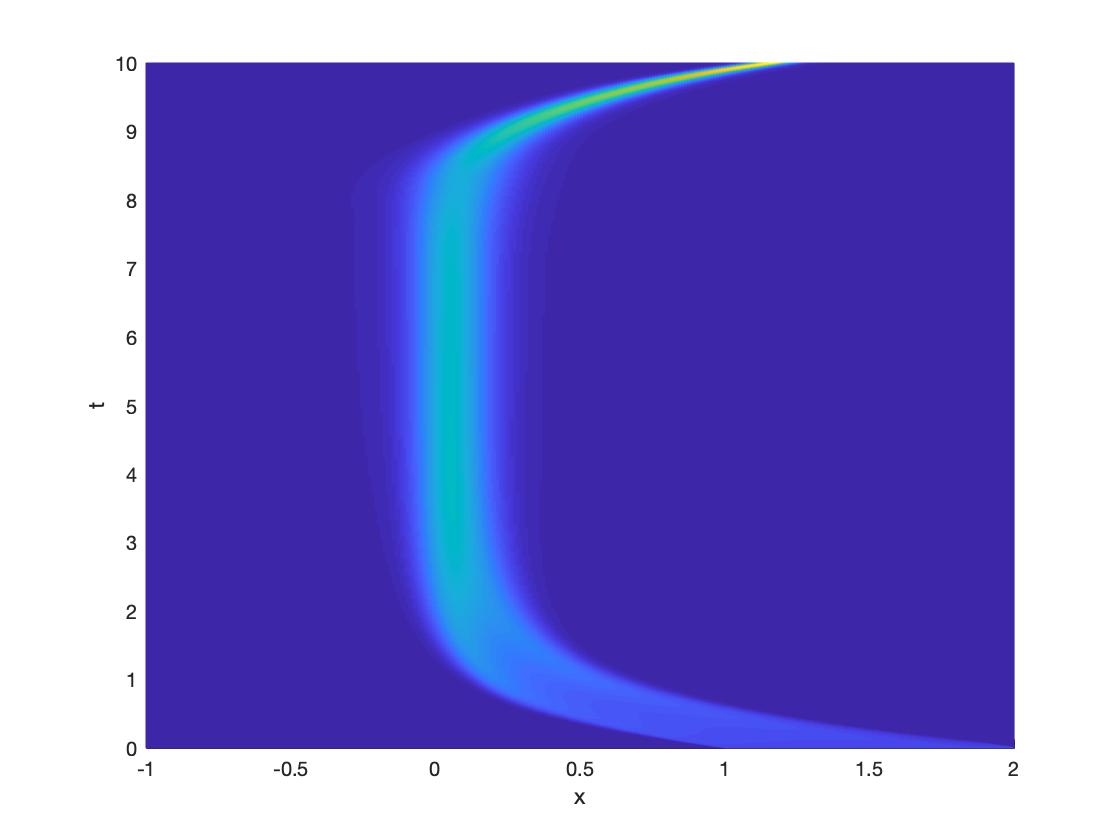}
        \caption{$f+ 0.4 \phi_{\delta}$}
        \label{fig:ex3u15}
    \end{subfigure}
    ~ %add desired spacing between images, e. g. ~, \quad, \qquad, \hfill etc. 
    \caption{ The solutions $m$ in Example 3}\label{fig:long_time}
\end{figure}

The players want to avoid each other in the case of $K=0.4$, so the solution is more spread out in space direction than in the case of $K=0$.

\begin{example}\label{ex4}
We compute the convergence rate
when $f$, $G$, $m_{0}$ are as in Example 1, $s=1.5$, $\nu=0.2$, $\delta=0.4$, and 
the domain $ [ 0,T ] \times [ a,b ] = [ 0,0.5 ] \times [ 0,1  ]$. 
We take $\rho = h$, $r = h^{\frac{1}{2 s}}$, and for simplicity $\epsilon = 0.25$. 

We calculate solutions for different values of $h$, and compare with a reference
solution computed at $h=2^{-10}$.
We calculate $L^{\infty}$ and $L^{1}$ relative errors restricted to the $x$-interval $[ \frac{1}{3}, \frac{2}{3} ]$ (to avoid boundary effects), 
and $t=0$ for  $u$ and  $t=T$ for  $m$:
\begin{align*}
   \textup{ERR}_u:= \frac{\| u_{\rho,h} ( 0,\cdot ) - u_{\text{ref}} ( 0,\cdot ) \|_{L^{\infty} ( \frac{1}{3}, \frac{2}{3} )}}{\|  u_{\text{ref}} ( 0,\cdot ) \|_{L^{\infty} ( \frac{1}{3} , \frac{2}{3} )}}, \quad
\textup{ERR}_m:=\frac{\| m_{\rho,h}^{\epsilon} ( T,\cdot ) - m_{\text{ref}} ( T,\cdot ) \|_{L^{1} ( \frac{1}{3} , \frac{2}{3} )}}{\|  m_{\text{ref}} ( T,\cdot ) \|_{L^{1} ( \frac{1}{3} , \frac{2}{3} )}}.
\end{align*}
The results are given in the table below.
\end{example}

\begin{table}[ht!]
\begin{tabular}{l|llllllll}
    h     & $2^{-2}$ & $2^{-3}$ & $2^{-4}$ & $2^{-5}$ & $2^{-6}$ & $2^{-7}$ & $2^{-8}$ & $2^{-9}$ \\
    \hline
    ERR$_u$ & 0.3155  &  0.1951  &  0.0920  &  0.0446 &   0.0218  &  0.0097  &  0.0035 & 0.0013 \\
    ERR$_m$ & 0.8055 &   0.4583 &   0.2886  &  0.1869  &  0.1023  &  0.0596  &  0.0300 & 0.0186
\end{tabular}
\end{table}

\noindent We see that when we halve $h$, the error is halved, i.e we observe an error of order $O ( h )$.
\iffalse
\begin{figure}[ht!]
    \centering
    \begin{subfigure}[b]{0.5\textwidth}
        \includegraphics[width=\textwidth]{figures/errorplot_u.png}
        \caption{$L^{\infty}$-error in $u$}
        \label{fig:u_error}
    \end{subfigure}
    ~ %add desired spacing between images, e. g. ~, \quad, \qquad, \hfill etc. 
    \begin{subfigure}[b]{0.5\textwidth}
        \includegraphics[width=\textwidth]{figures/errorplot_m.png}
        \caption{$L^{1}$-error in $m$}
        \label{fig:m_error}
    \end{subfigure}
    ~ %add desired spacing between images, e. g. ~, \quad, \qquad, \hfill etc. 
    \label{fig:errors}
\end{figure}
\fi

\appendix

\section{Proof of Proposition \ref{thm:existence_discrete_system}}\label{app:pf_ex}
   The proof is an adaptation of the Schauder fixed point argument used to prove existence for MFGs. We will use a direct consequence of Theorem \ref{thm:thightness_m} and \ref{thm:time_equicont_m}:
\begin{corollary}
\label{cor:thightness_m_rho-h}
    Assume \ref{A0},\ref{A2}, \ref{L1}--\ref{L2}, \ref{H1}, \ref{F2}, \ref{M1}, $\Psi$ is given by Proposition \ref{prop:tail-control-function}, and $m^{\epsilon}_{\rho,h}[\mu]$ is defined by \eqref{exten_disc_measure}. Then there is $C_{\rho,h,\epsilon}>0$, such that for any  $\mu \in  C ( [0,T], P ( \rd ))$ and $t,s\in[0,T]$,
\begin{align*}
\int_{\R^d} \Psi(x) \, dm^{\epsilon}_{\rho,h}[\mu](t) +\frac{d_0(m^\epsilon_{\rho,h}[\mu](t),m^\epsilon_{\rho,h}[\mu](s))}{\sqrt{|t-s|}}\leq C_{\rho,h,\epsilon}.
\end{align*}
\end{corollary}
\smallskip
The point is that $\rho,h,\epsilon$ are fixed in this result. Let 
 \begin{align*}% \label{set_fixedpoint}
    \mathcal{C} := \Big\{ &\mu \in  C (  0,T  ; P ( \R^d  ) ) : \mu ( 0 ) = m_{0}, \\ 
        & \quad \sup_{t,s \in [  0,T] }\Big[ \int_{\R^d} \psi ( x ) d \mu ( t,x ) +
        \frac{d_{0} ( \mu ( t ), \mu ( s ) )}{\sqrt{ | t-s | }} \Big]\leq C_{\rho,h, \epsilon} 
    \Big\}, \notag
\end{align*}
where $C_{\rho,h, \epsilon}$ is defined in Corollary \ref{cor:thightness_m_rho-h}.  
%These are uniformly bounded, independent of $\mu$. On $\mathcal{C}$ we define a fixed point map $S$ as follows: 
For $\mu \in \mathcal{C}$, let $u_{\rho,h} [ \mu ]$ be solution of \eqref{schme_HJ} and $u_{\rho,h}^{\epsilon} [ \mu ]$ defined by
\eqref{cont_extnd_dscrt_schm_HJ}. Then $m_{\rho,h}^{\epsilon}  = S  ( \mu )$ is defined to the corresponding 
solution of \eqref{Fokker-Planck_discretized}.
Note that a fixed point of $S$ will give a solution $(u,m)$ of the scheme \eqref{disc_mfg_system}. We now conclude the proof by applying Schauder's fixed point theorem since:
\medskip 

 \noindent 1.~~($\mathcal C$ is a convex, closed, compact set). It is a convex and closed by standard arguments and compact by the
Prokhorov and Arzel\`a-Ascoli theorems.
\medskip 

\noindent 2.~~($S$ is a self-map on $\mathcal C$). The map $S$ maps $\mathcal{C}$ into itself by Corollary \ref{cor:thightness_m_rho-h} (tightness and equicontinuity),  and Lemma \ref{mass_conservation} (positivity and mass preservation). \smallskip %(see also \cite[Proof of Theorem 2.6]{ersland2020classical}).
\medskip 

\noindent 3.~~($S$ is continuous). Let $ \mu_{n} \to \mu$ in $\mathcal{C}$. By
Theorem \ref{hjb_discrete_comparison} (comparison)
and \ref{F2},
\begin{align*}
&\| u_{\rho,h} [ \mu_{n} ] - u_{\rho,h} [ \mu ]  \|_{0} \\
&  \leq T \sup_{  t,x   } | F ( x, \mu_{n} ( t ) ) - F ( x, \mu ( t ) ) |  \notag
+ \sup_{x} | G ( x, \mu_{n} ( T ) ) - G ( x, \mu ( T ) ) | \\ 
&\leq T L_{F} \, \sup_{t} d_{0} ( \mu_{n} ( t ) , \mu ( t ) ) + L_{G} \,  d_{0} ( \mu_{n} ( T ) , \mu ( T ) ) \to 0. \notag
\end{align*}
Then
$ 
\sup_{i} \big| \frac{u_{i,k} [ \mu_{n} ] - u_{i-j,k} [ \mu_{n} ]}{ \rho   } - 
\frac{u_{i,k} [ \mu ] - u_{i-j,k} [ \mu ]}{ \rho   } \big| \to 0$ 
uniformly for $| i-j | =   1 $, $\| Du_{\rho,h}^{\epsilon} [ \mu_{n} ] - Du_{\rho,h}^{\epsilon} [ \mu ]  \|_{0} \to 0 $, and finally by Lemma \ref{L1_stability_fp},
\begin{align*}
        \sup_{t \in  [ 0,T ]} \| m_{\rho,h}^{\epsilon} [ \mu_{n} ] ( t, \cdot ) - m_{\rho,h}^{\epsilon} [ \mu ] ( t,\cdot )   \|_{L^{1} ( \R^d )}
        \leq \frac{cKT}{\rho} e^{- h \lambda_{r}}  \| Du_{\rho,h}^{\epsilon} [ \mu_{n} ] 
        - Du_{\rho,h}^{\epsilon} [ \mu ]  \|_{0} \to 0.
\end{align*}
Hence $\mathcal{S}$ is continuous.
%:
% \begin{align*}
%     \sup_{t \in [ 0,T ]} d_{0} (  \mathcal{S} ( \mu_{n} ), \mathcal{S} ( \mu )) 
% = \sup_{t \in [ 0,T ]} d_{0} (  m_{\rho,h}^{\epsilon} [ \mu_{n} ], m_{h,\rho}^{\epsilon} [ \mu ]) \to 0.
% \end{align*}

%By Schauder fixed point theorem there now
%exists a fixed point  $m_{\rho,h}^{\epsilon} = S ( m_{\rho,h}^{\epsilon} )$.

\section{Proof of Lemma \ref{thm:HJB-convergence-smoothsol} (ii) and (iii)}\label{app:gradlim}
Fix $(t,x)\in [0,T]\times \rd$ and consider a sequence $(t_k,x_k) \to (t,x)$. For any $y\in \rd$, a Taylor expansion shows that
\begin{align}\label{proof-esti1}
\begin{split}
    & u_{\rho_{n}, h_{n}}^{\epsilon_{n}} [ \mu_{n} ](t_k, x_k+y) - u_{\rho_{n}, h_{n}}^{\epsilon_{n}} [ \mu_{n} ](t_k,x_k) - Du_{\rho_{n}, h_{n}}^{\epsilon_{n}} [ \mu_{n} ](t_k, x_k)\cdot y \\
    & = \int_0^1 \big(Du_{\rho_{n}, h_{n}}^{\epsilon_{n}} [ \mu_{n} ](t_k, x_k+sy) - Du_{\rho_{n}, h_{n}}^{\epsilon_{n}} [ \mu_{n} ] (t_k, x_k) \big)\cdot y \,  ds := \int_0^1 I(s)\cdot y \, ds.
\end{split}
\end{align} 
Using first Lemma \ref{semiconcavity_eps} (a) and then part two of Lemma \ref{semiconcavity_eps}  (b), we find that
 \begin{align*} 
 \int_0^{\frac{\rho_n}{\epsilon_n|y|}} I(s) \cdot y \, ds & \leq   2 \|Du_{\rho_{n}, h_{n}}^{\epsilon_{n}} [ \mu_{n} ]\|_0 \frac{\rho_n}{\epsilon_n}\leq 2 ( ( L_{L} + L_{F} )T + L_{G} ) \frac{\rho_n}{\epsilon_n}\\
 \int^1_{\frac{\rho_n}{\epsilon_n|y|}} I(s) \cdot y \, ds   
 &\leq  c_1 \int^1_{\frac{\rho_n}{\epsilon_n|y|}}  \frac{1}{s} \Big(|sy|^2 + \frac{\rho_n^2}{\epsilon_n^2}\Big) ds = c_1  |y|^2 \int^1_{\frac{\rho_n}{\epsilon_n|y|}} \Big( s \,  +   \frac1{s}\frac{\rho_n^2}{|y|^2\epsilon_n^2}  \Big) \, ds \\
 &\leq c_1  |y|^2 \int^1_{\frac{\rho_n}{\epsilon_n|y|}} \Big( s \,  +   \frac1{s}s^2  \Big) \, ds \leq c_1 |y|^2. %\int_0^1 2s \, ds
 \end{align*}
By Lemma \ref{semiconcavity_eps} (a), the sequence $Du_{\rho_{n}, h_{n}}^{\epsilon_{n}} [ \mu_{n} ](t_k, x_k)$ is precompact. Now take any convergent subsequence as $n,k\to \infty$ and $\frac{\rho_n}{\epsilon_n}=o(1)$. If $p$ is the limit, then by passing to the limit in \eqref{proof-esti1} along this subsequence we have
 \begin{align*}
     u[\mu](x+y) - u[\mu](x) - p \cdot y \leq c_1 |y|^2 \quad \text{for every} \quad y \in \rd,
 \end{align*}
 and $p \in D^+u[\mu](t,x)$, the superdifferential of $u[\mu](t,x)$. At points $(x,t)$ where $u[\mu]$ is differentiable, $D^+u[\mu](t,x)=\{Du[\mu](t,x)\}$ and $p= Du[\mu](t,x)$, and then since the subsequence was arbitrary in the above argument and all limit points $p$ coincide,
 \begin{align}\label{limits}
 \begin{split}
     &\limsup_{(t_k, x_k)\to (t,x), n\to \infty } Du_{\rho_{n}, h_{n}}^{\epsilon_{n}} [ \mu_{n} ](t_k, x_k)\\
     &\qquad= \liminf_{(t_k, x_k)\to (t,x), n\to \infty } Du_{\rho_{n}, h_{n}}^{\epsilon_{n}} [ \mu_{n} ](t_k, x_k)\\
     &\qquad = Du(t,x). 
     \end{split}
 \end{align}
 We conclude that $Du_{\rho_{n}, h_{n}}^{\epsilon_{n}} [ \mu_{n} ] \rightarrow Du [ \mu ]$ at $(t,x)$.
Part (ii) now follows since $u[\mu]$ is Lipschitz in space by Proposition \ref{prop:viscosity_sol_HJB} (c) and then $x$-differentiable for a.e. $x$ and every $t$.
 
To prove part (iii), we note that $u$ is $C^1$ by \ref{U1}, so now \eqref{limits} holds for every $(t,x)$. Then in view of the uniform Lipschitz estimate from Lemma \ref{semiconcavity_eps} (a), local uniform convergence follows from \cite[Chapter V, Lemma 1.9]{BCD:book}. The proof is complete.

\section{Proof of Lemma \ref{lem:bound_char}}\label{app:pf}
We first show strong separation between any two characteristics $\Phi^{\epsilon,\pm}$: By Lemma \ref{lem:dimone-semicon},
\begin{align*}
&\big|\Phi^{\epsilon,\pm}_{j,k} - \Phi^{\epsilon,\pm}_{i,k}\big|^2  = \, \Big|x_j-x_i  \pm \sqrt{h}\sigma_r \mp \sqrt{h}\sigma_r - h\Big(D_p H(x_j,Du^{\epsilon}_{\rho,h}(t_k,x_j)) + B^{\sigma}_r \\
 & \hspace*{5cm} - D_p H(x_i,Du^{\epsilon}_{\rho,h}(t_k,x_i)) - B^{\sigma}_r\Big) \Big|^2\\
 &\geq  \, |x_j-x_i|^2 - 2 h \Big(D_pH\big(x_j,Du^{\epsilon}_{\rho,h}(t_k,x_j)\big)- D_p H \big(x_i, Du^{\epsilon}_{\rho,h}(t_k,x_i)\big)\Big) (x_j-x_i) \\
 & \geq  \, (1-c_0 h) |x_j-x_i|^2.
\end{align*}  
Hence, we have 
\begin{align}\label{dist_char}
\min\Big\{\big|\Phi^{\epsilon,+}_{j,k} - \Phi^{\epsilon,+}_{i,k}\big|, \big|\Phi^{\epsilon,-}_{j,k} - \Phi^{\epsilon,-}_{i,k}\big|\Big\} \geq \sqrt{1-c_0h} |j-i|\rho> \rho \sqrt{1-c_0h} .
\end{align}
The result now holds following the proof of \cite[Lemma 3.8]{carliniSilva2014semi1st}. We give the proof for completeness.

Since the diameter of the support of a (hat) basis functions $\beta_i$ is $2\rho$, by \eqref{dist_char} there can be at most 3 characteristics inside the $\textup{supp}(\beta_i)$ %for each $i \in Z$ 
for small enough $h$. The result is trivial if there is only one in characteristic $\textup{supp}(\beta_i)$. 
When  $\textup{supp}(\beta_i)$ contains 2 characteristics, say $\Phi^{\epsilon,+}_{j_1,k}$ and $\Phi^{\epsilon,+}_{j_2,k}$,  we see by \eqref{dist_char} (check the different orderings of $x_k$, $\Phi^{\epsilon,+}_{j_1,k}$, $\Phi^{\epsilon,+}_{j_2,k}$) that
\begin{align*}
\beta_i(\Phi^{\epsilon,+}_{j_1,k})+ \beta_i(\Phi^{\epsilon,+}_{j_2,k}) = &\,  1-\frac{\Big|x_i-\Phi^{\epsilon,+}_{j_1,k}\Big|}{\rho} + 1 -\frac{\Big|x_i-\Phi^{\epsilon,+}_{j_2,k}\Big|}{\rho} \\
 \leq & \, 2 - \frac{\Big|\Phi^{\epsilon,+}_{j_1,k}-\Phi^{\epsilon,+}_{j_2,k}\Big|}{\rho} \leq 2 - \sqrt{1-c_0h} \leq \, 1+ K_0 h. 
\end{align*}

 Finally, assume $support(\beta_i)$ contains 3 characteristics $\Phi^{\epsilon,+}_{j_1,k}, \Phi^{\epsilon,+}_{j_2,k}$ and $\Phi^{\epsilon,+}_{j_3,k}$. By \eqref{dist_char} that all three characteristics can not be on one side (left or right) of $x_i$. Without loss of generality we assume $\Phi^{\epsilon,+}_{j_1,k}<x_i< \Phi^{\epsilon,+}_{j_2,k}<\Phi^{\epsilon,+}_{j_3,k}$, and find 
\begin{align*}
& \beta_i(\Phi^{\epsilon,+}_{j_1,k}) + \beta_i(\Phi^{\epsilon,+}_{j_2,k})+ \beta_i(\Phi^{\epsilon,+}_{j_3,k}) =  1 -\frac{x_i-\Phi^{\epsilon,+}_{j_1,k}}{\rho}+  1 -\frac{\Phi^{\epsilon,+}_{j_2,k}-x_i}{\rho} +  1 -\frac{\Phi^{\epsilon,+}_{j_3,k}-x_i}{\rho} \\
& \leq  \, 3 - \frac{\Phi^{\epsilon,+}_{j_2,k}-\Phi^{\epsilon,+}_{j_1,k}}{\rho} -\frac{\Phi^{\epsilon,+}_{j_3,k}-\Phi^{\epsilon,+}_{j_2,k}}{\rho} \\
& \leq \, 3 - 2 \sqrt{1-c_0h} \leq 1 + 2(1- \sqrt{1-c_0h}) \leq 1+ K_0h.
\end{align*} 
 Combining all three cases we get 
 $$\sum_{j\in \Z} \beta_i (\Phi^{\epsilon,+}_{j,k}) \leq 1 + K_0 h \quad \mbox{for any} \, i \in \Z.$$ 
The estimate of $\sum_{j\in \Z} \beta_i (\Phi^{\epsilon,-}_{j,k})$  is similar. This completes the proof.

\section*{Acknowledgements}
The authors are supported by the Toppforsk (research excellence) project Waves and Nonlinear Phenomena (WaNP), grant no. 250070 from the Research Council of Norway. IC is partially supported by the Croatian Science Foundation under the project 4197. The authors would like to thank Elisabetta Carlini for sharing the code of the numerical methods introduced in \cite{carliniSilva2014semi1st}.

\bibliographystyle{plain}

\end{document}